\documentclass[10pt]{amsart}
      \usepackage[mathscr]{eucal}
      \usepackage{amsmath,amsfonts}
      \parskip\smallskipamount
          \newtheorem{theorem}{Theorem}[section]
      \newtheorem{definition}[theorem]{Definition}
      \newtheorem{proposition}[theorem]{Proposition}
      \newtheorem{corollary}[theorem]{Corollary}
      \newtheorem{lemma}[theorem]{Lemma}

      \makeatletter
      \@addtoreset{equation}{section}
      \makeatother

      \newcommand{\BB}{{\mathbb B}}
      \newcommand{\CC}{{\mathbb C}}
      \newcommand{\NN}{{\mathbb N}}
      
      \newcommand{\ZZ}{{\mathbb Z}}
      \newcommand{\DD}{{\mathbb D}}
      \newcommand{\RR}{{\mathbb R}}
      \newcommand{\FF}{{\mathbb F}}
      \newcommand{\TT}{{\mathbb T}}

      \newcommand{\cA}{{\mathcal A}}
      
      \newcommand{\cC}{{\mathcal C}}
      \newcommand{\cD}{{\mathcal D}}
      \newcommand{\cE}{{\mathcal E}}
      \newcommand{\cF}{{\mathcal F}}
      \newcommand{\cG}{{\mathcal G}}
      \newcommand{\cH}{{\mathcal H}}
      
      \newcommand{\cK}{{\mathcal K}}
       \newcommand{\cJ}{{\mathcal J}}
      
      \newcommand{\cM}{{\mathcal M}}

      \newcommand{\cP}{{\mathcal P}}

      \newcommand{\cT}{{\mathcal T}}

      \newcommand{\cX}{{\mathcal X}}

      \newdimen\expt
      \def\boxit#1{\setbox0\hbox{$\displaystyle{#1}$}
            \hbox{\lower.4\expt
       \hbox{\lower3\expt\hbox{\lower\dp0
            \hbox{\vbox{\hrule height.4\expt
       \hbox{\vrule width.4\expt\hskip3\expt
            \vbox{\vskip3\expt\box0\vskip2\expt}%
       \hskip3\expt\vrule width.4\expt}\hrule height.4\expt}}}}}}
      
\begin{document}
       \pagestyle{myheadings}
      \markboth{ Gelu Popescu}{Multi-Toeplitz operators on noncommutative hyperballs}

      \title [Brown-Halmos characterization of multi-Toeplitz operators]
      { Brown-Halmos characterization of multi-Toeplitz operators associated with noncommutative poly-hyperballs
      }
        \author{Gelu Popescu}
\date{January 21, 2019}
     \thanks{Research supported in part by  NSF grant DMS 1500922}
      \address{Department of Mathematics, The University of Texas
      at San Antonio \\ San Antonio, TX 78249, USA}
      \email{\tt gelu.popescu@utsa.edu}

\begin{abstract}
The {\it noncommutative $m$-hyperball}, $m\in \NN$,   is defined by
$$
\cD_n^m(\cH):=\left\{ X:=(X_1,\ldots, X_n)\in B(\cH)^n: \ (id-\Phi_{X})^k(I)\geq 0 \  \text{\rm  for } 1\leq k\leq m\right\},
$$
where $\Phi_{X}: B(\cH)\to B(\cH)$ is the completely positive map given by $\Phi_{X}(Y):=\sum_{i=1}^n X_i Y X_i^*$ for  $Y\in B(\cH)$.
   Its right universal model is  an $n$-tuple $\Lambda=(\Lambda_1,\ldots, \Lambda_n)$ of  {\it weighted right creation  operators} acting on the full Fock space  $F^2(H_n)$ with $n$ generators. We prove that an operator $T\in B(F^2(H_n))$ is a multi-Toeplitz operator with free pluriharmonic  symbol on $\cD_n^m(\cH)$  if and only if it satisfies the Brown-Halmos type equation
   $$
   \Lambda^{\prime*} T\Lambda^\prime=\bigoplus_{i=1}^n \left(\sum_{j=0}^{m-1} (-1)^j   \left(\begin{matrix} m\\j+1
\end{matrix}\right) \sum_{\alpha\in \FF_n^+, |\alpha|=j} \Lambda_\alpha T \Lambda_\alpha^*\right),
$$
 where $\Lambda'$ is the Cauchy dual of $\Lambda$ and $\FF_n^+$ is the free unital semigroup with $n$ generators. This is a noncommutative multivariable analogue of Louhichi and Olofsson  characterization of Toeplitz operators with harmonic symbols on the weighted Bergman space $A_m(\DD)$,  as well as Eschmeier and Langend\" orfer extension to the unit ball of $\CC^n$.

 All our results are proved in the more general setting of  {\it noncommutative poly-hyperballs} ${\bf D_n^m}(\cH)$,
  ${\bf n,m}\in \NN^k$,
and are used to characterize the bounded free $k$-pluriharmonic functions with operator coefficients on poly-hyperballs  and to solve the associated  Dirichlet
extension problem. In particular, the results hold for  the  reproducing kernel Hilbert space with kernel
 $$
 \kappa_{\bf m}(z,w):=\prod_{i=1}^k \frac{1}{(1-\bar z_i w_i)^{m_i}},\qquad z,w\in \DD^k,
  $$
  where $m_i\geq 1$. This includes the  Hardy space, the Bergman space, and the weighted Bergman space  over the polydisk.
\end{abstract}

      \maketitle

\bigskip

\section*{Introduction}

Let $H^2(\DD)$ be the Hardy space of all analytic functions on the open unit disc
$\DD:=\{z\in \CC: \ |z|<1\}$ with square-sumable coefficients. An operator $T\in B(H^2(\DD))$ is called a Toeplitz  operator if
$$Tf=P_+(\varphi f),\qquad f\in H^2(\TT),
$$
 for some $\varphi \in L^\infty(\TT)$, where $P_+$ is the orthogonal projection of the Lebesgue space $L^2(\TT)$ onto the Hardy space $H^2(\TT)$, which is identified with $H^2(\DD)$.
 Brown and Halmos \cite{BH} proved that a necessary and sufficient condition that an operator on the Hardy space $H^2(\DD)$ be a Toeplitz operator is that
  $$
  S^*TS=T,
  $$
  where $S$ is the unilateral shift on $H^2(\DD)$.
 The study of Toeplitz operators  originates  with O.Toeplitz \cite{T}, and was extended to Hilbert spaces of holomorphic functions  on the unit disc (see \cite{HKZ}) such as the Bergman space and weighted Bergman space,
 and also to higher dimensional setting  involving holomorphic functions in several complex variables on various classes of domains in $\CC^n$ (see Upmeier's book \cite{U}). We refer the reader to \cite{BS}, \cite{Dou}, \cite{RR}, and \cite{HKZ}  for a   comprehensive account on Toeplitz operators.

In \cite{LO}, Louhichi and Olofsson obtained a Brown-Halmos type characterization of Toeplitz operators with harmonic symbols on the weighted Bergman space $A_m(\DD)$,
 the Hilbert space of all analytic functions on the unit disc $\DD$ with
$$
\|f\|^2:=\frac{m-1}{\pi}\int_\DD|f(z)|^2 (1-|z|^2)^{m-2}dz<\infty.
$$
    They proved that an operator $T\in B(A_m(\DD))$ is a Toeplitz operator   with bounded harmonic symbol on $\DD$ if and only if $T$ satisfies the identity
 $$
 M^{\prime  *}_z TM_z'=\sum_{k=0}^{m-1}(-1)^k \left(\begin{matrix} m\\k+1
\end{matrix}\right) M_z^kTM_z^{*k},
$$
 where $M_z':=M_z(M_z^* M_z)^{-1}$ is the Cauchy dual of the multiplication operator $M_z$ on $A_m(\DD)$.
 Their result was recently extended by Eschmeier and Langend\" orfer  \cite{EL} to the analytic functional  Hilbert space $H_m(\BB)$ on the unit ball $\BB\subset \CC^n$ given by the reproducing kernel $\kappa_m(z,w):=\left(1-\left<z,w\right>\right)^{-m}$ for  $z,w\in \BB$, where  $m\geq 1$.

  A study of  unweighted multi-Toeplitz operators on the full Fock  space  $F^2(H_n)$ with $n$ generators   was initiated in \cite{Po-multi}, \cite{Po-analytic} and has had an important impact in multivariable operator theory  and the structure of free semigroups algebras (see \cite{DP2}, \cite{DKP}, \cite{DLP}, \cite{Po-entropy}, \cite{Po-pluriharmonic}, \cite{Ken1}, \cite{Ken2}).

Recently \cite{Po-Toeplitz}, we initiated the study of weighted  multi-Toeplitz operators associated with noncommutative regular domains    ${\cD}_f^m(\cH)$  generated by an arbitrary positive regular free holomorphic functions  $f$ in a neighborhood of the origin.
  This was  accompanied by the study of their symbols which are  free pluriharmonic functions on the radial part of
  ${\cD}_f^m(\cH)$.

  The goal of the present paper is to  provide a Brown-Halmos  type characterizations of  the weighted multi-Toeplitz operators associated with noncommutative poly-hyperballs and to use the results to  characterize the bounded free $k$-pluriharmonic functions with operator coefficients on poly-hyperballs  and to solve the associated  Dirichlet
extension problem.

In Section 1, we recall from \cite{Po-domains-models},  \cite{Po-domains}, \cite{Po-Berezin2}, and  \cite{Po-Berezin1} some basic facts concerning the noncommutative poly-hyperballs, their universal models, and the associated noncommutative Berezin transforms. These preliminaries are needed throughout the paper.

In Section 2, we introduce the multivariable Brown-Halmos type equations
$$
{\bf \Lambda}_i^{\prime *} X{\bf \Lambda}_i'=
{\bf diag}_{n_i}\left( \sum_{j=0}^{m_i-1} (-1)^j \left(\begin{matrix}  m_i\\j+1
\end{matrix}\right)\sum_{\beta\in \FF_{n_i}^+, |\beta|=j} {\bf \Lambda}_{i,\beta}X {\bf \Lambda}_{i,\beta}^*\right), \qquad i\in \{1,\ldots, k\},
$$
over the algebra of all  bounded linear operator on the tensor product  $\otimes_{i=1}^k F^2(H_{n_i})$ of full Fock spaces and ${\bf \Lambda}=({\bf \Lambda}_1,\ldots, {\bf \Lambda}_k)$ is the right universal model for the poly-hyperball
${\bf D_n^m}$.  Any solution of this equation is said to have the Brown-Halmos property (see Definition \ref{BH} for details). The main result of this section (see Theorem \ref{main}) provides a complete description of all solutions of the equations above. We prove that $T\in B(\otimes_{i=1}^k F^2(H_{n_i}))$ satisfies  the Brown-Halmos property
if and only if there is a bounded free $k$-pluriharmonic  function $F$ on the radial  part of the poly-hyperball $
{\bf D_{n}^m}$ such that
 $$
 T=\text{\rm SOT-}\lim_{r\to 1} F(r{\bf W}),
$$
where ${\bf W}=({\bf W}_1,\ldots, {\bf W}_k)$ is the left universal model of the poly-hyperball.

In Section 3, we introduce the weighted multi-Toeplitz operators which are associated with the poly-hyperball
${\bf D_{n}^m}$ and are acting on the tensor product $\otimes_{i=1}^k F^2(H_{n_i})$. The main result of this section
(see Theorem \ref{main2})  shows that  the weighted multi-Toeplitz operators are precisely those satisfying the Brown-Halmos equations. We also prove that each  weighted multi-Toeplitz operator $T$ has a unique  formal Fourier representation
$$
\varphi({\bf W}, {\bf W}^*):=\sum_{(\boldsymbol \alpha, \boldsymbol \beta)\in \cJ}a_{   (\boldsymbol \alpha, \boldsymbol \beta)} {\bf W}_{\boldsymbol\alpha}{\bf W}_{\boldsymbol\beta}^*, \qquad  a_{   (\boldsymbol \alpha, \boldsymbol \beta)}\in \CC,
$$
which can be viewed as a noncommutative symbol and can be used to recover the operator $T$. Conversely, given a formal series $\varphi({\bf W}, {\bf W}^*)$ of the form above, we provide necessary and sufficient conditions on $\varphi({\bf W}, {\bf W}^*)$ to be  the formal Fourier representation  of a weighted multi-Toeplitz operator (see Theorem \ref{Fourier}).

In Section 4,  we prove that the   bounded free $k$-pluriharmonic  functions on the radial  part of the poly-hyperball
${\bf D_{n}^m}$ are precisely those  that are noncommutative Berezin transforms of the weighted multi-Toeplitz operators. In this setting, we solve the Dirichlet extension problem (see Theorem \ref{Dirichlet}).

We should mention that our results are presented in the more general setting of weighted multi-Toeplitz matrices with operator-valued entries and  free $k$-pluriharmonic functions  with operator-valued coefficients.

\bigskip

\section{Noncommutative poly-hyperballs and  universal models }

This section of preliminaries contains basic  facts concerning the noncommutative poly-hyperballs, their universal models, and the associated noncommutative Berezin transforms.

Given two $k$-tuples ${\bf m}:=(m_1,\ldots, m_k)$ and ${\bf n}:=(n_1,\ldots, n_k)$  with $m_i,n_i\in  \NN:=\{1,2,\ldots\}$,  we associate with each   ${\bf X}:=(X_1,\ldots, X_k)\in B(\cH)^{n_1}\times\cdots \times B(\cH)^{n_k}$ with $X_i:=(X_{i,1},\ldots, X_{i, n_i})$  the {\it defect mapping} ${\bf \Delta_{X}^m}:B(\cH)\to  B(\cH)$ defined by
$$
{\bf \Delta_{X}^m}:=\left(id -\Phi_{X_1}\right)^{m_1}\circ \cdots \circ\left(id -\Phi_{X_k}\right)^{m_k},
$$
where $\Phi_{X_i}: B(\cH) \to B(\cH)$ is  the completely positive map given by $\Phi_{X_i}(X):=\sum_{j=1}^{n_i} X_{i,j}Y X_{i,j}^*$.
We denote by $B(\cH)^{n_1}\times_c\cdots \times_c B(\cH)^{n_k}$
   the set of all tuples  ${\bf X}=({ X}_1,\ldots, { X}_k)\in B(\cH)^{n_1}\times\cdots \times B(\cH)^{n_k}$      with the property that, for every $p,q\in \{1,\ldots, k\}$, $p\neq q$, the entries of ${ X}_p$ are commuting with the entries of ${ X}_q$. Note that the operators $X_{i,1},\ldots, X_{i,n_i}$ are not necessarily commuting.

The {\it noncommutative poly-hyperball}  ${\bf D_n^m}$ is defined by its representations on  Hilbert spaces $\cH$, i.e.
$$
{\bf D_n^m}(\cH):=\left\{ {\bf X}=(X_1,\ldots, X_k)\in B(\cH)^{n_1}\times_c\cdots \times_c B(\cH)^{n_k}: \ {\bf \Delta_{X}^p}(I)\geq 0 \ \text{ for }\ {\bf 0}\leq {\bf p}\leq {\bf m}\right\}.
$$
For each $i\in \{1,\ldots, k\}$,
let $\FF_{n_i}^+$ be the unital free semigroup on $n_i$ generators
$g_{1}^i,\ldots, g_{n_i}^i$ and the identity $g_{0}^i$.  The length of $\alpha\in
\FF_{n_i}^+$ is defined by $|\alpha|:=0$ if $\alpha=g_0^i$  and
$|\alpha|:=p$ if
 $\alpha=g_{j_1}^i\cdots g_{j_p}^i$, where $j_1,\ldots, j_p\in \{1,\ldots, n_i\}$.
  Let $H_{n_i}$ be
an $n_i$-dimensional complex  Hilbert space with orthonormal basis
$e_1^i,\dots,e_{n_i}^i$.
  We consider the full Fock space  of $H_{n_i}$ defined by
$$F^2(H_{n_i}):=\bigoplus_{p\geq 0} H_{n_i}^{\otimes p},$$
where $H_{n_i}^{\otimes 0}:=\CC 1$ and $H_{n_i}^{\otimes p}$ is the
(Hilbert) tensor product of $p$ copies of $H_{n_i}$. Set $e_\alpha^i :=
e^i_{j_1}\otimes \cdots \otimes e^i_{j_p}$ if
$\alpha=g^i_{j_1}\cdots g^i_{j_p}\in \FF_{n_i}^+$
 and $e^i_{g^i_0}:= 1\in \CC$.
It is  clear that $\{e^i_\alpha:\alpha\in\FF_{n_i}^+\}$ is an orthonormal
basis of $F^2(H_{n_i})$.

For each $i\in \{1,\ldots, k\}$, let
   $b_{i,g_0^i}^{(m_i)} :=1$, and
\begin{equation}
\label{b-al2}
 b_{i,\beta_i}^{(m_i)}: =
\left(\begin{matrix} |\beta_i|+m_i-1\\m_i-1
\end{matrix}\right)  \qquad
\text{ if } \ \beta_i\in \FF_{n_i}^+,  |\beta_i|\geq 1.
\end{equation}
The diagonal operators  $D_{i,j}:F^2(H_{n_i})\to F^2(H_{n_i})$   are defined  by setting
$$
D_{i,j}e^i_\alpha:=\sqrt{\frac{b_{i,\alpha}^{(m_i)}}{b_{i, g_j\alpha }^{(m_i)}}}
e^i_\alpha=\sqrt{\frac{|\alpha|+1}{|\alpha|+m_i}} e^i_\alpha\qquad
 \alpha\in \FF_{n_i}^+,
$$
where
$\{e^i_\alpha\}_{\alpha\in \FF_{n_i}^+}$ is the orthonormal basis of the full Fock space $F^2(H_{n_i})$.
    As in \cite{Po-domains}, we associate with the noncommutative $m_i$-hyperball
$$
\cD_{n_i}^{m_i}(\cH):=\{ X_i\in B(\cH)^{n_i}: \  (id-\Phi_{X_i})^{m_i}(I)\geq 0 \text{ for }  0\leq p\leq m_i\}
$$
the {\it weighted left creation  operators} $W_{i,j}:F^2(H_{n_i})\to
F^2(H_{n_i})$ defined
by $W_{i,j}:=S_{i,j}D_{ij}$, where
 $S_{i,1},\ldots, S_{i,n_i}$ are the left creation operators on the full
 Fock space $F^2(H_{n_i})$, i.e.
      $$
       S_{i,j} f:=e_j^i \otimes f, \qquad  f\in F^2(H_{n_i}),\  j\in \{1,\ldots,n_i\}.
      $$
      If $\alpha=g_{j_1}^i\cdots g_{j_p}^i\in\FF_{n_i}^+$, we set $W_{i,\alpha}:=W_{i,j_1}\cdots W_{i,j_p}$, and  $W_{i,g_0^i}:=I$.
       A simple calculation reveals that
\begin{equation*}
W_{i,\beta} e^i_\gamma= \frac {\sqrt{b_{i,\gamma}^{(m_i)}}}{\sqrt{b_{i,\beta
\gamma}^{(m_i)}}} e^i_{\beta \gamma} \quad \text{ and }\quad W_{i,\beta}^*
e^i_\alpha =\begin{cases} \frac
{\sqrt{b_{i,\gamma}^{(m_i)}}}{\sqrt{b_{i,\alpha}^{(m_i)}}}e^i_\gamma& \text{ if
}
\alpha=\beta\gamma \\
0& \text{ otherwise }
\end{cases}
\end{equation*}
 for every $\alpha, \beta \in \FF_{n_i}^+$.
       For each $i\in \{1,\ldots, k\}$ and $j\in \{1,\ldots, n_i\}$, we
define the operator ${\bf W}_{i,j}$ acting on the tensor Hilbert space
$F^2(H_{n_1})\otimes\cdots\otimes F^2(H_{n_k})$ by setting
$${\bf W}_{i,j}:=\underbrace{I\otimes\cdots\otimes I}_{\text{${i-1}$
times}}\otimes W_{i,j}\otimes \underbrace{I\otimes\cdots\otimes
I}_{\text{${k-i}$ times}}.
$$
 According to \cite{Po-Berezin1}, if  ${\bf W}_i:=({\bf W}_{i,1},\ldots,{\bf W}_{i,n_i})$, then
 $$
 (id-\Phi_{{\bf W}_1})^{m_1}\circ\cdots \circ (id-\Phi_{{\bf W}_k})^{m_k}(I)={\bf P}_\CC,
 $$
  where ${\bf P}_\CC$ is the
 orthogonal projection from $\otimes_{i=1}^k F^2(H_{n_i})$ onto $\CC 1\subset \otimes_{i=1}^k F^2(H_{n_i})$, where $\CC 1$ is identified with $\CC 1\otimes\cdots \otimes \CC 1$.
     Moreover,  ${\bf W}:=({\bf W}_1,\ldots, {\bf W}_k)$ is  a pure $k$-tuple, i.e. $\Phi_{W_i}^p(I)\to 0$ strongly as $p\to \infty$,  in the
noncommutative poly-hyperball $ {\bf D_n^m}(\otimes_{i=1}^kF^2(H_{n_i}))$.

The {\it noncommutative Berezin kernel} associated with any element
   ${\bf X}=\{X_{i,j}\}$ in the noncommutative poly-hyperball ${\bf D_n^m}(\cH)$ is the operator
   $${\bf K_{X}}: \cH \to F^2(H_{n_1})\otimes \cdots \otimes  F^2(H_{n_k}) \otimes  \overline{{\bf \Delta_{X}^m}(I) (\cH)}$$
   defined by
   $$
   {\bf K_{X}}h:=\sum_{\beta_i\in \FF_{n_i}^+, i=1,\ldots,k}
   \sqrt{b_{1,\beta_1}^{(m_1)}}\cdots \sqrt{b_{k,\beta_k}^{(m_k)}}
   e^1_{\beta_1}\otimes \cdots \otimes  e^k_{\beta_k}\otimes {\bf \Delta_{X}^m}(I)^{1/2} X_{1,\beta_1}^*\cdots X_{k,\beta_k}^*h,
   $$
where   the defect operator is given  by
$$
{\bf \Delta_{X}^m}(I)  :=(id-\Phi_{X_1})^{m_1}\circ\cdots \circ (id-\Phi_{X_k})^{m_k}(I).
$$
The noncommutative Berezin kernel associated with a $k$-tuple
${\bf X}=({ X}_1,\ldots, { X}_k)$ in the noncommutative poly-hyperball ${\bf D_n^m}(\cH)$ has the following properties.
\begin{enumerate}
\item[(i)] ${\bf K_{X}}$ is a contraction  and
$$
{\bf K_{X}^*}{\bf K_{X}}=
\lim_{q_k\to\infty}\ldots \lim_{q_1\to\infty}  (id-\Phi_{X_k}^{q_k})\circ\cdots \circ (id-\Phi_{X_1}^{q_1})(I).
$$
where the limits are in the weak  operator topology.
\item[(ii)]  If ${\bf X}$ is {\it pure},   i.e. $\Phi_{X_i}^p(I)\to 0$ strongly as $p\to \infty$, then
$${\bf K_{X}^*}{\bf K_{X}}=I_\cH. $$
\item[(iii)]  For everyfor every $i\in \{1,\ldots, k\}$ and $j\in \{1,\ldots, n_i\}$,  $${\bf K_{X}} { X}^*_{i,j}= ({\bf W}_{i,j}^*\otimes I)  {\bf K_{X}}.
    $$
\end{enumerate}
The $k$-tuple   ${\bf W}:=({\bf W}_1,\ldots, {\bf W}_k)$
 plays the role of the {\it left universal model for the  noncommutative poly-hyperball}
${\bf D_n^m}$.

For each $i\in \{1,\ldots, k\}$ and $j\in \{1,\ldots, n_i\}$, we define the {\it weighted right creation operators}
$\Lambda_{i,j}:F^2(H_{n_i})\to F^2(H_{n_i})$ by setting $\Lambda_{i,j}:= R_{i ,j}D_{i,j}$,
  where $R_{i,1},\ldots, R_{i,n_i}$ are
 the right creation operators on the full Fock space $F^2(H_{n_i})$.
 In this case, we have
\begin{equation*}
\Lambda_{i,\beta} e^i_\gamma= \frac {\sqrt{b_{i,\gamma}^{(m_i)}}}{\sqrt{b_{i,
\gamma \tilde\beta}^{(m_i)}}} e^i_{ \gamma \tilde \beta} \quad \text{
and }\quad \Lambda_{i,\beta}^* e^i_\alpha =\begin{cases} \frac
{\sqrt{b_{i,\gamma}^{(m_i)}}}{\sqrt{b_{i,\alpha}^{(m_i)}}}e_\gamma& \text{ if
}
\alpha=\gamma \tilde \beta \\
0& \text{ otherwise }
\end{cases}
\end{equation*}
 for every $\alpha, \beta \in \FF_{n_i}^+$, where $\tilde \beta$ denotes
 the reverse of $\beta=g^i_{j_1}\cdots g^i_{j_p}$, i.e.,
 $\tilde \beta=g_{j_p}^i\cdots g^i_{j_1}$. Note that $\Lambda_{i,j} W_{i,p}=W_{i,p}\Lambda_{i,j}$ fpr any $i\in \{1,\ldots,k\}$ and $j,p\in \{1,\ldots, n_i\}$.
  We introduce  the operator ${\bf \Lambda}_{i,j}$ acting on
$F^2(H_{n_1})\otimes\cdots\otimes F^2(H_{n_k})$ and given by
$${\bf \Lambda}_{i,j}:=\underbrace{I\otimes\cdots\otimes I}_{\text{${i-1}$
times}}\otimes \Lambda_{i,j}\otimes \underbrace{I\otimes\cdots\otimes
I}_{\text{${k-i}$ times}}.
$$
We set   ${\bf \Lambda}_i:=({\bf \Lambda}_{i,1},\ldots,{\bf \Lambda}_{i,n_i})$  for each $i\in \{1,\ldots, k\}$.
The $k$-tuple   ${\bf \Lambda}:=({\bf \Lambda}_1,\ldots, {\bf \Lambda}_k)$
 plays the role of the {\it right universal model for the  noncommutative poly-hyperball}
${\bf D_n^m}$. When necessary, we also denote by ${\bf \Lambda}_i$ the row operator $[{\bf \Lambda}_{i,1}\cdots {\bf \Lambda}_{i,n_i}]$ acting on the direct sum $(\otimes_{s=1}^k F^2(H_{n_s}))^{(n_i)}$.
More on  noncommutative polydomains, universal models, noncommutative Berezin transforms  and their applications can be found in \cite{Po-poisson}, \cite{Po-domains-models},  \cite{Po-domains}, \cite{Po-Berezin2}, and  \cite{Po-Berezin1}.

\bigskip

\bigskip

\section{A multivariable Brown-Halmos type equation: solutions, free $k$-pluriharmonic functions }

In this section, we introduce the multivariable Brown-Halmos type equations associated with the poly-hyperballs and
provide a complete description of all solutions   in terms of bounded free $k$-pluriharmonic functions.

  For each $i\in \{1,\ldots, k\}$, we
define  the bounded linear operator  ${\Omega}_i:F^2(H_{n_i})\to F^2(H_{n_i})$ by setting
$$
{\Omega_i}\left(\sum_{\alpha\in \FF_{n_i}^+}   a_{(\alpha)}e^i_\alpha\right):=a_{(g_0^i)}+ \sum_{j=1}^\infty\frac{m_i+j-1}{j}\sum_{\alpha\in \FF_{n_i}^+, |\alpha|=j} a_{(\alpha)}\otimes e^i_\alpha,
$$
and the operator ${\bf \Omega}_i$  is  acting on the tensor Hilbert space
$\otimes_{s=1}^k F^2(H_{n_s})$ by setting
$${\bf \Omega}_i:=\underbrace{I\otimes\cdots\otimes I}_{\text{${i-1}$
times}}\otimes \Omega_{i}\otimes \underbrace{I\otimes\cdots\otimes
I}_{\text{${k-i}$ times}}.
$$

If $A\in B(\cH)$ and $n\in \NN$, we use the notation ${\bf diag}_n(A)$ for the direct sum of $n$ copies of $A$, acting on the Hilbert space $\cH^{(n)}$.

\begin{proposition} \label{WWW} For each  $i\in \{1,\ldots, k\}$, the operator ${\bf\Lambda}_i$ satisfies the following properties.
\begin{enumerate}
\item[(i)]
  ${\bf \Lambda}_i^*{\bf \Lambda}_i$ is an invertible operator   acting on $\text{\rm range}\,{\bf \Lambda}_i^*=(\otimes_{s=1}^k F^2(H_{n_s}))^{(n_i)}$ and
$$
({\bf \Lambda}_i^*{\bf \Lambda}_i)^{-1}{\bf \Lambda}_i^*={\bf \Lambda}_i^*{\bf \Omega}_i.
$$
\item[(ii)] The operator
  ${\bf \Lambda}_i({\bf \Lambda}_i^*{\bf \Lambda}_i)^{-1}{\bf \Lambda}_i^*$ is the orthogonal projection of
$\otimes_{s=1}^k F^2(H_{n_s})$ onto
$$\text{\rm range}\,{\bf\Lambda}_i
=\left(\otimes_{s=1}^{i-1} F^2(H_{n_s})\right)\otimes(F^2(H_{n_i})\ominus \CC)\otimes \left(\otimes_{s=i+1}^{k} F^2(H_{n_s})\right).
$$
\item[(iii)]  The following identity holds:
$${\bf \Lambda}_i({\bf \Lambda}_i^*{\bf \Lambda}_i)^{-1}={\bf \Lambda}_i {\bf diag}_{n_i}\left( \sum_{j=0}^{m_i-1} (-1)^j \left(\begin{matrix}  m_i\\j+1
\end{matrix}\right)\sum_{\beta\in \FF_{n_i}^+, |\beta|=j} {\bf \Lambda}_{i,\beta} {\bf \Lambda}_{i,\beta}^*\right).
$$
\end{enumerate}
 \end{proposition}
\begin{proof} To prove part (i), let    $f_1,\ldots, f_{n_i}\in F^2(H_{n_i})$ and note that
$$g:=R_{i,1}D_{i,1}^{-1}f_1+\cdots +R_{i,n_i}D_{i,n_i}^{-1}f_{n_i}\in F^2(H_{n_i}),
$$
where $R_{i,j}$ and $ D_{i,j}$ are defined in Section 1.
 Since $R_{i,j}^*R_{i,s}=\delta_{js}I_{F^2(H_{n_i})}$, we have
$$
\Lambda_{i,j}^*g=D_{i,j}R_{i,j}^*R_{i,j}D_{i,j}^{-1}f_j=f_j
$$
and, consequently, $\Lambda_i^*g=\oplus_{j=1}^{n_i} f_j \in F^2(H_{n_i})^{(n_i)}$.
This proves that $\text{\rm range}\,\Lambda_i^*=F^2(H_{n_i})^{(n_i)}$ and, therefore, $\text{\rm range}\,{\bf \Lambda}_i^*=(\otimes_{s=1}^k F^2(H_{n_s}))^{(n_i)}$.

Now, let $f:=\sum_{\alpha\in \FF_{n_i}^+} a_{(\alpha)} e^i_\alpha\in F^2(H_{n_i})$ and note that
$$
\Lambda_{i,j}^*f=\sum_{\gamma\in \FF_{n_i}^+} \sqrt{\frac{|\gamma|+1}{|\gamma|+m_i}} a_{(\gamma g_j^i)} e_\gamma^i,\qquad j\in \{1,\ldots, n_i\},
$$
and
$$
D_{i,j}^{-2}\Lambda_{i,j}^*f=\sum_{\gamma\in \FF_{n_i}^+} \sqrt{\frac{|\gamma|+m_i}{|\gamma|+1}}a_{(\gamma g_j^i)} e_\gamma^i .
$$
On the other hand, we have
\begin{equation*}
\begin{split}
\Lambda_{i,j}^*{\Omega_{i}} f&=\Lambda_{i,j}^*\left(a_{(g_0^i)}+\sum_{j=1}^{n_i} \sum_{\gamma\in \FF_{n_i}^+}  \frac{|\gamma|+m_i}{|\gamma|+1}a_{(\gamma g_j^i)} e^i_{\gamma g_j}\right)\\
&=\sum_{\gamma\in \FF_{n_i}^+}  \frac{|\gamma|+m_i}{|\gamma|+1}
\sqrt{\frac{|\gamma|+1}{|\gamma|+m_i}} a_{(\gamma g_j)} e^i_\gamma
= \sum_{\gamma\in \FF_{n_i}^+}  \sqrt{\frac{|\gamma|+m_i}{|\gamma|+1}}
  a_{(\gamma g_j^i)} e^i_\gamma.
\end{split}
\end{equation*}
Consequently, $D_{i,j}^{-2}\Lambda_{i,j}^*=\Lambda_{i,j}^*{\Omega_{i}}$ for every $j\in \{1,\ldots, n_i\}$, which shows that
$$
\left(\oplus_{j=1}^{n_i} D_{i,j}^{-2}\right)\Lambda_i^*=\Lambda_i^*\Omega_i.
$$
Since $\Lambda_{i,j} :=R_{i,j}D_{i,j}$, the operators  $\Lambda_{i,1},\ldots, \Lambda_{i,n_i}$   have orthogonal ranges and
$\Lambda_i^*\Lambda_i=\oplus_{j=1}^{n_i} D_{i,j}^2$. Consequently,
$(\Lambda_i^*\Lambda_i)^{-1} \Lambda_i^*=\Lambda_i^* \Omega_i$, which implies item (i).

A straightforward computation reveals that
$$ \Lambda_i \Lambda_i^*=\sum_{j=1}^\infty \frac{1}{m_i+j-1} {\bf P}_{\{\text{\rm span}\,  e^i_\alpha: \ \alpha\in \FF_{n_i}^+, |\alpha|=j\}},
$$
where ${\bf P}_\cM$ is the orthogonal projection onto $\cM$.
Consequently, if $f:=\sum_{\alpha\in \FF_{n_i}^+} a_{(\alpha)} e^i_\alpha\in F^2(H_{n_i})$, part (i) implies
$$
\Lambda_i(\Lambda_i^*\Lambda_i)^{-1} \Lambda_i^*f=\Lambda_i\Lambda_i^* \Omega_i f =\sum_{\alpha\in \FF_{n_i}^+, |\alpha|\geq 1} a_{(\alpha)} e^i_\alpha,
$$
which shows that $\Lambda_i(\Lambda_i^*\Lambda_i)^{-1} \Lambda_i^*$ is the orthogonal projection of $F^2(H_{n_i})$ onto $F^2(H_{n_i})\ominus \CC$. Now, item (ii) follows.

To prove item (iii), we recall   that
$$
 (id-\Phi_{{\bf \Lambda}_i})^{m_i}(I)=\underbrace{I\otimes\cdots\otimes I}_{\text{${i-1}$
times}}\otimes {\bf P}_\CC\otimes \underbrace{I\otimes\cdots\otimes
I}_{\text{${k-i}$ times}},\qquad i\in \{1,\ldots, k\},
 $$
  where ${\bf P}_\CC$ is the
 orthogonal projection of  $  F^2(H_{n_i})$ onto $\CC 1\subset F^2(H_{n_i})$.  Consequently, using  item (ii), we deduce that

 \begin{equation*}
 \begin{split}
 {\bf \Lambda}_i({\bf \Lambda}_i^*{\bf \Lambda}_i)^{-1}{\bf \Lambda}_i^*  &= I_{\otimes_{s=1}^k F^2(H_{n_s})}-(id-\Phi_{{\bf \Lambda}_i})^{m_i}(I)\\
 &=
 \sum_{j=0}^{m_i-1} (-1)^j \left(\begin{matrix}  m_i\\j+1
\end{matrix}\right)\sum_{\beta\in \FF_{n_i}^+, |\beta|=j+1} {\bf \Lambda}_{i,\beta} {\bf \Lambda}_{i,\beta}^*\\
&={\bf \Lambda}_i{\bf diag}_{n_i}\left( \sum_{j=0}^{m_i-1} (-1)^j \left(\begin{matrix}  m_i\\j+1
\end{matrix}\right)\sum_{\beta\in \FF_{n_i}^+, |\beta|=j} {\bf \Lambda}_{i,\beta} {\bf \Lambda}_{i,\beta}^*\right){\bf \Lambda}_i^*.
\end{split}
 \end{equation*}
Since $\text{\rm range}\,{\bf \Lambda}_i^*=(\otimes_{s=1}^k F^2(H_{n_s}))^{(n_i)}$,  item (iii) follows.
The proof is complete.
\end{proof}

The operator ${\bf \Lambda}_i': (\otimes_{s=1}^k F^2(H_{n_s}))^{(n_i)}\to \otimes_{s=1}^k F^2(H_{n_s})$, $i\in \{1,\ldots, k\}$,
defined by ${\bf \Lambda}_i':={\bf \Lambda}_i({\bf \Lambda}_i^*{\bf \Lambda}_i)^{-1}$ is called the {\it Cauchy dual }of ${\bf \Lambda}_i$.
\begin{definition} An operator $T\in B(\otimes_{s=1}^k F^2(H_{n_s}))$ is said to have the Brown-Halmos property if
$$
{\bf \Lambda}_i^{\prime *} T{\bf \Lambda}_i'=
{\bf diag}_{n_i}\left( \sum_{j=0}^{m_i-1} (-1)^j \left(\begin{matrix}  m_i\\j+1
\end{matrix}\right)\sum_{\beta\in \FF_{n_i}^+, |\beta|=j} {\bf \Lambda}_{i,\beta}T {\bf \Lambda}_{i,\beta}^*\right)
$$
for every $i\in \{1,\ldots, k\}$.
\end{definition}
We mention  that in the particular case in which $k=1, m_1=1, n_1\in \NN$, the condition in the definition above becomes $R_j^*TR_s=\delta_{js} T$ for $j,s\in \{1,\ldots, n_1\}$. The class of the operators satisfying these equations coincides with the class of   multi-Toeplitz operators on full Fock spaces which  has been studied in  several papers (see \cite{DP2}, \cite{DKP}, \cite{DLP}, \cite{Po-entropy}, \cite{Po-pluriharmonic}, \cite{Ken1}, \cite{Ken2}).

 Note also that if  $n_1=\cdots =n_k=1$ and $m_1=\cdots =m_k=1$, then the  equations become $M_{z_i}^* TM_{z_i}=T$ for every $i\in \{1,\ldots, k\}$, where $M_{z_i}$ is the multiplication by the coordinate function $z_i$ on  $H^2(\DD^k)$, the Hardy space of the  polydisc.  The class of operators  satisfying this  condition coincides with the class of Toeplitz operators on $H^2(\DD^k)$ (see \cite{MSS}).
 Taking here $k=1$,   we obtain the Brown-Halmos condition $S^* TS=T$, where $S$ is the unilateral shift on $H^2(\DD)$ (see \cite{BH}).

   If $\cK$ is a separable Hilbert space, we say that an operator  $T\in B\left(\cK\otimes \bigotimes_{s=1}^k F^2(H_{n_s})\right)$ satisfies  the Brown-Halmos condition if
\begin{equation} \label{BH}
\widetilde{\bf \Lambda}_i^{\prime *} T\widetilde{\bf \Lambda}_i'=
{\bf diag}_{n_i}\left( \sum_{j=0}^{m_i-1} (-1)^j \left(\begin{matrix}  m_i\\j+1
\end{matrix}\right)\sum_{\beta\in \FF_{n_i}^+, |\beta|=j} \widetilde{\bf \Lambda}_{i,\beta}T \widetilde{\bf \Lambda}_{i,\beta}^*\right),\qquad i\in \{1,\ldots, k\},
\end{equation}
where $\widetilde{\bf \Lambda}_{i,j}:=I_\cK\otimes {\bf \Lambda}_{i,j}$, $\widetilde{\bf \Lambda}_i:=[\widetilde{\bf \Lambda}_{i,1}\cdots \widetilde{\bf \Lambda}_{i,n_i}]$ , and
the Cauchy dual operator
$$\widetilde{\bf \Lambda}_i': \left(\cK\otimes\bigotimes_{s=1}^k F^2(H_{n_s})\right)^{(n_i)}\to \cK\otimes\bigotimes_{s=1}^k F^2(H_{n_s})$$
is defined by $\widetilde{\bf \Lambda}_i':=\widetilde{\bf \Lambda}_i(\widetilde{\bf \Lambda}_i^*\widetilde{\bf \Lambda}_i)^{-1}$.

Let $\Gamma:\TT^k\to B( \otimes_{s=1}^k F^2(H_{n_s}))$ be the strongly continuous unitary representation of the $k$-dimensional torus, defined by
$$
\Gamma(e^{i\theta_1},\ldots, e^{i\theta_k})f
:=\sum_{{\alpha_s\in \FF_{n_s}^+ }\atop{s\in \{1,\ldots,k\}}} e^{i\theta_1|\alpha_1|}\cdots e^{i\theta_k|\alpha_k|} a_{\alpha_1,\ldots, \alpha_k} e^1_{\alpha_1}\otimes \cdots e^k_{\alpha_k}.
$$
for every $f=\sum_{{\alpha_s\in \FF_{n_s}^+ }\atop{s\in \{1,\ldots,k\}}} a_{\alpha_1,\ldots, \alpha_k} e^1_{\alpha_1}\otimes \cdots e^k_{\alpha_k}\in \otimes_{s=1}^k F^2(H_{n_s})$.
We have the orthogonal decomposition
$$
 \otimes_{s=1}^k F^2(H_{n_s})=\bigoplus_{(p_1,\ldots, p_k)\in \ZZ^k} \cE_{p_1,\ldots, p_k},
 $$
where the spectral subspace $\cE_{p_1,\ldots, p_k}$ is the image of the orthogonal projection ${\bf P}_{p_1,\ldots, p_k}\in B(\otimes_{s=1}^k F^2(H_{n_s}))$ defined by
$$
{\bf P}_{p_1,\ldots, p_k}:= \left(\frac{1}{2\pi}\right)^k
\int_0^{2\pi}\cdots \int_0^{2\pi} e^{-ip_1\theta_1}\cdots e^{-ip_k\theta_k}\Gamma(e^{i\theta_1},\ldots, e^{i\theta_k})d\theta_1\ldots d\theta_k,
$$
where the integral is defined as a weak integral  and the integrant is a continuous function in  the strong operator topology.
We remark that if $p_s<0$ for some $s\in \{1,\ldots, k\}$, then ${\bf P}_{p_1,\ldots, p_k}=0$ and, therefore, $\cE_{p_1,\ldots,p_k}=\{0\}$. Note that the spectral subspaces of $\Gamma$ are
$$
\cE_{p_1,\ldots, p_k}:=\left\{ f\in \otimes_{s=1}^k F^2(H_{n_s}): \
\Gamma(e^{i\theta_1},\ldots, e^{i\theta_k})f=e^{i\theta_1 p_1}\cdots e^{i\theta_k p_k}f\right\}
$$
for $(p_1,\ldots, p_k)\in \ZZ^k$.  From now on, we  use the notation
$\Gamma (e^{i {\boldsymbol\theta}}):= \Gamma(e^{i\theta_1},\ldots, e^{i\theta_k})$.

\begin{definition}  \label{mhp} If $T\in B(\cK\otimes \bigotimes_{s=1}^k  F^2(H_{n_s}))$ and $(s_1,\ldots, s_k)\in \ZZ^k$ we define the $(s_1,\ldots, s_k)$-multi-homogeneous part of $T$  to be the operator $T_{s_1,\ldots, s_k}\in B(\cK\otimes \bigotimes_{s=1}^k  F^2(H_{n_s}))$ defined by
$$
T_{s_1,\ldots, s_k}:=
\left(\frac{1}{2\pi}\right)^k
\int_0^{2\pi}\cdots \int_0^{2\pi} e^{-is_1\theta_1}\cdots e^{-is_k\theta_k}\left(I_\cK\otimes \Gamma (e^{i {\boldsymbol\theta}})\right) T
\left(I_\cK\otimes \Gamma (e^{i {\boldsymbol\theta}})\right)^* d\theta_1\ldots d\theta_k.
$$

\end{definition}

It is easy to see that $(T^*)_{s_1,\ldots, s_k}=(T_{-s_1,\ldots, -s_k})^*$ and
$$\Gamma (e^{i {\boldsymbol\theta}})^*|_{\cE_{p_1,\ldots, p_k}}
=e^{-ip_1\theta_1}\cdots e^{-ip_k\theta_k}I_{\cE_{p_1,\ldots, p_k}}
$$
for every $(p_1,\ldots, p_k)\in \ZZ^k$.
Fix $(s_1,\ldots, s_k)\in \ZZ^k$ and note that, for every $f\in \cK\otimes \cE_{p_1,\ldots, p_k}$,
\begin{equation*}
\begin{split}
T_{s_1,\ldots, s_k}f&=\left(\frac{1}{2\pi}\right)^k
\int_0^{2\pi}\cdots \int_0^{2\pi} e^{-i(s_1+p_1)\theta_1}\cdots e^{-i(s_k+p_k)\theta_k}\left(I_\cK\otimes \Gamma (e^{i {\boldsymbol\theta}})\right) Tf d\theta_1\ldots d\theta_k\\
 &=(I_\cK\otimes{\bf P}_{s_1+p_1,\ldots, s_k+p_k})Tf.
\end{split}
\end{equation*}
Consequently,
$$T_{s_1,\ldots, s_k}\left(\cK\otimes \cE_{p_1,\ldots, p_k}\right)\subset  \cK\otimes \cE_{s_1+p_1,\ldots, s_k+p_k}$$
 for every $(p_1,\ldots, p_k)\in \ZZ^k$
and  $(s_1,\ldots, s_k)\in \ZZ^k$.

\begin{definition} An operator  $A\in B(\cK\otimes \bigotimes_{s=1}^k  F^2(H_{n_s}))$ is said to be multi-homogeneous of degree $(s_1,\ldots, s_k)\in \ZZ^k$ if
$$
A\left(\cK\otimes \cE_{p_1,\ldots, p_k}\right)\subset  \cK\otimes \cE_{s_1+p_1,\ldots, s_k+p_k}
$$
 for every $(p_1,\ldots, p_k)\in \ZZ^k$
\end{definition}

\begin{lemma}  \label{lem}  If $\{x_s\}_{s\in \ZZ}$ is a sequence of orthogonal vectors in $\cH$ such that $\sum_{s\in \ZZ} \|x_s\|^2<\infty$, then
$$
\sum_{s\in \ZZ}x_s=\lim_{N\to \infty }\sum_{| s |\leq N}\left(1-\frac{| s |}{N+1}\right)x_s,
$$
where the convergence is in norm.
\end{lemma}
We omit the proof of the lemma which is straightforward.
We recall from \cite{K} that if $\cX$ is a Banach space, $\varphi$ is a continuous $\cX$-valued function on $\TT$ and $\kappa_n$ is a summability kernel, then
$$
\varphi(0)=\lim_{n\to\infty}\frac{1}{2\pi}\int_0^{2\pi} \kappa_n(e^{i\theta})\varphi(e^{i\theta})d\theta.
$$
We use this result to prove the following

\begin{proposition}\label{homo-decomp}
If $T\in B(\cK\bigotimes \otimes_{s=1}^k  F^2(H_{n_s}))$ and
   $\{T_{s_1,\ldots, s_k}\}_{(s_1,\ldots, s_k)\in \ZZ^k}$ are the multi-homogeneous parts of $T$, then
   $$
   Tf=\lim_{N_1\to \infty}\ldots \lim_{N_k\to \infty} \sum_{(s_1,\ldots, s_k)\in \ZZ^k, |s_j|\leq N_j}
   \left(1-\frac{|s_1|}{N_1+1}\right)\cdots \left(1-\frac{|s_k|}{N_k+1}\right) T_{s_1,\ldots, s_k}f
   $$
for every $f\in \cK\otimes \bigotimes_{s=1}^k  F^2(H_{n_s})$, where the convergence is in norm. Moreover,
$$
Tf=\lim_{N_1\to \infty}\ldots \lim_{N_k\to \infty} \sum_{(s_1,\ldots, s_k)\in \ZZ^k, |s_j|\leq N_j} T_{s_1,\ldots, s_k}f
$$
for every $f\in \cK\otimes\cE_{p_1,\ldots, p_k}$  and every $(p_1,\ldots, p_k)\in \ZZ^k$, where the convergence is in norm.
\end{proposition}
\begin{proof}
Let $f\in \cK\otimes\cE_{p_1,\ldots, p_k}$  and $\psi:\RR^k\to \cK\otimes \bigotimes_{s=1}^k  F^2(H_{n_s})$ be the continuous function defined by
$$\psi(\theta_1,\ldots, \theta_k):=(I_\cK\otimes \Gamma(e^{i\theta_1},\ldots, e^{i\theta_k}) )T(I_\cK\otimes \Gamma(e^{i\theta_1},\ldots, e^{i\theta_k})^*)f.
$$
For each $j \in \{1,\ldots, k\}$,  we consider the Fej\' er kernel $K_{N_j}(e^{i\theta_j}):= \sum_{| s_j |\leq N_j}
\left(1-\frac{| s_j |}{N_j+1}\right) e^{is_j\theta_j}.
$
According to the remark preceding the proposition, we have
$$
\psi(0,\theta_2,\ldots, \theta_k)=\lim_{N_1\to \infty} \frac{1}{2\pi} \int_0^{2\pi}K_{N_1}(e^{i\theta_1}) \psi(\theta_1,\ldots, \theta_k)d\theta_1.
$$
Similarly, we obtain
$$
\psi(0,0,\theta_3,\ldots, \theta_k)=\lim_{N_2\to \infty} \frac{1}{2\pi} \int_0^{2\pi}K_{N_2}(e^{i\theta_2}) \psi(0,\theta_2,\ldots, \theta_k)d\theta_2.
$$
Continuing this process and combining the resulting relations, we deduce that
\begin{equation*}
\begin{split}
Tf&=\psi(0,\ldots, 0)\\
&=\lim_{N_1\to \infty}\ldots \lim_{N_k\to \infty} \left( \frac{1}{2\pi}\right)^k \int_0^{2\pi}\cdots \int_0^{2\pi}K_{N_1}(e^{i\theta_1})\cdots K_{N_k}(e^{i\theta_k})\psi(\theta_1,\ldots, \theta_k)d\theta_1\ldots d\theta_k\\
&=\lim_{N_1\to \infty}\ldots \lim_{N_k\to \infty} \left(\prod_{j=1}^k \sum_{|s_j|\leq N_j}\left(1-\frac{| s_j |}{N_j+1}\right) \right) \\
&\qquad \times \left( \frac{1}{2\pi}\right)^k \int_0^{2\pi}\cdots \int_0^{2\pi} e^{-i(s_1+p_1)\theta_1}\cdots e^{-i(s_k+p_k)\theta_k}\left(I_\cK\otimes \Gamma (e^{i {\boldsymbol\theta}})\right) T\left(I_\cK\otimes \Gamma (e^{i {\boldsymbol\theta}})\right)^* f d\theta_1\ldots d\theta_k\\
&=\lim_{N_1\to \infty}\ldots \lim_{N_k\to \infty} \sum_{(s_1,\ldots, s_k)\in \ZZ^k, |s_j|\leq N_j}
   \left(1-\frac{|s_1|}{N_1+1}\right)\cdots \left(1-\frac{|s_k|}{N_k+1}\right)T_{s_1,\ldots, s_k}f
   \end{split}
   \end{equation*}
for every $f\in \cK\otimes \bigotimes_{s=1}^k  F^2(H_{n_s})$.

To prove the second part of the proposition, assume that $f\in \cE_{p_1,\ldots, p_k}$  and  $(p_1,\ldots, p_k)\in \ZZ^k$. In this case, we have  $T_{s_1,\ldots, s_k}f\in \cE_{s_1+p_1,\ldots, s_k+p_k}$ and, consequently,
the vectors $\{T_{s_1,\ldots, s_k}f\}_{(s_1,\ldots, s_k)\in \ZZ^k}$ are pairwise orthogonal. Moreover,
$T_{s_1,\ldots, s_k}f=(I_\cK\otimes{\bf P}_{s_1+p_1,\ldots, s_k+p_k})Tf$ and
$$
\sum_{(p_1,\ldots, p_k)\in \ZZ^k}T_{s_1,\ldots, s_k}f=\left(\sum_{(p_1,\ldots, p_k)\in \ZZ^k}(I_\cK\otimes{\bf P}_{s_1+p_1,\ldots, s_k+p_k})\right)Tf.
$$
Hence, we deduce that $\sum_{(p_1,\ldots, p_k)\in \ZZ^k}T_{s_1,\ldots, s_k}f$ is convergent  in $\cK\otimes \bigotimes_{s=1}^k  F^2(H_{n_s})$ and
$$\sum_{(p_1,\ldots, p_k)\in \ZZ^k}\|T_{s_1,\ldots, s_k}f\|^2\leq \|Tf\|^2.
$$
 Due to the first part of the proposition, we have
 \begin{equation}\label{TAs}
   Tf=\lim_{N_1\to \infty}  \sum_{ |s_1|\leq N_1}
   \left(1-\frac{|s_1|}{N_1+1}\right)  A_{s_1}f,
   \end{equation}
where
$$A_{s_1}f:=\lim_{N_2\to \infty}\ldots \lim_{N_k\to \infty} \sum_{(s_2,\ldots, s_k)\in \ZZ^{k-1}, |s_j|\leq N_j}
   \left(1-\frac{|s_2|}{N_2+1}\right)\cdots \left(1-\frac{|s_k|}{N_k+1}\right)T_{s_1,s_2\ldots, s_k}f.
   $$
  Note that  $A_{s_1}f\in \cK\otimes \cE^1_{s_1+p_1}\otimes F^2(H_{n_2})\otimes \cdots \otimes F^2(H_{n_k})$ , where $\cE^1_{s_1+p_1}:=\text{\rm span} \{e^1_\alpha: \ \alpha\in \FF_{n_1}^+, |\alpha|=s_1+p_1\}$, and
  $$
  \sum_{s_1\in \ZZ}\|A_{s_1 }f\|^2\leq  \sum_{(p_1,\ldots, p_k)\in \ZZ^k}\|T_{s_1,\ldots, s_k}f\|^2\leq \|Tf\|^2.
  $$
   Since the sequence $\{A_{s_1}f\}_{s_1\in \ZZ}$  consists of pairwise orthogonal vectors, we can apply Lemma \ref{lem}   and use relation \eqref{TAs}, to deduce that
  $$
  Tf=\lim_{N_1\to \infty}  \sum_{ |s_1|\leq N_1}
      A_{s_1}f.
   $$
  Similar arguments lead to the relation
   $$
  A_{s_1}f=\lim_{N_2\to \infty}  \sum_{ |s_2|\leq N_2}
      A_{s_1, s_2}f,
   $$
   where
   $$A_{s_1, s_2}f:=\lim_{N_3\to \infty}\ldots \lim_{N_k\to \infty} \sum_{(s_3,\ldots, s_k)\in \ZZ^{k-2}, |s_j|\leq N_j}
   \left(1-\frac{|s_3|}{N_3+1}\right)\cdots \left(1-\frac{|s_k|}{N_k+1}\right)T_{s_1,s_2\ldots, s_k}f.
   $$
   Iterating this process, we deduce that
    $$
  A_{s_1,\ldots, s_{k-1}}f=\lim_{N_k\to \infty}  \sum_{ |s_k|\leq N_k}
      A_{s_1,\ldots, s_k}f,
   $$
   where
   $A_{s_1, \ldots, s_k}f:=\ T_{s_1,s_2\ldots, s_k}f.
   $
   Combining these relations, we deduce that
   $$
Tf=\lim_{N_1\to \infty}\ldots \lim_{N_k\to \infty} \sum_{(s_1,\ldots, s_k)\in \ZZ^k, |s_j|\leq N_j} T_{s_1,\ldots, s_k}f
$$
for every $f\in \cK\otimes\cE_{p_1,\ldots, p_k}$ , which completes the proof.
   \end{proof}

\begin{theorem}\label{BH-cond} If $T\in B(\cK\otimes \bigotimes_{s=1}^k  F^2(H_{n_s}))$ satisfies the Brown-Halmos condition \eqref{BH}, then so does the multi-homogeneous part $T_{s_1,\ldots, s_k}$ for every $(s_1,\ldots, s_k)\in \ZZ^k$.
\end{theorem}

\begin{proof}
As in the proof of Proposition \ref{WWW}, one can prove that if $s\in\{1,\ldots, k\}$,  $j\in \{1,\ldots, n_s\}$, and $f\in
\cK\otimes \bigotimes_{i=1}^k  F^2(H_{n_i})$, then
$$
\widetilde{\bf \Lambda}_{s,j} ^*(I_\cK\otimes {\bf \Omega}_s )\left(I_\cK\otimes
\Gamma (e^{i {\boldsymbol\theta}})\right)f=e^{i\theta_s} \left(I_\cK\otimes
\Gamma (e^{i {\boldsymbol\theta}})\right)\widetilde{\bf \Lambda}_{s,j}^* (I_\cK\otimes {\bf \Omega}_s )f.
$$
Hence, we deduce that
\begin{equation}
\label{LAO}
\widetilde{\bf \Lambda}_{s}^* (I_\cK\otimes {\bf \Omega}_s )\left(I_\cK\otimes
\Gamma (e^{i {\boldsymbol\theta}})\right)
=e^{i\theta_s} {\bf diag}_{n_i}\left(I_\cK\otimes
\Gamma (e^{i {\boldsymbol\theta}})\right)\widetilde{\bf \Lambda}_{s}^* (I_\cK\otimes {\bf \Omega}_s )
\end{equation}
for every $s\in\{1,\ldots, k\}$.

Note that
$$
\left(I_\cK\otimes
\Gamma (e^{i {\boldsymbol\theta}})\right)\widetilde{\bf \Lambda}_{s,j} =e^{i\theta_s}\widetilde{\bf \Lambda}_{s,j} \left(I_\cK\otimes
\Gamma (e^{i {\boldsymbol\theta}})\right)
$$
and, consequently,
$$
\left(I_\cK\otimes
\Gamma (e^{i {\boldsymbol\theta}})\right)\widetilde{\bf \Lambda}_{s,\alpha} T\widetilde{\bf \Lambda}_{s,\alpha} ^*
\left(I_\cK\otimes
\Gamma (e^{i {\boldsymbol\theta}})\right)^*
=\widetilde{\bf \Lambda}_{s,\alpha} \left(I_\cK\otimes
\Gamma (e^{i {\boldsymbol\theta}})\right)T \left(I_\cK\otimes \Gamma (e^{i {\boldsymbol\theta}})\right)^*
\widetilde{\bf \Lambda}_{s,\alpha} ^*
$$
for every $s\in \{1,\ldots, k\}$ and $\alpha\in \FF_{n_s}^+$.

Since $\widetilde{\bf \Lambda}_{i,j}:=I_\cK\otimes {\bf \Lambda}_{i,j}$ and  $\widetilde{\bf \Lambda}_i:=[\widetilde{\bf \Lambda}_{i,1}\cdots \widetilde{\bf \Lambda}_{i,n_i}]$, Proposition \ref{WWW} implies
$\widetilde{\bf \Lambda}_s^*(I_\cK\otimes {\bf \Omega}_s)=(\widetilde{\bf \Lambda}_i^*\widetilde{\bf \Lambda}_i)^{-1} \widetilde {\bf \Lambda}_s^*$
Consequently,
the Cauchy dual operator
$$\widetilde{\bf \Lambda}_s': \left(\cK\otimes\bigotimes_{i=1}^k F^2(H_{n_i})\right)^{(n_s)}\to \cK\otimes\bigotimes_{i=1}^k F^2(H_{n_i})$$
 defined by $\widetilde{\bf \Lambda}_s':=\widetilde{\bf \Lambda}_s(\widetilde{\bf \Lambda}_s^*\widetilde{\bf \Lambda}_s)^{-1}$ satisfies the relation
 $\widetilde{\bf \Lambda}_s'=(I_\cK\otimes {\bf \Omega}_s)\widetilde {\bf \Lambda}_s$.

Now, note that, for each $s\in \{1,\ldots, k\}$, we have
\begin{equation*}
\begin{split}
&\widetilde{\bf \Lambda}_s^{\prime *} T_{s_1,\ldots, s_k}\widetilde{\bf \Lambda}_s' \\
&=
\widetilde {\bf \Lambda}_s^*(I_\cK\otimes {\bf \Omega}_s) T_{s_1,\ldots, s_k}(I_\cK\otimes {\bf \Omega}_s)\widetilde {\bf \Lambda}_s\\
&=
\left(\frac{1}{2\pi}\right)^k
\int_0^{2\pi}\cdots \int_0^{2\pi} e^{-is_1\theta_1}\cdots e^{-is_k\theta_k}\widetilde {\bf \Lambda}_s^*(I_\cK\otimes {\bf \Omega}_s)\left(I_\cK\otimes \Gamma (e^{i {\boldsymbol\theta}})\right) T
\left(I_\cK\otimes \Gamma (e^{i {\boldsymbol\theta}})\right)^* (I_\cK\otimes {\bf \Omega}_s)\widetilde {\bf \Lambda}_sd\theta_1\ldots d\theta_k
\\
&=\left(\frac{1}{2\pi}\right)^k
\int_0^{2\pi}\cdots \int_0^{2\pi} e^{-is_1\theta_1}\cdots e^{-is_k\theta_k}
{\bf diag}_{n_s}\left(I_\cK\otimes
\Gamma (e^{i {\boldsymbol\theta}})\right)\widetilde{\bf \Lambda}_{s}^* (I_\cK\otimes {\bf \Omega}_s) T
 (I_\cK\otimes {\bf \Omega}_s)\widetilde{\bf \Lambda}_{s}{\bf diag}_{n_s}\left(I_\cK\otimes
\Gamma (e^{i {\boldsymbol\theta}})^*\right)
\\
&=\left(\frac{1}{2\pi}\right)^k
\int_0^{2\pi}\cdots \int_0^{2\pi} e^{-is_1\theta_1}\cdots e^{-is_k\theta_k}
{\bf diag}_{n_s}\left(I_\cK\otimes
\Gamma (e^{i {\boldsymbol\theta}})\right)\widetilde{\bf \Lambda}_s^{\prime *} T\widetilde{\bf \Lambda}_s'{\bf diag}_{n_s}\left(I_\cK\otimes
\Gamma (e^{i {\boldsymbol\theta}})^*\right)
\\
&=\left(\frac{1}{2\pi}\right)^k
\int_0^{2\pi}\cdots \int_0^{2\pi} e^{-is_1\theta_1}\cdots e^{-is_k\theta_k}
{\bf diag}_{n_s}\left(I_\cK\otimes
\Gamma (e^{i {\boldsymbol\theta}})\right)\\
&\qquad \qquad \times {\bf diag}_{n_s}\left( \sum_{j=0}^{m_s-1} (-1)^j \left(\begin{matrix}  m_s\\j+1
\end{matrix}\right)\sum_{\beta\in \FF_{n_s}^+, |\beta|=j} \widetilde{\bf \Lambda}_{s,\beta}T\widetilde{\bf \Lambda}_{s,\beta}^*\right){\bf diag}_{n_s}\left(I_\cK\otimes
\Gamma (e^{i {\boldsymbol\theta}})^*\right)
\\
&=\left(\frac{1}{2\pi}\right)^k
\int_0^{2\pi}\cdots \int_0^{2\pi} e^{-is_1\theta_1}\cdots e^{-is_k\theta_k}\\
&\qquad \qquad \times
 {\bf diag}_{n_s}\left( \sum_{j=0}^{m_s-1} (-1)^j \left(\begin{matrix}  m_s\\j+1
\end{matrix}\right)\sum_{\beta\in \FF_{n_s}^+, |\beta|=j} \left(I_\cK\otimes
\Gamma (e^{i {\boldsymbol\theta}})\right)\widetilde{\bf \Lambda}_{s,\beta}T\widetilde{\bf \Lambda}_{s,\beta}^*\left(I_\cK\otimes
\Gamma (e^{i {\boldsymbol\theta}})^*\right)\right) \\
&=\left(\frac{1}{2\pi}\right)^k
\int_0^{2\pi}\cdots \int_0^{2\pi} e^{-is_1\theta_1}\cdots e^{-is_k\theta_k}\\
&\qquad \qquad \times
 {\bf diag}_{n_s}\left( \sum_{j=0}^{m_s-1} (-1)^j \left(\begin{matrix}  m_s\\j+1
\end{matrix}\right)\sum_{\beta\in \FF_{n_s}^+, |\beta|=j} \widetilde{\bf \Lambda}_{s,\beta}\left(I_\cK\otimes
\Gamma (e^{i {\boldsymbol\theta}})\right)T\left(I_\cK\otimes
\Gamma (e^{i {\boldsymbol\theta}})^*\right)\right) \widetilde{\bf \Lambda}_{s,\beta}^*\\
&=
{\bf diag}_{n_s}\left( \sum_{j=0}^{m_s-1} (-1)^j \left(\begin{matrix}  m_s\\j+1
\end{matrix}\right)\sum_{\beta\in \FF_{n_s}^+, |\beta|=j} \widetilde{\bf \Lambda}_{s,\beta}T_{s_1,\ldots, s_k} \widetilde{\bf \Lambda}_{s,\beta}^*\right).
\end{split}
\end{equation*}
The proof is complete.
\end{proof}

In what follows we use the standard notation $s^+:=\max\{s,0\}$ and $s^-:=\max\{-s,0\}$ for every $s\in \ZZ$.

\begin{theorem} \label{multi-homo}
 Let $A\in B(\cK\otimes \bigotimes_{s=1}^k  F^2(H_{n_s}))$ be a  multi-homogeneous  operator of degree $(s_1,\ldots, s_k)\in \ZZ^k$
 and satisfying the Brown-Halmos condition \eqref{BH}. Then
 $$
 Af=q_{s_1,\ldots, s_k}({\bf W}, {\bf W}^*) f,\qquad f\in \cK\otimes \bigotimes_{s=1}^k  F^2(H_{n_s}),
 $$
 where
 $$
 q_{s_1,\ldots, s_k}({\bf W}, {\bf W}^*):=
\sum_{{\alpha_i,\beta_i\in \FF_{n_i}^+, i\in \{1,\ldots, k\}}\atop{|\alpha_i|=s_i^+, |\beta_i|=s_i^-}} A_{(\alpha_1,\ldots,\alpha_k,\beta_1,\ldots, \beta_k)}\otimes {\bf W}_{1,\alpha_1}\cdots {\bf W}_{k,\alpha_k}{\bf W}_{1,\beta_1}^*\cdots {\bf W}_{k,\beta_k}^*,
$$
and  the coefficients $A_{(\alpha_1,\ldots,\alpha_k,\beta_1,\ldots, \beta_k)}\in B(\cK)$ are given by
\begin{equation} \label{AA}
\left<A_{(\alpha_1,\ldots,\alpha_k,\beta_1,\ldots, \beta_k)}h,\ell\right>:=\left(\prod_{i=1}^k \sqrt{b_{i,\alpha_i}^{(m_i)} b_{i,\beta_i}^{(m_i)}}\right)\left<A(h\otimes x), \ell\otimes y\right>,\qquad  h,\ell\in \cK,
\end{equation}
where $x:=x_1\otimes \cdots \otimes x_k$, $y=y_1\otimes \cdots \otimes y_k$ with
\begin{equation*}
\begin{cases} x_i=e^i_{\beta_i} \text{ and } y_i=1,& \quad \text{if } s_i\leq 0\\
 x_i=1 \text{ and } y_i=e^i_{\alpha_i},& \quad \text{if } s_i>0
 \end{cases}
 \end{equation*}
 for every $i\in \{1,\ldots, k\}$.
\end{theorem}
\begin{proof}
Fix $(s_1,\ldots, s_k)\in \ZZ^k$ and let us prove that
\begin{equation}
\label{M}
Af=q_{s_1,\ldots, s_k}({\bf W}, {\bf W}^*) f,\qquad f\in  \cM,
\end{equation}
for every subspace $\cM$ of the form $\cK\otimes \cE_{p_1,\ldots, p_k}$, where  $p_1,\ldots, p_k\in \{0,1,\ldots \}$.

{\bf Case I.}  Assume that there is $i_0\in \{1,\ldots, k\}$ such that  $s_{i_0}<0$.

 If $p_{i_0}<s_{i_0}^-$, we have  $s_{i_0}+ p_{i_0}<0 $ and
$$A(\cK\otimes \cE_{p_1,\ldots, p_k})\subset \cK\otimes \cE_{s_1+p_1,\ldots, s_k+p_k}=\{0\}.
$$
Since $q_{s_1,\ldots, s_k}({\bf W}, {\bf W}^*) f=0$ for every $f\in \cK\otimes \cE_{p_1,\ldots, p_k}$, we deduce that
relation \eqref{M} holds for every subspace $\cM=\cK\otimes \cE_{p_1,\ldots, p_k}$, where $p_{i_0}<s_{i_0}^-$ and
$p_1,\ldots, p_{i_0-1}, p_{i_0+1},\ldots p_k\in \{0,1,\ldots \}$.

{\bf Case II.}   Assume that there is at least one $i_0\in \{1,\ldots, k\}$ such that $s_{i_0}\geq 0$.

Without loss of generality, we may
assume that there is $d\in \{1,\ldots, k\} $ such that $s_i\geq 0$ if $i\in \{1,\ldots, d\}$ and, if $d<k$, then $s_i<0$ for every $i\in \{d+1,\ldots, k\}$. In what follows we prove that relation \eqref{M} holds for every
$\cM=\cK\otimes \cE_{p_1,\ldots, p_k}$, where $p_1,\ldots, p_d\in \{0,1,\ldots\}$ and $p_j=s_j^-$ for $j\in\{d+1,\ldots, k\}$.

The first step is to prove relation \eqref{M} for every subspace $\cM=\cK\otimes \cE_{p_1,0,\ldots,0, s_{d+1}^-,\ldots,s_k^-}$ where $p_1\in \{0,1,2,\ldots\}$. We proceed by induction over $p_1$.
Let $p_1=0$ and note that if $f\in \cK\otimes\cE_{0,\ldots, 0, s_{d+1}^-,\ldots, s_k^-}$ then
$Af\in \cK\otimes \cE_{s_1,\ldots, s_d,0,\ldots, 0}$.
Let $\alpha_i,\beta_i\in \FF_{n_i}^+$ be such that
$$
|\alpha_i|=\begin{cases}s_i^+,& \  \text{ if } i\in \{1,\ldots, d\},\\
1, &\  \text{ if } i\in \{d+1,\ldots, k\},
\end{cases}
\quad \text{and}\quad
|\beta_i|=\begin{cases}1,& \  \text{ if } i\in \{1,\ldots, d\},\\
s_i^-, &\  \text{ if } i\in \{d+1,\ldots, k\}.
\end{cases}
$$
Note that, due to relation \eqref{AA} and the definition of the universal model ${\bf W}$, for every $h, \ell\in \cK$,  we have
\begin{equation*}
\begin{split}
&\left<A(h\otimes e_{\beta_1}^1\otimes\cdots \otimes e_{\beta_k}^k), \ell\otimes e_{\alpha_1}^1\otimes\cdots \otimes e_{\alpha_k}^k\right>\\
&\qquad\qquad=\left(\prod_{i=1}^k\frac{1}{ \sqrt{b_{i,\alpha_i}^{(m_i)} b_{i,\beta_i}^{(m_i)}}}\right)
\left<A_{(\alpha_1,\ldots,\alpha_k,\beta_1,\ldots, \beta_k)}h,\ell\right>\\
&\qquad\qquad=\left<q_{s_1,\ldots, s_k}({\bf W}, {\bf W}^*)(h\otimes e_{\beta_1}^1\otimes\cdots \otimes e_{\beta_k}^k), \ell\otimes e_{\alpha_1}^1\otimes\cdots \otimes e_{\alpha_k}^k\right>.
\end{split}
\end{equation*}
Hence, and using the fact that $q_{s_1,\ldots, s_k}({\bf W}, {\bf W}^*)\left(\cK\otimes \cE_{0,\ldots,0, s_{d+1}^-,\ldots,s_k^-}\right)\subset \cK\otimes \cE_{s_1^+,\ldots,s_d^+, 0,\ldots,0}$, we deduce that relation \eqref{M} holds for $\cM=\cK\otimes \cE_{0,\ldots,0, s_{d+1}^-,\ldots,s_k^-}$.
Now, let $q\in \NN$ and  assume that  relation \eqref{M} holds for $\cM=\cK\otimes \cE_{p_1, 0,\ldots,0, s_{d+1}^-,\ldots,s_k^-}$ for every $p_1\in \{0,1,\ldots, q\}$. Let $f\in \cK\otimes \cE_{q+1, 0,\ldots,0, s_{d+1}^-,\ldots,s_k^-}$. Using Proposition \ref{WWW} and the definition of the operator ${\bf \Omega}_1$, we obtain
\begin{equation}\label{triu}
\begin{split}
\widetilde{\bf \Lambda}_1^*(I_\cK\otimes {\bf \Omega}_1)A(I_\cK\otimes {\bf \Omega}_1)\widetilde{\bf \Lambda}_1\widetilde{\bf \Lambda}_1^*f
&=
\widetilde{\bf \Lambda}_1^*(I_\cK\otimes {\bf \Omega}_1)A\widetilde{\bf \Lambda}_1  (\widetilde{\bf \Lambda}_1^*\widetilde{\bf \Lambda}_1)^{-1}\widetilde{\bf \Lambda}_1^*f\\
&=
\widetilde{\bf \Lambda}_1^*(I_\cK\otimes {\bf \Omega}_1)A\left( {\bf P}_{\text{\rm range} \widetilde{\bf \Lambda}_1}\otimes I_{F^2(H_{n_2})}\otimes \cdots \otimes I_{F^2(H_{n_k})}\right)f\\
&=
\widetilde{\bf \Lambda}_1^*(I_\cK\otimes {\bf \Omega}_1)Af\\
&=
\frac{m_1+q+s_1}{q+s_1+1}\widetilde{\bf \Lambda}_1^*Af.
\end{split}
\end{equation}
Using  again Proposition \ref{WWW} (see items (i) and (iii)), we deduce that
\begin{equation}
\label{dinti}
\widetilde{\bf \Lambda}_1^*(I_\cK\otimes {\bf \Omega}_1)={\bf diag}_{n_1}\left({\bf \Psi_{\widetilde {\bf \Lambda}_1}}(I) \right)\widetilde{\bf \Lambda}_1^*,
\end{equation}
where
$$
{\bf \Psi_{\widetilde {\bf \Lambda}_1}}(X):={\bf diag}_{n_1}\left( \sum_{j=0}^{m_1-1} (-1)^j \left(\begin{matrix}  m_1\\j+1
\end{matrix}\right)\sum_{\beta\in \FF_{n_1}^+, |\beta|=j} {\bf \Lambda}_{i,\beta} X{\bf \Lambda}_{i,\beta}^*\right), \qquad X\in B(\otimes_{s=1}^k F^2(H_{n_s})).
$$
Note that, for every $\beta\in \FF_{n_1}^+$, $j\in \{1,\ldots, n_1\}$, we have  $\Lambda_{1,\beta}^*\Lambda_{1,g_j}^*f\in \cE_{p,0,\ldots,0,s_{d+1}^-,\ldots, s_k^-}$, where  $p\in \{0,\ldots, q\}$. Due to the induction hypothesis, we have  $A=q_{s_1,\ldots, s_k}({\bf W}, {\bf W}^*)$ on the subspaces $\cK\otimes \cE_{p_1,0,\ldots,0,s_{d+1}^-,\ldots, s_k^-}$ with $p_1\in\{0,\ldots, q\}$. Consequently,

\begin{equation}
\label{2c}
{\bf diag}_{n_1}\left({\bf \Psi_{\widetilde {\bf \Lambda}_1}}(A) \right)\widetilde{\bf \Lambda}_1^*f=
{\bf diag}_{n_1}\left({\bf \Psi_{\widetilde {\bf \Lambda}_1}}(q_{s_1,\ldots, s_k}({\bf W}, {\bf W}^*)) \right)\widetilde{\bf \Lambda}_1^*f
\end{equation}
Using the equation \eqref{dinti} and  the definition of  ${\bf \Omega}_1$, we deduce that
\begin{equation*}
\begin{split}
{\bf diag}_{n_1}\left({\bf \Psi_{\widetilde {\bf \Lambda}_1}}(q_{s_1,\ldots, s_k}({\bf W}, {\bf W}^*)) \right)\widetilde{\bf \Lambda}_1^*f
&=
{\bf diag}_{n_1}\left(q_{s_1,\ldots, s_k}({\bf W}, {\bf W}^*){\bf \Psi_{\widetilde {\bf \Lambda}_1}}(I) \right)\widetilde{\bf \Lambda}_1^*f\\
&=
{\bf diag}_{n_1}\left(q_{s_1,\ldots, s_k}({\bf W}, {\bf W}^*)\right){\bf diag}_{n_1}\left({\bf \Psi_{\widetilde {\bf \Lambda}_1}}(I) \right)\widetilde{\bf \Lambda}_1^*f\\
&=
{\bf diag}_{n_1}\left(q_{s_1,\ldots, s_k}({\bf W}, {\bf W}^*)\right)\widetilde{\bf \Lambda}_1^*(I_\cK\otimes {\bf \Omega}_1f\\
&=\frac{m_1+q}{q+1}{\bf diag}_{n_1}\left(q_{s_1,\ldots, s_k}({\bf W}, {\bf W}^*)\right)\widetilde{\bf \Lambda}_1^* f.
\end{split}
\end{equation*}
Hence, using relations \eqref{triu}, \eqref{2c}, and the fact that $A$ satisfies the Brown-Halmos condition \eqref{BH} (when $i=1$), i.e
$$
\widetilde{\bf \Lambda}_1^*(I_\cK\otimes {\bf \Omega}_1)A (I_\cK\otimes {\bf \Omega}_1)\widetilde{\bf \Lambda}_1
={\bf diag}_{n_1}\left({\bf \Psi_{\widetilde {\bf \Lambda}_1}}(A)\right),
$$
we deduce that
\begin{equation*}
\begin{split}
\frac{m_1+q+s_1}{q+s_1+1} \widetilde{\bf \Lambda}_1^*Af
&=
\widetilde{\bf \Lambda}_1^*(I_\cK\otimes {\bf \Omega}_1)A (I_\cK\otimes {\bf \Omega}_1)\widetilde{\bf \Lambda}_1
\widetilde{\bf \Lambda}_1^*f\\
&=
{\bf diag}_{n_1}\left({\bf \Psi_{\widetilde {\bf \Lambda}_1}}(A)\right)\widetilde{\bf \Lambda}_1^*f\\
&=
{\bf diag}_{n_1}\left({\bf \Psi_{\widetilde {\bf \Lambda}_1}}(q_{s_1,\ldots, s_k}({\bf W}, {\bf W}^*)) \right)\widetilde{\bf \Lambda}_1^*f\\
&=\frac{m_1+q}{q+1}{\bf diag}_{n_1}\left(q_{s_1,\ldots, s_k}({\bf W}, {\bf W}^*)\right)\widetilde{\bf \Lambda}_1^* f.
\end{split}
\end{equation*}
Applying the operator $
\widetilde{\bf \Lambda}_1(\widetilde{\bf \Lambda}_1^*\widetilde{\bf \Lambda}_1)^{-1}=(I_\cK\otimes {\bf \Omega}_1)\widetilde{\bf \Lambda}_1
$
 (see Proposition \ref{WWW}) to both sides of the relation above and taking into account that
$\widetilde{\bf \Lambda}_1(\widetilde{\bf \Lambda}_1^*\widetilde{\bf \Lambda}_1)^{-1}\widetilde{\bf \Lambda}_1^*$ is the orthogonal projection of $\cK\otimes \bigotimes_{s=1}^k  F^2(H_{n_s})$
onto $\cK\otimes (F^2(H_{n_1})\ominus \CC)\otimes F^2(H_{n_2})\otimes \cdots\otimes F^2(H_{n_k})$, we obtain
\begin{equation}
\label{mqs}
\begin{split}
\frac{m_1+q+s_1}{q+s_1+1} Af &= \widetilde{\bf \Lambda}_1(\widetilde{\bf \Lambda}_1^*\widetilde{\bf \Lambda}_1)^{-1}\left(\frac{m_1+q+s_1}{q+s_1+1} \widetilde{\bf \Lambda}_1^*Af\right)\\
&=
\frac{m_1+q}{q+1}(I_\cK\otimes {\bf \Omega}_1)\widetilde{\bf \Lambda}_1{\bf diag}_{n_1}\left(q_{s_1,\ldots, s_k}({\bf W}, {\bf W}^*)\right)\widetilde{\bf \Lambda}_1^* f\\
&=
\frac{m_1+q}{q+1}\frac{m_1+q+s_1}{q+s_1+1}q_{s_1,\ldots, s_k}({\bf W}, {\bf W}^*)\widetilde{\bf \Lambda}_1\widetilde{\bf \Lambda}_1^*f.
\end{split}
\end{equation}
The latter equality is due to the relation $\Lambda_{1,j} W_{1,j'}=W_{1,j'}\Lambda_{1,j}$ for every $j,j'\in \{1,\ldots,n_1\}$,  the definition of ${\bf \Omega}_1$, and the fact that
$q_{s_1,\ldots, s_k}({\bf W}, {\bf W}^*)\widetilde{\bf \Lambda}_1\widetilde{\bf \Lambda}_1^*f\in \cK\otimes
\cE_{q+1+s_1, s_2,\ldots, s_d, 0,\ldots,0}$.
A careful calculation reveals  that
\begin{equation*}
 {\bf \Lambda}_1 {\bf \Lambda}_1^*=\left(\sum_{j=1}^\infty \frac{1}{m_1+j-1} {\bf P}_{\{\text{\rm span}\,  e^1_\alpha: \ \alpha\in \FF_{n_1}^+, |\alpha|=j\}}\right) \otimes I_{F^2(H_{n_2})}\otimes\cdots\otimes I_{F^2(H_{n_k})}.
\end{equation*}
Since $f\in \cK\otimes
\cE_{q+1, 0,\ldots, 0, s_{d+1}^-,\ldots,s_k^-}$, we deduce that ${\bf \Lambda}_1{\bf \Lambda}_1^*f=\frac{q+1}{m_1+q}f$.
Consequently, relation \eqref{mqs} implies
$Af=q_{s_1,\ldots, s_k}({\bf W}, {\bf W}^*) f$ for every $f\in \cK\otimes
\cE_{q+1, 0,\ldots, 0, s_{d+1}^-,\ldots,s_k^-}$, which completes the induction.
In a similar manner, replacing $\widetilde{\bf \Lambda}_1$, $I_\cK\otimes {\bf \Omega}_1$, ${\bf \Psi}_{\widetilde{\bf \Lambda}_1}$ with $\widetilde{\bf \Lambda}_i$, $I_\cK\otimes {\bf \Omega}_i$, ${\bf \Psi}_{\widetilde{\bf \Lambda}_i}$ respectively, as $i=1,\ldots, d$, we can prove by induction over $p_i\in \{0,1, \ldots\}$ that relation \eqref{M} holds for every subspace $\cM$ of the form
$\cK\otimes
\cE_{p_1, \ldots, p_d, s_{d+1}^-,\ldots,s_k^-}$, where $p_1,\ldots, p_d\in \{0,1,\ldots\}$.

We remark that if $d=k$, the proof of the theorem is complete. Assume now  that $1\leq d<k$. We prove by induction that
relation \eqref{M} holds for every
$f\in \cK\otimes
\cE_{p_1, \ldots, p_d, q_{d+1},s_{d+2}^-,\ldots,s_k^-}$ for every
$p_1,\ldots, p_d\in \{0,1,\ldots\}$ and any $q_{d+1}\geq s_{d+1}^-$.
Let $q\geq s_{d+1}^-$ and assume that the relation above holds for every $q_{d+1}\in \{s_{d+1}^-, s_{d+1}^-+1,\ldots, q\}$ and any
$p_1,\ldots, p_d\in \{0,1,\ldots\}$.
We need to show that
$f\in \cK\otimes
\cE_{p_1, \ldots, p_d, q +1,s_{d+2}^-,\ldots,s_k^-}$ for every
$p_1,\ldots, p_d\in \{0,1,\ldots\}$.
To this end, assume that $f\in \cK\otimes
\cE_{p_1, \ldots, p_d, q +1,s_{d+2}^-,\ldots,s_k^-}$ and note that
$Af\in \cK\otimes
\cE_{p_1+s_1, \ldots, p_d+s_d, q +1+s_{d+1},0,\ldots,0}$. As in the first part of the proof, replacing $\widetilde{\bf \Lambda}_1$, $I_\cK\otimes {\bf \Omega}_1$, ${\bf \Psi}_{\widetilde{\bf \Lambda}_1}$ with $\widetilde{\bf \Lambda}_{d+1}$, $I_\cK\otimes {\bf \Omega}_{d+1}$, ${\bf \Psi}_{\widetilde{\bf \Lambda}_{d+1}}$ respectively, we deduce that
\begin{equation}
\label{triu2}
\widetilde{\bf \Lambda}_{d+1}^*(I_\cK\otimes {\bf \Omega}_1)A(I_\cK\otimes {\bf \Omega}_{d+1})\widetilde{\bf \Lambda}_{d+1}\widetilde{\bf \Lambda}_{d+1}^*f
=\frac{m_{d+1}+q+s_{d+1}}{q+s_{d+1}+1}\widetilde{\bf \Lambda}_{d+1}^*Af
\end{equation}
and
\begin{equation}
\label{2c2}
{\bf diag}_{n_{d+1}}\left({\bf \Psi_{\widetilde {\bf \Lambda}_{d+1}}}(A) \right)\widetilde{\bf \Lambda}_{d+1}^*f=
{\bf diag}_{n_{d+1}}\left({\bf \Psi_{\widetilde {\bf \Lambda}_{d+1}}}(q_{s_1,\ldots, s_k}({\bf W}, {\bf W}^*)) \right)\widetilde{\bf \Lambda}_{d+1}^*f
\end{equation}
due to the induction hypothesis.
Since
\begin{equation*}
\widetilde{\bf \Lambda}_{d+1}^*(I_\cK\otimes {\bf \Omega}_{d+1})={\bf diag}_{n_{d+1}}\left({\bf \Psi_{\widetilde {\bf \Lambda}_{d+1}}}(I) \right)\widetilde{\bf \Lambda}_{d+1}^*,
\end{equation*}
and following the same type of arguments as in the first part of the proof, we deduce that
\begin{equation*}
\begin{split}
{\bf diag}_{n_{d+1}}\left({\bf \Psi_{\widetilde {\bf \Lambda}_{d+1}}}(q_{s_1,\ldots, s_k}({\bf W}, {\bf W}^*)) \right)\widetilde{\bf \Lambda}_{d+1}^*f
&=
{\bf diag}_{n_{d+1}}\left({\bf \Psi_{\widetilde {\bf \Lambda}_{d+1}}}(I)q_{s_1,\ldots, s_k}({\bf W}, {\bf W}^*) \right)\widetilde{\bf \Lambda}_{d+1}^*f\\
&=
{\bf diag}_{n_{d+1}}\left({\bf \Psi_{\widetilde {\bf \Lambda}_1}}(I) \right) \widetilde{\bf \Lambda}_{d+1}^*q_{s_1,\ldots, s_k}({\bf W}, {\bf W}^*) f\\
&= \widetilde{\bf \Lambda}_{d+1}^*(I_\cK\otimes {\bf \Omega}_{d+1})q_{s_1,\ldots, s_k}({\bf W}, {\bf W}^*) f
 \\
 &= \frac{m_{d+1}+q+s_{d+1}}{q+1+s_{d+1}}
 \widetilde{\bf \Lambda}_{d+1}^* q_{s_1,\ldots, s_k}({\bf W}, {\bf W}^*) f
 \\
&=\frac{m_{d+1}+q+s_{d+1}}{q+1+s_{d+1}}
{\bf diag}_{n_{d+1}}\left(q_{s_1,\ldots, s_k}({\bf W}, {\bf W}^*)\right)\widetilde{\bf \Lambda}_{d+1}^* f.
\end{split}
\end{equation*}
Consequently, using relations \eqref{triu2}, \eqref{2c2}, and the Brown-Halmos condition, we deduce that
\begin{equation}\label{3c}
\begin{split}
\frac{m_{d+1}+q+s_{d+1}}{q+s_{d+1}+1} \widetilde{\bf \Lambda}_{d+1}^*Af
 =\left(\frac{m_{d+1}+q+s_{d+1}}{q+1+s_{d+1}}\right)^2 \widetilde{\bf \Lambda}_{d+1}\widetilde{\bf \Lambda}_{d+1}^* q_{s_1,\ldots, s_k}({\bf W}, {\bf W}^*) f.
\end{split}
\end{equation}
Using the fact that ${\bf \Lambda}_{d+1} {\bf \Lambda}_{d+1}^*$ is equal to
\begin{equation*}
 I_{F^2(H_{n_1})}\otimes\cdots I_{F^2(H_{n_{d}})}\otimes
 \left(\sum_{j=1}^\infty \frac{1}{m_{d+1}+j-1} {\bf P}_{\{\text{\rm span}\,  e^{d+1}_\alpha: \ \alpha\in \FF_{n_{d+1}}^+, |\alpha|=j\}}\right) \otimes I_{F^2(H_{n_{d+2}})}\otimes\cdots\otimes I_{F^2(H_{n_k})}
\end{equation*}
and $q_{s_1,\ldots, s_k}({\bf W}, {\bf W}^*) f\in \cK\otimes \cE_{p_1+s_1,\ldots p_d+s_d< q+1,0\ldots,o}$ we have
$$
{\bf \Lambda}_{d+1} {\bf \Lambda}_{d+1}^*q_{s_1,\ldots, s_k}({\bf W}, {\bf W}^*) f=\frac{q+1+s_{d+1}}{m_{d+1}+q+s_{d+1}}
q_{s_1,\ldots, s_k}({\bf W}, {\bf W}^*) f.
$$
Consequently, relation \eqref{3c} implies
$Af=q_{s_1,\ldots, s_k}({\bf W}, {\bf W}^*) f$ for every $f\in \cK\otimes
\cE_{p_1, \ldots, p_d,q+1, s_{d+2}^-,\ldots,s_k^-}$.
Therefore relation \eqref{M} holds for every
 $f\in \cK\otimes
\cE_{p_1, \ldots, p_d,p_{d+1}, s_{d+2}^-,\ldots,s_k^-}$ and every
$p_1,\ldots, p_{d+1}\in \{0,1,\ldots \}$.
Continuing this process, we conclude that relation \eqref{M} holds.

{\bf Case III:} for every $i\in \{1,\ldots, k\}$, $s_i<0$.

The proof of the theorem is similar to the one above in the case $d<k$. The proof is complete.
\end{proof}

The radial part of ${\bf D_n^m}$ is the noncommutative domain ${\bf D}_{{\bf n},rad}^{\bf m}$ whose representation on any Hilbert space $\cH$ is ${\bf D}_{{\bf n},rad}^{\bf m}(\cH):= \cup_{r\in [0,1)}r{\bf D_n^m}(\cH)$.

\begin{definition}We say that $F$ is a free  $k$-pluriharmonic function on the radial part of  ${\bf D_n^m}$  with coefficients in $B(\cK)$,  if its representation on a Hilbert space $\cH$ has the form
$$
F({\bf X})=\sum_{s_1\in \ZZ}\cdots \sum_{s_k\in \ZZ}
\sum_{{\alpha_i,\beta_i\in \FF_{n_i}^+, i\in \{1,\ldots, k\}}\atop{|\alpha_i|=s_i^+, |\beta_i|=s_i^-}} A_{(\alpha_1,\ldots,\alpha_k,\beta_1,\ldots, \beta_k)}\otimes {\bf X}_{1,\alpha_1}\cdots {\bf X}_{k,\alpha_k}{\bf X}_{1,\beta_1}^*\cdots {\bf X}_{k,\beta_k}^*
$$
for every ${\bf X}\in {\bf D}_{{\bf n},rad}^{\bf m}(\cH)$, where the convergence is in the operator norm topology.
\end{definition}

An application  of the noncommutative  Berezin transforms associated with poly-hyperballs reveals  that $F$ is a free $k$-pluriharmonic function on the radial part of  ${\bf D_n^m}$  with coefficients in $B(\cK)$,  if and only if the series
$$
\sum_{s_1\in \ZZ}\cdots \sum_{s_k\in \ZZ}
\sum_{{\alpha_i,\beta_i\in \FF_{n_i}^+, i\in \{1,\ldots, k\}}\atop{|\alpha_i|=s_i^+, |\beta_i|=s_i^-}} r^{|\alpha_1|+\cdots +|\alpha_k|+|\beta_1|+\cdots+ |\beta_k|}A_{(\alpha_1,\ldots,\alpha_k,\beta_1,\ldots, \beta_k)}\otimes {\bf W}_{1,\alpha_1}\cdots {\bf W}_{k,\alpha_k}{\bf W}_{1,\beta_1}^*\cdots {\bf W}_{k,\beta_k}^*
$$
  convergences  in the operator norm topology for every $r\in [0,1)$.
A free  $k$-pluriharmonic function on the radial part of  ${\bf D_n^m}$  is said to be bounded if
$$\|F\|:=\sup_{{\bf X} \in {\bf D}_{{\bf n},rad}^{\bf m}(\cH)}||F({\bf X} )\|<\infty,
$$
where the supremum is taken over all Hilbert spaces $\cH$.

\begin{lemma}\label{qBH}  If  $(s_1,\ldots, s_k)\in \ZZ^k$, the  operator
$$
 q_{s_1,\ldots, s_k}({\bf W}, {\bf W}^*):=
\sum_{{\alpha_i,\beta_i\in \FF_{n_i}^+, i\in \{1,\ldots, k\}}\atop{|\alpha_i|=s_i^+, |\beta_i|=s_i^-}} A_{(\alpha_1,\ldots,\alpha_k,\beta_1,\ldots, \beta_k)}\otimes {\bf W}_{1,\alpha_1}\cdots {\bf W}_{k,\alpha_k}{\bf W}_{1,\beta_1}^*\cdots {\bf W}_{k,\beta_k}^*,
$$
where  $A_{(\alpha_1,\ldots,\alpha_k,\beta_1,\ldots, \beta_k)}\in B(\cK)$, satisfies the Brown-Halmos condition \eqref{BH}.
\end{lemma}
\begin{proof}  Let $i\in \{1,\ldots, k\}$ and set $G:=A_{(\alpha_1,\ldots,\alpha_k;\beta_1,\ldots, \beta_k)}\otimes {\bf W}_{1,\alpha_1}\cdots {\bf W}_{k,\alpha_k}{\bf W}_{1,\beta_1}^*\cdots {\bf W}_{k,\beta_k}^*$.
If   $s_i\geq 0$, then  $G\widetilde{\bf \Lambda}_i=\widetilde{\bf \Lambda}_i {\bf diag}_{n_i}(G)$  and ${\bf diag}_{n_i}(G){\bf diag}_{n_i}\left({\bf \Psi_{\widetilde {\bf \Lambda}_i}}(I) \right)={\bf diag}_{n_i}\left({\bf \Psi_{\widetilde {\bf \Lambda}_i}}(G) \right)$. Using Proposition \ref{WWW},  we have
\begin{equation*}
\begin{split}
(\widetilde{\bf \Lambda}_i^*\widetilde{\bf \Lambda}_i)^{-1} \widetilde{\bf \Lambda}_i^*
G\widetilde{\bf \Lambda}_i(\widetilde{\bf \Lambda}_i^*\widetilde{\bf \Lambda}_i)^{-1}
&= (\widetilde{\bf \Lambda}_i^*\widetilde{\bf \Lambda}_i)^{-1} \widetilde{\bf \Lambda}_i^*
G(I_\cK\otimes {\bf \Omega}_i)\widetilde{\bf \Lambda}_i\\
&= (\widetilde{\bf \Lambda}_i^*\widetilde{\bf \Lambda}_i)^{-1} \widetilde{\bf \Lambda}_i^*
G\widetilde{\bf \Lambda}_i{\bf diag}_{n_i}\left({\bf \Psi_{\widetilde {\bf \Lambda}_i}}(I) \right)\\
&=
 (\widetilde{\bf \Lambda}_1^*\widetilde{\bf \Lambda}_i)^{-1}\widetilde{\bf \Lambda}_i^* \widetilde{\bf \Lambda}_i
{\bf diag}_{n_i}(G){\bf diag}_{n_i}\left({\bf \Psi_{\widetilde {\bf \Lambda}_i}}(I) \right)\\
&={\bf diag}_{n_i}\left({\bf \Psi_{\widetilde {\bf \Lambda}_i}}(G) \right).
\end{split}
\end{equation*}
If   $s_i< 0$, then  $ \widetilde{\bf \Lambda}_i^*G= {\bf diag}_{n_i}(G)\widetilde{\bf \Lambda}_i ^*$  and ${\bf diag}_{n_i}\left({\bf \Psi_{\widetilde {\bf \Lambda}_i}}(I) \right){\bf diag}_{n_i}(G)={\bf diag}_{n_i}\left({\bf \Psi_{\widetilde {\bf \Lambda}_i}}(G) \right)$.
Using Proposition \ref{WWW},  we have
\begin{equation*}
\begin{split}
(\widetilde{\bf \Lambda}_i^*\widetilde{\bf \Lambda}_i)^{-1} \widetilde{\bf \Lambda}_i^*
G\widetilde{\bf \Lambda}_i(\widetilde{\bf \Lambda}_i^*\widetilde{\bf \Lambda}_i)^{-1}
&=
\widetilde{\bf \Lambda}_i^*(I_\cK\otimes {\bf \Omega}_i)G\widetilde{\bf \Lambda}_i(\widetilde{\bf \Lambda}_i^*\widetilde{\bf \Lambda}_i)^{-1}
\\
&={\bf diag}_{n_i}\left({\bf \Psi_{\widetilde {\bf \Lambda}_i}}(I)\right)\widetilde{\bf \Lambda}_i^* G
\widetilde{\bf \Lambda}_i(\widetilde{\bf \Lambda}_i^*\widetilde{\bf \Lambda}_i)^{-1}\\
&={\bf diag}_{n_i}\left({\bf \Psi_{\widetilde {\bf \Lambda}_i}}(I) \right){\bf diag}_{n_i}(G)\widetilde{\bf \Lambda}_i^*\widetilde{\bf \Lambda}_i(\widetilde{\bf \Lambda}_i^*\widetilde{\bf \Lambda}_i)^{-1}\\
&={\bf diag}_{n_i}\left({\bf \Psi_{\widetilde {\bf \Lambda}_i}}(G) \right).
 \end{split}
\end{equation*}
The proof is complete.
 \end{proof}

Let $\cJ$ be  the set of all tuples $(\boldsymbol \alpha, \boldsymbol \beta):=(\alpha_1,\ldots,\alpha_k,\beta_1,\ldots, \beta_k)$,  where $\alpha_i,\beta_i\in \FF_{n_i}^+$ with $|\alpha_i|=s_i^+$, $ |\beta_i|=s_i^-$, and $s_i\in\ZZ$.  In what follows, we  also use the notation ${\bf W}_{\boldsymbol\alpha}:={\bf W}_{1,\alpha_1}\cdots {\bf W}_{k,\alpha_k}$  whenever $\boldsymbol\alpha:=(\alpha_1,\ldots, \alpha_k)\in {\bf F}_{\bf n}^+:=\FF_{n_1}^+\times\cdots \times \FF_{n_k}^+$.

The main result of this section is the following

\begin{theorem}\label{main}
If $T\in B(\cK\otimes \bigotimes_{s=1}^k  F^2(H_{n_s}))$, then the following statements are equivalent.
\begin{enumerate}
\item[(i)] $T$ satisfies the Brown-Halmos condition \eqref{BH}.
\item[(ii)] There is a unique bounded free  $k$-pluriharmonic  function $F$ on  the radial poly-hyperball ${\bf D}_{{\bf n},rad}^{\bf m}$ with coefficients in $B(\cK)$ such that
    $$T=\text{\rm SOT-}\lim_{r\to 1} F(r{\bf W}).$$
    \item[(iii)]
     $T\in\text{\rm span}\left\{ C\otimes {\bf W}_{\boldsymbol\alpha}{\bf W}_{\boldsymbol\beta}^* :\   C\in B(\cK),   (\boldsymbol \alpha, \boldsymbol \beta)\in \cJ\right\}^{-\text{\rm SOT}}$.
      \item[(iv)]
     $T\in\text{\rm span}\left\{ C\otimes {\bf W}_{\boldsymbol\alpha}{\bf W}_{\boldsymbol\beta}^* :\   C\in B(\cK),   (\boldsymbol \alpha, \boldsymbol \beta)\in \cJ\right\}^{-\text{\rm WOT}}$.
     \end{enumerate}

\end{theorem}
\begin{proof} We prove the implication (i)$\implies$(ii). Assume that $T$ satisfies the Brown-Halmos condition \eqref{BH} and let
   $\{T_{s_1,\ldots, s_k}\}_{(s_1,\ldots, s_k)\in \ZZ^k}$ be the multi-homogeneous parts of $T$. According to Proposition \ref{homo-decomp}, we have
   $$
Tf=\lim_{N_1\to \infty}\ldots \lim_{N_k\to \infty} \sum_{(s_1,\ldots, s_k)\in \ZZ^k, |s_j|\leq N_j} T_{s_1,\ldots, s_k}f
$$
for every $f\in \cK\otimes\cE_{p_1,\ldots, p_k}$  and every $(p_1,\ldots, p_k)\in \ZZ^k$. On the other hand, Theorem
\ref{BH-cond}  shows that $T_{s_1,\ldots, s_k}$ satisfies the Brown-Halmos condition and, due to Theorem
\ref{multi-homo}, we have
$$ T_{s_1,\ldots, s_k}g=q_{s_1,\ldots, s_k}({\bf W}, {\bf W}^*) g,\qquad g\in \cK\otimes \bigotimes_{s=1}^k  F^2(H_{n_s}),
$$
where
 \begin{equation}
 \label{q}
 q_{s_1,\ldots, s_k}({\bf W}, {\bf W}^*):=
\sum_{{\alpha_i,\beta_i\in \FF_{n_i}^+, i\in \{1,\ldots, k\}}\atop{|\alpha_i|=s_i^+, |\beta_i|=s_i^-}} A_{(\alpha_1,\ldots,\alpha_k,\beta_1,\ldots, \beta_k)}\otimes {\bf W}_{1,\alpha_1}\cdots {\bf W}_{k,\alpha_k}{\bf W}_{1,\beta_1}^*\cdots {\bf W}_{k,\beta_k}^*,
\end{equation}
for some operators $A_{(\alpha_1,\ldots,\alpha_k;\beta_1,\ldots, \beta_k)}\in B(\cK)$.
Denote by $\cP_\cK$ the linear span of all vectors of the form
$h\otimes e^1_{\alpha_1}\otimes \cdots\otimes e^k_{\alpha_k}$, where $h\in \cK$, $\alpha_i\in \FF_{n_i}^+$.
Combining the results above, we deduce that
\begin{equation}
\label {Tpp}
Tp=\lim_{N_1\to \infty}\ldots \lim_{N_k\to \infty} \sum_{(s_1,\ldots, s_k)\in \ZZ^k, |s_j|\leq N_j} q_{s_1,\ldots, s_k}({\bf W}, {\bf W}^*)p,\qquad p\in \cP_\cK.
\end{equation}
We remark that, for every $x,y\in  \cK\otimes \bigotimes_{s=1}^k  F^2(H_{n_s})$, we have
\begin{equation*}
\begin{split}
|\left<T_{s_1,\ldots, s_k}x,y\right>|
&\leq
\left(\frac{1}{2\pi}\right)^k
\int_0^{2\pi}\cdots \int_0^{2\pi}  |\left<T
\left(I_\cK\otimes \Gamma (e^{i {\boldsymbol\theta}})^*\right) x, \left(I_\cK\otimes \Gamma (e^{i {\boldsymbol\theta}})^*\right)y\right>|d\theta_1\ldots d\theta_k\\
&\leq \|T\|\|x\|\|y\|,
\end{split}
\end{equation*}
which implies $\|T_{s_1,\ldots, s_k}\|\leq \|T\|$ for every $(s_1,\ldots, s_k)\in \ZZ^k$. As a consequence, we deduce that the series
\begin{equation*}
\begin{split}
\sum_{(s_1,\ldots, s_k)\in \ZZ^k}q_{s_1,\ldots, s_k}(r{\bf W}, r{\bf W}^*)&=
\sum_{(s_1,\ldots, s_k)\in \ZZ^k}r^{|s_1|+\cdots +|s_k|}q_{s_1,\ldots, s_k}({\bf W}, {\bf W}^*)\\
&=\sum_{(s_1,\ldots, s_k)\in \ZZ^k}r^{|s_1|+\cdots +|s_k|} T_{s_1,\ldots, s_k}
\end{split}
\end{equation*}
are convergent in the operator norm topology.
Now, we prove that
\begin{equation}
\label{conv}
\left\|\sum_{(s_1,\ldots, s_k)\in \ZZ^k}q_{s_1,\ldots, s_k}(r{\bf W}, r{\bf W}^*)p-\sum_{(s_1,\ldots, s_k)\in \ZZ^k}q_{s_1,\ldots, s_k}({\bf W}, {\bf W}^*)p\right\|\to 0,\quad \text{as } \ r\to 1,
\end{equation}
for every $p\in \cP_\cK$. It is enough to prove the relation when $p=f\in \cK\otimes \cE_{p_1,\ldots, p_k}$.
To this end,  due to relation \eqref{Tpp}, if  $\epsilon>0$ , then there is a finite set $\Gamma\subset  \ZZ^k$ such that
$\left\|\sum_{(s_1,\ldots, s_k)\in \ZZ^k\backslash \Gamma}q_{s_1,\ldots, s_k}({\bf W}, {\bf W}^*)f \right\|<\epsilon$.
Since the verctors  $\{q_{s_1,\ldots, s_k}({\bf W}, {\bf W}^*)f \}_{(s_1,\ldots, s_k)\in \ZZ^k}$ are pairwise orthogonal
and
$$
\sum_{(s_1,\ldots, s_k)\in \ZZ^k\backslash \Gamma}\left\|q_{s_1,\ldots, s_k}(r{\bf W}, r{\bf W}^*)f \right\|^2
=\sum_{(s_1,\ldots, s_k)\in \ZZ^k\backslash \Gamma}r^{(|s_1|+\cdots +|s_k|)}\left\|q_{s_1,\ldots, s_k}({\bf W}, {\bf W}^*)f \right\|^2\leq \epsilon^2,
$$
we deduce that
\begin{equation*}
\begin{split}
&\left\|\sum_{(s_1,\ldots, s_k)\in \ZZ^k}q_{s_1,\ldots, s_k}(r{\bf W}, r{\bf W}^*)f-\sum_{(s_1,\ldots, s_k)\in \ZZ^k}q_{s_1,\ldots, s_k}({\bf W}, {\bf W}^*)f\right\|\\
&\qquad\qquad \leq
\sum_{(s_1,\ldots, s_k)\in \Gamma}\left\|q_{s_1,\ldots, s_k}(r{\bf W}, r{\bf W}^*)f-q_{s_1,\ldots, s_k}({\bf W}, {\bf W}^*)f\right\|\\
&\qquad \qquad+
\left(\sum_{(s_1,\ldots, s_k)\in \ZZ^k\backslash \Gamma}\left\|q_{s_1,\ldots, s_k}(r{\bf W}, r{\bf W}^*)f \right\|^2\right)^{1/2}
+\left(\sum_{(s_1,\ldots, s_k)\in \ZZ^k\backslash \Gamma}\left\|q_{s_1,\ldots, s_k}({\bf W}, {\bf W}^*)f \right\|^2\right)^{1/2}\\
&\qquad \qquad \leq
\sum_{(s_1,\ldots, s_k)\in \Gamma}\left\|q_{s_1,\ldots, s_k}(r{\bf W}, r{\bf W}^*)f- q_{s_1,\ldots, s_k}({\bf W}, {\bf W}^*)f\right\| + 2\epsilon.
\end{split}
\end{equation*}
Taking $r\to 1$, one can easily  deduce relation \eqref{conv}.
Now we prove that
\begin{equation}
\label{VN}
\left\|\sum_{(s_1,\ldots, s_k)\in \ZZ^k}q_{s_1,\ldots, s_k}(r{\bf W}, r{\bf W}^*)\right\|\leq \|T\|,\qquad r\in[0,1).
\end{equation}

We recall that  the noncommutative Berezin kernel associated with $r{\bf W}\in {\bf D_n^m}(\otimes _{i=1}^k F^2(H_{n_i}))$ is defined on
$\otimes_{i=1}^kF^2(H_{n_{i}})$ with values in $ \otimes_{i=1}^kF^2(H_{n_{i}})\otimes \cD_{r{\bf W}}\subset \left(\otimes_{i=1}^kF^2(H_{n_{i}})\right)\otimes \left(\otimes_{i=1}^kF^2(H_{n_{i}})\right)$, where $\cD_{r{\bf W}}:=\overline{\boldsymbol{\Delta}^{\bf m}_{r{\bf W}}(I)(\otimes_{i=1}^kF^2(H_{n_i}))}$.
Let $\boldsymbol\gamma=(\gamma_1,\ldots, \gamma_k)$ and $\boldsymbol \omega=(\omega_1,\ldots, \omega_k)$ be  $k$-tuples in ${\bf F}_{\bf n}^+$, set $q:=\max\{|\gamma_1|, \ldots |\gamma_k|, |\omega_1|,\ldots,|\omega_k|\}$, and define the operator
$$
T_q:=\sum_{s_1\in \ZZ, |s_1|\leq q}\cdots\sum_{s_k\in \ZZ, |s_k|\leq q}\sum_{{\alpha_i,\beta_i\in \FF_{n_k}^+, i\in \{1,\ldots, k\}}\atop{|\alpha_i|=s_i^+, |\beta_i|=s_i^-}} A_{(\alpha_1,\ldots, \alpha_k, \beta_1,\ldots, \beta_k)}\otimes {\bf W}_{\boldsymbol\alpha} {\bf W}_{\boldsymbol\beta}^*,
$$
where we use the notation ${\bf W}_{\boldsymbol\alpha}:={\bf W}_{1,\alpha_1}\cdots {\bf W}_{k,\alpha_k}$ if $\boldsymbol\alpha:=(\alpha_1,\ldots, \alpha_k)\in{\bf F_n}^+:= \FF_{n_1}^+\times\cdots \times \FF_{n_k}^+$. We also set $e_{\boldsymbol\alpha}:=e_{\alpha_1}^1\otimes \cdots \otimes e_{\alpha_k}^k$ and $b_{\boldsymbol\alpha}^{({\bf m})}:=\sqrt{b_{1,\alpha_1}^{(m_1)}}\cdots \sqrt{b_{k,\alpha_k}^{(m_k)}}$.

Note that
\begin{equation*}
\begin{split}
&\left<(I_\cK\otimes {\bf K}_{r{\bf W}}^*) (T\otimes I_{\otimes_{i=1}^kF^2(H_{n_{i}})})(I_\cK\otimes {\bf K}_{r{\bf W}})(h\otimes e_{\boldsymbol\gamma}), h'\otimes e_{\boldsymbol\omega}\right>\\
&=\left<(T\otimes I_{\otimes_{i=1}^kF^2(H_{n_{i}})})
\sum_{\boldsymbol\alpha\in {\bf F}_{\bf n}^+}h\otimes b_{\boldsymbol\alpha}^{({\bf m})}e_{\boldsymbol\alpha}\otimes \boldsymbol\Delta_{r{\bf W}}(I)^{1/2}{\bf W}_{\boldsymbol\alpha}^*(e_{\boldsymbol\gamma})\right.,\\
&\qquad\qquad \left.
\sum_{\boldsymbol\beta\in {\bf F}_{\bf n}^+}h'\otimes b_{\boldsymbol\beta}^{({\bf m})}e_{\boldsymbol\beta}\otimes \boldsymbol\Delta_{r{\bf W}}(I)^{1/2}{\bf W}_{\boldsymbol\beta}^*(e_{\boldsymbol\omega})\right>
\\
&=
\sum_{\boldsymbol\alpha\in {\bf F}_{\bf n}^+}
\sum_{\boldsymbol\beta\in {\bf F}_{\bf n}^+}
\left<
T(h\otimes b_{\boldsymbol\alpha}^{({\bf m})}e_{\boldsymbol\alpha})\otimes \boldsymbol\Delta_{r{\bf W}}(I)^{1/2}{\bf W}_{\boldsymbol\alpha}^*(e_{\boldsymbol\gamma}),
h'\otimes b_{\boldsymbol\beta}^{({\bf m})}e_{\boldsymbol\beta}\otimes \boldsymbol\Delta_{r{\bf W}}(I)^{1/2}{\bf W}_{\boldsymbol\beta}^*(e_{\boldsymbol\omega})\right>
\\
&=
\sum_{\boldsymbol\alpha\in {\bf F}_{\bf n}^+}
\sum_{\boldsymbol\beta\in {\bf F}_{\bf n}^+}
\left<
T(h\otimes b_{\boldsymbol\alpha}^{({\bf m})}e_{\boldsymbol\alpha}), h'\otimes b_{\boldsymbol\beta}^{({\bf m})}e_{\boldsymbol\beta}\right>
\left<\boldsymbol\Delta_{r{\bf W}}(I)^{1/2}{\bf W}_{\boldsymbol\alpha}^*(e_{\boldsymbol\gamma}),
 \boldsymbol\Delta_{r{\bf W}}(I)^{1/2}{\bf W}_{\boldsymbol\beta}^*(e_{\boldsymbol\omega})\right> \\
 &=
 \sum_{s_1\in \ZZ, |s_1|\leq q}\cdots\sum_{s_k\in \ZZ, |s_k|\leq q}\sum_{{\alpha_i,\beta_i\in \FF_{n_k}^+, i\in \{1,\ldots, k\}}\atop{|\alpha_i|=s_i^+, |\beta_i|=s_i^-}}
 \left<
T_q(h\otimes b_{\boldsymbol\alpha}^{({\bf m})}e_{\boldsymbol\alpha}), h'\otimes b_{\boldsymbol\beta}^{({\bf m})}e_{\boldsymbol\beta}\right>\\
&\qquad \qquad \times
\left<\boldsymbol\Delta_{r{\bf W}}(I)^{1/2}{\bf W}_{\boldsymbol\alpha}^*(e_{\boldsymbol\gamma}),
 \boldsymbol\Delta_{r{\bf W}}(I)^{1/2}{\bf W}_{\boldsymbol\beta}^*(e_{\boldsymbol\omega})\right>
 \\
&=
 \sum_{\boldsymbol\alpha\in {\bf F}_{\bf n}^+}
\sum_{\boldsymbol\beta\in {\bf F}_{\bf n}^+}
\left<
T_q(h\otimes b_{\boldsymbol\alpha}^{({\bf m})}e_{\boldsymbol\alpha}), h'\otimes b_{\boldsymbol\beta}^{({\bf m})}e_{\boldsymbol\beta}\right>
\left<\boldsymbol\Delta_{r{\bf W}}(I)^{1/2}{\bf W}_{\boldsymbol\alpha}^*(e_{\boldsymbol\gamma}),
 \boldsymbol\Delta_{r{\bf W}}(I)^{1/2}{\bf W}_{\boldsymbol\beta}^*(e_{\boldsymbol\omega})\right> \\
 &=
 \left<(I_{\cK}\otimes {\bf K}_{r{\bf W}}^*)(T_q\otimes I_{\otimes_{i=1}^kF^2(H_{n_{i}})})
(I_{\cE}\otimes {\bf K}_{r{\bf W}})(h\otimes e_{\boldsymbol\gamma}), h'\otimes e_{\boldsymbol\omega}\right>\\
 &=
\sum_{s_1\in \ZZ, |m_1|\leq q}\cdots\sum_{s_k\in \ZZ, |m_k|\leq q}\sum_{{\alpha_i,\beta_i\in \FF_{n_k}^+, i\in \{1,\ldots, k\}}\atop{|\alpha_i|=s_i^+, |\beta_i|=s_i^-}}\\
&\qquad \qquad
\left< \left(A_{(\alpha_1,\ldots, \alpha_k, \beta_1,\ldots, \beta_k)}\otimes r^{\sum_{i=1}^k(|\alpha_i|+|\beta_i|)} {\bf W}_{\boldsymbol\alpha} {\bf W}_{\boldsymbol\beta}^*\right)(h\otimes e_{\gamma}), h'\otimes e_{\boldsymbol\omega}\right>\\
&=\left<\sum_{(s_1,\ldots, s_k)\in \ZZ^k}q_{s_1,\ldots, s_k}(r{\bf W}, r{\bf W}^*)(h\otimes e_{\boldsymbol\gamma}), h'\otimes e_{\boldsymbol\omega}\right>.
\end{split}
\end{equation*}
Hence, we deduce that
\begin{equation}
\label{KK}
(I_\cK\otimes {\bf K}_{r{\bf W}}^*) (T\otimes I_{\otimes_{i=1}^kF^2(H_{n_{i}})})(I_\cK\otimes {\bf K}_{r{\bf W}})
=\sum_{(s_1,\ldots, s_k)\in \ZZ^k}q_{s_1,\ldots, s_k}(r{\bf W}, r{\bf W}^*)
\end{equation}
for every $r\in [0,1)$. Since $ {\bf K}_{r{\bf W}}$ is an isometry, we obtain inequality \eqref{VN}.

Now, we prove that
\begin{equation}\label{SOT}
\text{\rm SOT-}\lim_{r \to 1}\sum_{(s_1,\ldots, s_k)\in \ZZ^k}q_{s_1,\ldots, s_k}(r{\bf W}, r{\bf W}^*)=T.
\end{equation}
Indeed, let  $\epsilon>0$  and $\xi\in \cK\otimes \bigotimes_{s=1}^k  F^2(H_{n_s})$. Then there is $p\in \cP_\cK$ such that $\|\xi-p\|<\epsilon$ and, due to relations \eqref{Tpp} and \eqref{VN}, we deduce that
\begin{equation*}
\begin{split}
&\left\|\sum_{(s_1,\ldots, s_k)\in \ZZ^k}q_{s_1,\ldots, s_k}(r{\bf W}, r{\bf W}^*)\xi-T\xi\right\|
\leq
\left\|\sum_{(s_1,\ldots, s_k)\in \ZZ^k}q_{s_1,\ldots, s_k}(r{\bf W}, r{\bf W}^*)(\xi-p)\right\|\\
&\qquad +\left\|\sum_{(s_1,\ldots, s_k)\in \ZZ^k}q_{s_1,\ldots, s_k}(r{\bf W}, r{\bf W}^*)p-\sum_{(s_1,\ldots, s_k)\in \ZZ^k}q_{s_1,\ldots, s_k}({\bf W}, {\bf W}^*)p\right\|\\
&\qquad \qquad
+
\left\|\sum_{(s_1,\ldots, s_k)\in \ZZ^k}q_{s_1,\ldots, s_k}({\bf W}, {\bf W}^*)p-T\xi\right\|
\\
&\leq 2\|T\|\|\xi-p\| +\left\|\sum_{(s_1,\ldots, s_k)\in \ZZ^k}q_{s_1,\ldots, s_k}(r{\bf W}, r{\bf W}^*)p-\sum_{(s_1,\ldots, s_k)\in \ZZ^k}q_{s_1,\ldots, s_k}({\bf W}, {\bf W}^*)p\right\|.
\end{split}
\end{equation*}
Consequently, using relation \eqref{conv}, one can easily see that relation \eqref{SOT} holds.
Note that $F({\bf X}):=\sum_{(s_1,\ldots, s_k)\in \ZZ^k}q_{s_1,\ldots, s_k}({\bf X}, {\bf X}^*) $ is a free  $k$-pluriharmonic function on the radial part of ${\bf D_n^m}$.
This completes the proof of the implication (i)$\implies$(ii). Now, we prove the implication (ii)$\implies$(i).

Assume that there is a free $k$-pluriharmonic  function $F$ on  the radial poly-hyperball ${\bf D}_{{\bf n},rad}^{\bf m}$ with coefficients in $B(\cK)$ such that
    $$T=\text{\rm SOT-}\lim_{r\to 1} F(r{\bf W}).$$
    Consequently,
    $$
    F(r{\bf W})=\sum_{(s_1,\ldots, s_k)\in \ZZ^k}q_{s_1,\ldots, s_k}(r{\bf W}, r{\bf W}^*),
    $$
    where the convergence of the series is in the operator norm topology and $q_{s_1,\ldots, s_k}(r{\bf W}, r{\bf W})^*$ has the form described by relation \eqref{q}.
Due to Lemma \ref{qBH}, relation
\begin{equation}
\label{BH-X}(\widetilde{\bf \Lambda}_i^*\widetilde{\bf \Lambda}_i)^{-1} \widetilde{\bf \Lambda}_i^*
X\widetilde{\bf \Lambda}_i(\widetilde{\bf \Lambda}_i^*\widetilde{\bf \Lambda}_i)^{-1}
={\bf diag}_{n_i}\left({\bf \Psi_{\widetilde {\bf \Lambda}_i}}(X) \right)
\end{equation}
is satisfied  when $X=q_{s_1,\ldots, s_k}({\bf W}, {\bf W}^*)$. Hence, we deduce that the same relation holds when $X=F(r{\bf W})$. Taking the SOT-limit, as $r\to 1$ in the resulting relation  and using the fact that  $T=\text{\rm SOT-}\lim_{r\to 1} F(r{\bf W})$ we conclude that  that $T$ satisfies the Brown-Halmos condition. Therefore condition (i) holds. Since the the implications (ii)$\implies $(iii) and (iii)$\implies $(i) are obvious, it remains to prove the implication (iv)$\implies $(i).

To this end, assume that condition (iv) holds. Due to Lemma \ref{qBH}, for every  $ C\in B(\cK)$ and $   (\boldsymbol \alpha, \boldsymbol \beta)\in \cJ$,   the operator $ X=C\otimes {\bf W}_{\boldsymbol\alpha}{\bf W}_{\boldsymbol\beta}^*$ satisfies the relation \eqref{BH-X}.
 Taking linear combinations of this type of operators and then WOT-limits, one can easily see that  the relation \eqref{BH-X} holds
when $X=T$.
The proof is complete.
\end{proof}

\begin{corollary} \label{norm} Under the conditions of Theorem \ref{main}, we have
$$
\|T\|=\sup_{r\in [0,1)}\|F(r{\bf W})\|=\lim_{r\to 1}\|F(r{\bf W})\|=\sup_{p\in \cP_\cK, \|p\|\leq 1}\|F({\bf W})p\|.
$$
\end{corollary}
\begin{proof}
According to the proof of Theorem \ref{main}, we have
$\sup_{r\in [0,1)}\|F(r{\bf W})\|\leq\|T\|$ and also  $T=\text{\rm SOT}-\lim_{r\to 1} F(r{\bf W})$. Hence, we deduce  that
$\|T\|=\sup_{r\in [0,1)}\|F(r{\bf W})\|$. If $r_1,r_2\in [0,1)$ with $r_1<r_2$, then
$$
F(r_1{\bf W})=(I_\cK\otimes {\bf K}_{{\frac{r_1}{r_2}}{\bf W}}^*) \left(F(r_2{\bf W})\otimes I_{\otimes_{i=1}^kF^2(H_{n_{i}})}\right)(I_\cK\otimes {\bf K}_{{\frac{r_1}{r_2}}{\bf W}}).
$$
Hence, we deduce that $\|F(r_1{\bf W})\| \leq \|F(r_2{\bf W})\|$, which implies
$\sup_{r\in [0,1)}\|F(r{\bf W})\|=\lim_{r\to 1}\|F(r{\bf W})\|$.
On the other hand, since $Tp=F({\bf W})p$ for every $p\in \cP_\cK$, it is clear  that $\|T\|=\sup_{p\in \cP_\cK, \|p\|\leq 1}\|F({\bf W})p\|$.
This completes the proof.
\end{proof}

\bigskip

\section{Brown-Halmos type characterization of weighted multi-Toeplitz operators}

In this section, we introduce the weighted multi-Toeplitz operators which are associated with the poly-hyperball
${\bf D_{n}^m}$ and   show that  they are precisely those satisfying the Brown-Halmos equations.
We also prove that each  weighted multi-Toeplitz operator   has a unique  formal Fourier representation
which can be used to recover the operator. Conversely, given a formal series, we provide necessary and sufficient conditions on it to be  the formal Fourier representation  of a weighted multi-Toeplitz operator.

If $\omega, \gamma\in \FF_n^+$,
we say that $\omega
\geq_{r}\gamma$ if there is $\sigma\in
\FF_n^+ $ such that $\omega=\sigma \gamma$. In this
case  we set $\omega\backslash_r \gamma:=\sigma$. If $\sigma\neq g_0$ we write $\omega>_r \gamma$. We say that $\omega$ and $\gamma$ are {\it comparable} if either $\omega
\geq_{r}\gamma$ or $\gamma>_r\omega$.
Let $\boldsymbol\omega=(\omega_1,\ldots, \omega_k)$ and $\boldsymbol\gamma=(\gamma_1,\ldots, \gamma_k)$ be in ${\bf F}_{\bf n }^+:=\FF_{n_1}^+\times\cdots \times \FF_{n_k}^+$. We say that $\boldsymbol\omega$ and $\boldsymbol\gamma$ are comparable, and write  $\boldsymbol\omega\sim_c \boldsymbol\gamma$. if, for each $i\in \{1,\ldots, k\}$,  either one of the relations  $\omega_i<_r \gamma_i$, $\gamma_i<_r \omega_i$, or $\omega_i=\gamma_i$ holds.

We denote by $\cC$ the set of all pairs $(\boldsymbol \sigma, \boldsymbol \beta)\in {\bf F}_{\bf n}^+\times {\bf F}_{\bf n}^+$ which are comparable, and  note that
 $\cJ$   is the subset  of $\cC$ of all pairs $(\boldsymbol \sigma, \boldsymbol \beta):=(\sigma_1,\ldots,\sigma_k,\beta_1,\ldots, \beta_k)$,  where   $\sigma_i,\beta_i\in \FF_{n_i}^+$ with $|\sigma_i|=s_i^+$, $ |\beta_i|=s_i^-$, and $s_i\in\ZZ$.
We introduce the {\it simplification function}  ${\bf s}:\cC\to \cJ$ defined by
${\bf s}(\boldsymbol \omega, \boldsymbol\gamma):= (\boldsymbol \sigma,\boldsymbol \beta)$, where,   if  $\boldsymbol\omega=(\omega_1,\ldots, \omega_k)$ and $\boldsymbol\gamma=(\gamma_1,\ldots, \gamma_k)$, then
for every $i\in \{1,\ldots, k\}$,
$$
\sigma_i:=
\begin{cases} \omega_i\backslash_r\gamma_i, & \text{ if } \omega_i\geq_r \gamma_i,\\
g_0^i, &\text{otherwise,}
\end{cases}
  \quad
\beta_i:=\begin{cases} \gamma_i\backslash_r\omega_i, & \text{ if } \gamma_i\geq_r \omega_i,\\
g_0^i, &\text{otherwise}.
\end{cases}
$$

Brown and Halmos \cite{BH} proved that a necessary and sufficient condition that an operator on the Hardy space $H^2(\TT)$ be a Toeplitz operator is that its matrix
 $[\lambda_{ij}]$ with respect to the standard basis $\chi_k(e^{i\theta})=e^{ik\theta}$, $k\in \{0,1,\ldots\}$, be a Toeplitz  matrix, i.e
 $$
 \lambda_{i+1,j+1}=\lambda_{ij},\qquad  i,j\in \{0,1,\ldots\},
 $$
 which is equivalent to   the fact that
 $\lambda_{ij}=a_{i-j}$, where $\varphi =\sum_{k\in \ZZ}a_k \chi_k$ is the Fourier expansion of the symbol $\varphi\in L^\infty(\TT)$.  In what follows, we find an  extension of their result to our noncommutative  multivariable setting.

\begin{definition} \label{MT}  An operator $T\in B(\cK\otimes \bigotimes_{s=1}^k  F^2(H_{n_s}))$ is called  weighted (right) multi-Toeplitz if the exist operators $\{A_{(\boldsymbol \sigma; \boldsymbol \beta)}\}_{(\boldsymbol \sigma; \boldsymbol \beta)\in \cJ}\subset B(\cK)$ such that,
for every $ \boldsymbol\omega, \boldsymbol\gamma \in {\bf F}_{\bf n}^+$ and $x,y\in \cK$,
$$
\left<T(x\otimes e_{\boldsymbol\gamma }), y\otimes e_{\boldsymbol\omega} \right>
=\begin{cases}
\tau_{(\boldsymbol\omega,\boldsymbol\gamma)}\left<A_{{\bf s}(\boldsymbol \omega, \boldsymbol\gamma)}x,y\right>, & \text{ if  } (\boldsymbol \omega, \boldsymbol\gamma)\in \cC,\\
0,&  \text{ if  } (\boldsymbol \omega, \boldsymbol\gamma)\in ({\bf F}_{\bf n }^+\times {\bf F}_{\bf n }^+)\backslash \cC,

\end{cases}
$$
where the weights $ \{\tau_{(\boldsymbol\omega,\boldsymbol\gamma)}\}_{(\boldsymbol \omega, \boldsymbol\gamma)\in \cC}$ are given by
$$ \tau_{(\boldsymbol\omega,\boldsymbol\gamma)}:=
 \prod_{i=1}^k \sqrt{\frac{b^{(m_i)}_{i,\min\{\omega_i, \gamma_i\}}}{b^{(m_i)}_{i,\max\{\omega_i, \gamma_i\}}}}
 $$
 and the coefficients $b_{i,\beta_i}^{(m_i)}$ by  relation \eqref{b-al2}.
  \end{definition}

  \begin{proposition}\label{Toep-def2}
  An operator $T\in B(\cK\otimes \bigotimes_{s=1}^k  F^2(H_{n_s}))$ is   weighted multi-Toeplitz  if and only if,
  for every $ \boldsymbol\omega, \boldsymbol\gamma \in  {\bf F}_{n}^+$,
\begin{equation}
\label{MT2}
\left<T(x\otimes e_{\boldsymbol\gamma }), y\otimes e_{\boldsymbol\omega} \right>
=\begin{cases}
\frac{\tau_{(\boldsymbol\omega,\boldsymbol\gamma)}}{\tau_{(\boldsymbol \omega', \boldsymbol\gamma')}}\left<T(x\otimes e_{\boldsymbol\gamma'}), y\otimes e_{\boldsymbol\omega'} \right>, & \text{ if  } (\boldsymbol \omega, \boldsymbol\gamma)\in \cC,\\
0,&  \text{ if  } (\boldsymbol \omega, \boldsymbol\gamma)\in ({\bf F}_{\bf n }^+\times {\bf F}_{\bf n }^+)\backslash \cC.
\end{cases}
\end{equation}
where  $(\boldsymbol \omega', \boldsymbol\gamma'):={\bf s}(\boldsymbol \omega, \boldsymbol\gamma)$  when $(\boldsymbol \omega, \boldsymbol\gamma)\in \cC$
  \end{proposition}
  \begin{proof}
  Assume that $T$ is a weighted multi-Toeplitz operator. Note that if
  $(\boldsymbol \omega', \boldsymbol\gamma')\in \cJ$, then ${\bf s}(\boldsymbol \omega', \boldsymbol\gamma')=(\boldsymbol \omega', \boldsymbol\gamma')$ and
  $\tau_{(\boldsymbol \omega', \boldsymbol\gamma')}=\prod_{i=1}^k \frac{1}{\sqrt{b_{i,\max\{\boldsymbol \omega_i', \boldsymbol\gamma_i'\}}}}$.
  Consequently, Definition \ref{MT} implies
  $$
  \left<T(x\otimes e_{\boldsymbol\gamma' }), y\otimes e_{\boldsymbol\omega'} \right>
= \tau_{(\boldsymbol\omega',\boldsymbol\gamma')}\left<A_{(\boldsymbol \omega', \boldsymbol\gamma')}x,y\right>,\qquad x,y\in \cK.
$$
 Let $ (\boldsymbol\omega, \boldsymbol\gamma) \in \cC$ and set  $(\boldsymbol \omega', \boldsymbol\gamma'):={\bf s}(\boldsymbol \omega, \boldsymbol\gamma)$. The relation above implies
 $$
 \left<A_{{\bf s}(\boldsymbol \omega, \boldsymbol\gamma)}x,y\right>=\frac{1}{\tau_{(\boldsymbol \omega', \boldsymbol\gamma')}}\left<T(x\otimes e_{\boldsymbol\gamma' }), y\otimes e_{\boldsymbol\omega'} \right>.
 $$
 Combining this relation with the one in Definition  \ref{MT}, we deduce  relation  \eqref{MT2}.
 Conversely, assume that relation \eqref{MT2} holds.
 For everyfor every  $(\boldsymbol \omega', \boldsymbol\gamma')\in \cJ$, we define the operator
 $A_{(\boldsymbol \omega', \boldsymbol\gamma')}\in B(\cK)$ by setting
 $$
 \left< A_{(\boldsymbol \omega', \boldsymbol\gamma')}x,y\right>:=\frac{1}{\tau_{(\boldsymbol \omega', \boldsymbol\gamma')}}\left<T(x\otimes e_{\boldsymbol\gamma' }), y\otimes e_{\boldsymbol\omega'} \right>,\qquad x,y\in \cK.
 $$
 Consequently,  since ${\bf s}(\boldsymbol \omega, \boldsymbol\gamma)\in \cJ$  when $(\boldsymbol \omega, \boldsymbol\gamma)\in \cC$, we can use  the latter  relation when $(\boldsymbol \omega', \boldsymbol\gamma'):={\bf s}(\boldsymbol \omega, \boldsymbol\gamma)$ and  relation
  \eqref{MT2}, to deduce that $T$ is a weighted multi-Toeplitz operator.
  The proof is complete.
  \end{proof}

  We remark that when $k=1, n_1=1, m_1=1$, and $\cE=\CC$ we recover   the classical Toeplitz operators on the
Hardy space $H^2(\DD)$. Also if $k=1$,  $n_1\geq 2$, and $m_1=1$, we obtain the unweighted  multi-Toeplitz operators on the full Fock space $F^2(H_{n_1})$ (see \cite{Po-multi}, \cite{Po-analytic} and \cite{Po-pluriharmonic}). On the other  hand, if $k\geq 2$, $n_i=m_i=1$ for $i\in \{1,\ldots, k\}$, then $T$ is a Toeplitz operator on the Hardy space $H^2(\DD^k)$.

\begin{theorem}\label{compact} If $T\in B(\otimes_{s=1}^k  F^2(H_{n_s}))$ is a compact weighted multi-Toeplitz operator, then $T=0$.
\end{theorem}
\begin{proof} Fix an arbitrary  pair $(\boldsymbol \omega', \boldsymbol\gamma')\in \cJ$ and let
$(\boldsymbol \omega, \boldsymbol\gamma)\in \cC$ be such that
${\bf s}(\boldsymbol \omega, \boldsymbol\gamma)=(\boldsymbol \omega', \boldsymbol\gamma')$.
According to Proposition \ref{Toep-def2},  we have
\begin{equation}\label{T}
\left<Te_{\boldsymbol\gamma },  e_{\boldsymbol\omega} \right>=
\frac{\tau_{(\boldsymbol\omega,\boldsymbol\gamma)}}{\tau_{(\boldsymbol \omega', \boldsymbol\gamma')}}\left<Te_{\boldsymbol\gamma'}, e_{\boldsymbol\omega'} \right>.
\end{equation}
Note that, for every $i\in \{1,\ldots, k\}$, $j\in \{1,\ldots, n_i\}$, if $\alpha\in \FF_{n_i}^+$, then
\begin{equation}
\label{bfrac}
\frac{b_{i,g^j_i\alpha}^{(m_i)}}{b_{i,\alpha}}=
\frac{\left(\begin{matrix}  |\alpha|+m_i\\m_i-1
\end{matrix}\right)}{\left(\begin{matrix}  |\alpha|+m_i-1\\m_i-1
\end{matrix}\right)}\to 1,\qquad  \text{as }\ |\alpha|\to \infty.
\end{equation}
Consequently,  for every $\sigma\in \FF_{n_i}^+$, we also have  $\frac{b_{i, \sigma\alpha}^{(m_i)}}{b_{i,\alpha}}\to 1$  as $\ |\alpha|\to \infty$.
On the other hand, we have
$$
\frac{\tau_{(\boldsymbol\omega,\boldsymbol\gamma)}}{\tau_{(\boldsymbol \omega', \boldsymbol\gamma')}}
= \left(\prod_{i=1}^k \sqrt{\frac{b^{(m_i)}_{i,\min\{\omega_i, \gamma_i\}}}{b^{(m_i)}_{i,\max\{\omega_i, \gamma_i\}}}}\right)
\left(\prod_{i=1}^k {\sqrt{b^{(m_i)}_{i,\max\{\boldsymbol \omega_i', \boldsymbol\gamma_i'\}}}}\right).
$$
Due to relation \eqref{bfrac}, we deduce that
$$
\lim_{|\min \{\omega_i, \gamma_i\}|\to \infty} \frac{b^{(m_i)}_{i,\min\{\omega_i, \gamma_i\}}}{b^{(m_i)}_{i,\max\{\omega_i, \gamma_i\}}}=1,
$$
which implies
\begin{equation}
\label{TT}
\lim_{|\min \{\omega_i, \gamma_i\}|\to \infty} \frac{\tau_{(\boldsymbol\omega,\boldsymbol\gamma)}}{\tau_{(\boldsymbol \omega', \boldsymbol\gamma')}}=\prod_{i=1}^k {\sqrt{b^{(m_i)}_{i,\max\{\boldsymbol \omega_i', \boldsymbol\gamma_i'\}}}}\neq 0.
\end{equation}
Now, note that $e_{\boldsymbol\gamma}\to 0$ weakly as $ |\min \{\omega_i, \gamma_i\}|\to \infty$ for each $i\in \{1,\ldots, k\}$.  If $T$ is compact operator, then $Te_{\boldsymbol\gamma}\to 0$ in norm. Using relations
\eqref{T}  and \eqref{TT}, we deduce that
$\left<Te_{\boldsymbol\gamma'}, e_{\boldsymbol\omega'} \right>=0$. Now, using again relation \eqref{T}, we deduce that $\left<Te_{\boldsymbol\gamma },  e_{\boldsymbol\omega} \right>=0$
for every $(\boldsymbol \omega, \boldsymbol\gamma)\in \cC$ such that
${\bf s}(\boldsymbol \omega, \boldsymbol\gamma)=(\boldsymbol \omega', \boldsymbol\gamma')$.
Taking into account  Proposition \ref{Toep-def2}, we conclude that $T=0$.
The proof is complete.
\end{proof}

 \begin{lemma}
 \label{monom1}
  Let $ (\boldsymbol \alpha,\boldsymbol \beta)\in \cJ$ and
  $$
  q_ {(\boldsymbol \alpha, \boldsymbol \beta)}({\bf W}, {\bf W}^*):={\bf W}_{1,\alpha_1}\cdots {\bf W}_{k,\alpha_k}{\bf W}_{1,\beta_1}^*\cdots {\bf W}_{k,\beta_k}^*.
  $$
  The following statements hold.
  \begin{enumerate}
  \item[(i)]
  If $\boldsymbol\gamma \in {\bf F}_{\bf n }^+$, then the family $\{q_ {(\boldsymbol \alpha, \boldsymbol \beta)}({\bf W}, {\bf W}^*)e_{\boldsymbol \gamma}\}_{(\boldsymbol \alpha; \boldsymbol \beta)\in \cJ}$ consists of pairwise orthogonal vectors.
  \item[(ii)] If  $\boldsymbol\omega, \boldsymbol\gamma \in {\bf F}_{\bf n }^+$, then
  $$
  \left<q_ {(\boldsymbol \alpha, \boldsymbol \beta)}({\bf W}, {\bf W}^*)e_{\boldsymbol \gamma}, e_{\boldsymbol\omega}\right>\neq 0
  $$
  if and only if  $(\boldsymbol \omega, \boldsymbol\gamma)\in \cC$ and
  ${\bf s}(\boldsymbol \omega, \boldsymbol\gamma)=(\boldsymbol\alpha,\boldsymbol\beta)$.
  \end{enumerate}
  \end{lemma}
\begin{proof} Recall that $\cJ$   the subset  of $\cC$ of all pairs $(\boldsymbol \sigma, \boldsymbol \beta):=(\sigma_1,\ldots,\sigma_k,\beta_1,\ldots, \beta_k)$,  where   $\sigma_i,\beta_i\in \FF_{n_i}^+$ with $|\sigma_i|=s_i^+$, $ |\beta_i|=s_i^-$, and $s_i\in\ZZ$.  Each pair $(\boldsymbol \sigma, \boldsymbol \beta)\in \cJ$ correspondds to a $k$-tuple $(s_1,\ldots, s_k)\in \ZZ^k$.  Note that
$q_ {(\boldsymbol \alpha, \boldsymbol \beta)}({\bf W}, {\bf W}^*)$ is a multi-homogeneous operator of degree $(s_1,\ldots, s_k)\in \ZZ^k$ , i.e.
$$
q_ {(\boldsymbol \alpha, \boldsymbol \beta)}({\bf W}, {\bf W}^*)\left( \cE_{p_1,\ldots, p_k}\right)\subset   \cE_{s_1+p_1,\ldots, s_k+p_k}
$$
 for every $(p_1,\ldots, p_k)\in \ZZ^k$. Let  $(p_1',\ldots, p_k')\in \ZZ^k$ be such that $e_{\boldsymbol \gamma}\in\cE_{p_1',\ldots, p_k'}$.
 Since the subspaces  $\{\cE_{s_1+p_1',\ldots, s_k+p_k'}\}_{(s_1,\ldots, s_k)\in \ZZ^k}$ are pairwise orthogonal, we deduce item (i).

To prove item (ii), let $\boldsymbol \alpha=(\alpha_1,\ldots, \alpha_k)$, $\boldsymbol \beta=(\beta_1,\ldots, \beta_k)$, $\boldsymbol \omega=(\omega_1,\ldots, \omega_k)$, and $\boldsymbol \gamma=(\gamma_1,\ldots, \gamma_k)$. Since  $ (\boldsymbol \alpha,\boldsymbol \beta)\in \cJ$, we have
${\bf W}_{i,\alpha_i} {\bf W}_{j,\beta_j}^*={\bf W}_{j,\beta_j}^*{\bf W}_{i,\alpha_i} $ for every $i,j\in \{1,\ldots, k\}$ and $\alpha_i\in \FF_{n_i}^+$, $\beta_j\in\FF_{n_j}^+$.
Consequently, using the definition of the universal model ${\bf W}$, we deduce that
\begin{equation*}
\begin{split}
\left<{\bf W}_{\boldsymbol\alpha}{\bf W}_{\boldsymbol\beta}^*e_{\boldsymbol\gamma},e_{\boldsymbol\omega}\right>&=\left<{\bf W}_{1,\alpha_1}\cdots {\bf W}_{k,\alpha_k}{\bf W}_{1,\beta_1}^*\cdots {\bf W}_{k,\beta_k}^*
(e_{\gamma_1}^1\otimes \cdots \otimes e_{\gamma_k}^k), e_{\omega_1}^1\otimes \cdots \otimes e_{\omega_k}^k
\right>\\
&=\prod_{i=1}^k \left< {\bf W}_{i,\alpha_i}e^i_{\gamma_i}, {\bf W}_{i,\beta_i}e^i_{\omega_i}
\right>\\
&=
\prod_{i=1}^k  \left(\sqrt{\frac{b^{(m_i)}_{i,\min\{\omega_i, \gamma_i\}}}{b^{(m_i)}_{i,\max\{\omega_i, \gamma_i\}}}}
\left< e^i_{\alpha_i\gamma_i}, e^i_{\beta_i\omega_i}\right>\right).
\end{split}
\end{equation*}
Hence,  $
  \left<q_ {(\boldsymbol \alpha, \boldsymbol \beta)}({\bf W}, {\bf W}^*)e_{\boldsymbol \gamma}, e_{\boldsymbol\omega}\right>\neq 0
  $
  if and only if
  \begin{equation}
  \label{agbo}
  \alpha_i\gamma_i=\beta_i \omega_i \qquad \text{  for every  }  i\in \{1,\ldots, k\}.
  \end{equation}
 Since  $ (\boldsymbol \alpha,\boldsymbol \beta)\in \cJ$, we have $|\alpha_i|=s_i^+$ and $|\beta_i|=s_i^+$ for some $s_i\in \ZZ$.  If $s_i\in \ZZ$ with $s_i\geq 0$, then
 $\beta_i=g_0^i$ and relation \eqref{agbo} becomes $\omega_i=\alpha_i\gamma_i$. In case, $s_i\in \ZZ$ with $s_i< 0$, we must have $\alpha_i=g_0^i$ and $\beta_i \omega_i=\gamma_i$. Therefore,  if $ (\boldsymbol \alpha,\boldsymbol \beta)\in \cJ$, then the relation \eqref{agbo}   holds  if and  and only if
 $(\boldsymbol \omega, \boldsymbol\gamma)\in \cC$ and
  ${\bf s}(\boldsymbol \omega, \boldsymbol\gamma)=(\boldsymbol\alpha,\boldsymbol\beta)$.
The proof is complete.
\end{proof}

\begin{lemma}\label{monom2}  The following statements hold.
\begin{enumerate}
\item[(i)] If $C\in B(\cK)$ and $(\boldsymbol \alpha, \boldsymbol \beta)\in \cJ$, then
$C\otimes {\bf W}_{\boldsymbol\alpha}{\bf W}_{\boldsymbol\beta}^*$ is a weighted multi-Toeplitz operator.
\item[(ii)] The set of all weighted multi-Toeplitz operators is WOT-closed.
\end{enumerate}
\end{lemma}
\begin{proof}  If $(\boldsymbol \alpha, \boldsymbol \beta)\in \cJ$ and $(\boldsymbol\omega, \boldsymbol\gamma)\in {\bf F}_{\bf n }^+\times  {\bf F}_{\bf n }^+$, then, as in the proof of Lemma \ref{monom1}, we deduce that
\begin{equation*}
\begin{split}
&\left<C\otimes {\bf W}_{1,\alpha_1}\cdots {\bf W}_{k,\alpha_k}{\bf W}_{1,\beta_1}^*\cdots {\bf W}_{k,\beta_k}^*
(e_{\gamma_1}^1\otimes \cdots \otimes e_{\gamma_k}^k), e_{\omega_1}^1\otimes \cdots \otimes e_{\omega_k}^k
\right> \\
&\qquad\qquad=
\begin{cases}
\tau_{(\boldsymbol\omega,\boldsymbol\gamma)}
\left<Cx,y\right>, & \text{ if  } (\boldsymbol \omega, \boldsymbol\gamma)\in \cC,\\
0,&  \text{ if  } (\boldsymbol \omega, \boldsymbol\gamma)\in ( {\bf F}_{\bf n }^+\times  {\bf F}_{\bf n }^+)\backslash \cC
\end{cases}
\end{split}
\end{equation*}
where $\tau_{(\boldsymbol\omega,\boldsymbol\gamma)}=\prod_{i=1}^k  \sqrt{\frac{b^{(m_i)}_{i,\min\{\omega_i, \gamma_i\}}}{b^{(m_i)}_{i,\max\{\omega_i, \gamma_i\}}}}$.
Therefore, item (i) holds. Since item (ii) follows easily  due to Proposition  \ref{Toep-def2} by taking appropriate limits, the proof is complete.
\end{proof}

\begin{theorem}  \label{formal2} Any weighted multi-Toeplitz operator  $T\in B(\cK\otimes \bigotimes_{s=1}^k  F^2(H_{n_s}))$ has
a unique formal Fourier representation
$$
\varphi_T({\bf W}, {\bf W}^*):=\sum_{(\boldsymbol \alpha, \boldsymbol \beta)\in \cJ}A_{   (\boldsymbol \alpha, \boldsymbol \beta)}\otimes {\bf W}_{\boldsymbol\alpha}{\bf W}_{\boldsymbol\beta}^*
$$
where $\{A_{   (\boldsymbol \alpha, \boldsymbol \beta)}\}_{(\boldsymbol \alpha, \boldsymbol \beta)\in \cJ}$ are some operators on the Hilbert space $\cK$ such that
$$
Tp=\varphi_T({\bf W}, {\bf W}^*)p,\qquad p\in \cP_K.
$$
\end{theorem}
\begin{proof} Since $T\in B(\cK\otimes \bigotimes_{s=1}^k  F^2(H_{n_s}))$ is a weighted multi-Toeplitz operator,  there exist operators $\{A_{(\boldsymbol \alpha, \boldsymbol \beta)}\}_{(\boldsymbol \alpha; \boldsymbol \beta)\in \cJ}\subset B(\cK)$ such that Definition \ref{MT} holds. More precisely,
$$
 \left<A_{(\boldsymbol \alpha, \boldsymbol \beta)}x,y\right>=\frac{1}{\tau_{(\boldsymbol \alpha, \boldsymbol \beta)}}\left<T(x\otimes e_{\boldsymbol\beta}), y\otimes e_{\boldsymbol\alpha} \right>,\qquad x,y\in \cK.
 $$
Due to  Definition \ref{MT}, if $\boldsymbol\gamma\in {\bf F}_{\bf n }^+$ and $x\in \cK$, we have
$$
T(x\otimes e_{\boldsymbol\gamma }) =\sum_{\boldsymbol\omega\in {\bf F}_{\bf n }^+: (\boldsymbol\omega, \boldsymbol\gamma)\in \cC}
 \tau_{(\boldsymbol\omega,\boldsymbol\gamma)}A_{{\bf s}(\boldsymbol \omega, \boldsymbol\gamma)}x \otimes e_{\boldsymbol\omega }
$$
is a vector in $\cK\otimes \bigotimes_{i=1}^k F^2(H_{n_i})$. Consequently, the series
\begin{equation}
\label{con}
\sum_{\boldsymbol\omega\in {\bf F}_{\bf n }^+: (\boldsymbol\omega, \boldsymbol\gamma)\in \cC}
 \tau_{(\boldsymbol\omega,\boldsymbol\gamma)}^2A_{{\bf s}(\boldsymbol \omega, \boldsymbol\gamma)}^*A_{{\bf s}(\boldsymbol \omega, \boldsymbol\gamma)}\quad \text{ is WOT-convergent}.
\end{equation}

Now, we consider the formal power series
$$
\varphi_T({\bf W}, {\bf W}^*):=\sum_{(\boldsymbol \alpha, \boldsymbol \beta)\in \cJ}A_{   (\boldsymbol \alpha, \boldsymbol \beta)}\otimes {\bf W}_{\boldsymbol\alpha}{\bf W}_{\boldsymbol\beta}^*
$$
and show that
$$
\varphi_T({\bf W}, {\bf W}^*)(x\otimes e_{\boldsymbol \gamma}):=\sum_{(\boldsymbol \alpha, \boldsymbol \beta)\in \cJ}A_{   (\boldsymbol \alpha, \boldsymbol \beta)}x\otimes {\bf W}_{\boldsymbol\alpha}{\bf W}_{\boldsymbol\beta}^*e_{\boldsymbol \gamma}
$$
is convergent for every $x\in \cK$ and $\boldsymbol\gamma\in {\bf F}_{\bf n }^+$. Indeed,  due to Lemma \ref{monom1},
if  $\boldsymbol\omega, \boldsymbol\gamma \in {\bf F}_{\bf n }^+$, then
  $$
  \left<{\bf W}_{\boldsymbol\alpha}{\bf W}_{\boldsymbol\beta}^*e_{\boldsymbol \gamma}, e_{\boldsymbol\omega}\right>\neq 0
  $$
  if and only if  $(\boldsymbol \omega, \boldsymbol\gamma)\in \cC$ and
  ${\bf s}(\boldsymbol \omega, \boldsymbol\gamma)=(\boldsymbol\alpha,\boldsymbol\beta)$. In this case, we have
  $\left<{\bf W}_{\boldsymbol\alpha}{\bf W}_{\boldsymbol\beta}^*e_{\boldsymbol \gamma}, e_{\boldsymbol\omega}\right>=\tau_{(\boldsymbol\omega, \boldsymbol\gamma)}$.
Using  Parseval's identity, we deduce that
\begin{equation*}
\begin{split}
\|{\bf W}_{\boldsymbol\alpha}{\bf W}_{\boldsymbol\beta}^*e_{\boldsymbol \gamma}\|^2
&=\sum_{\omega\in {\bf F}_{\bf n }^+} \left|\left<{\bf W}_{\boldsymbol\alpha}{\bf W}_{\boldsymbol\beta}^*e_{\boldsymbol \gamma},e_{\boldsymbol\omega}\right>\right|^2 \\
&=
\sum_{{\omega\in {\bf F}_{\bf n }^+: (\boldsymbol \omega, \boldsymbol\gamma)\in \cC}\atop{
  {\bf s}(\boldsymbol \omega, \boldsymbol\gamma)=(\boldsymbol\alpha,\boldsymbol\beta)}} \left|\left<{\bf W}_{\boldsymbol\alpha}{\bf W}_{\boldsymbol\beta}^*e_{\boldsymbol \gamma},e_{\boldsymbol\omega}\right>\right|^2
  \\
  &=\sum_{{\omega\in {\bf F}_{\bf n }^+: (\boldsymbol \omega, \boldsymbol\gamma)\in \cC}\atop{
  {\bf s}(\boldsymbol \omega, \boldsymbol\gamma)=(\boldsymbol\alpha,\boldsymbol\beta)}} \tau_{(\boldsymbol\omega, \boldsymbol\gamma)}^2.
\end{split}
\end{equation*}
Consequently, due to Lemma \ref{monom1} part (i), we have
\begin{equation*}
\begin{split}
\left\|\varphi_T({\bf W}, {\bf W}^*)(x\otimes e_{\boldsymbol \gamma})\right\|^2
&=
\sum_{{(\boldsymbol\alpha,\boldsymbol\beta)\in \cJ }}
\|A_{   (\boldsymbol \alpha, \boldsymbol \beta)}x\|^2\sum_{{\omega\in {\bf F}_{\bf n }^+: (\boldsymbol \omega, \boldsymbol\gamma)\in \cC}\atop{
  {\bf s}(\boldsymbol \omega, \boldsymbol\gamma)=(\boldsymbol\alpha,\boldsymbol\beta)}} \tau_{(\boldsymbol\omega, \boldsymbol\gamma)}^2\\
  &=\sum_{{\omega\in {\bf F}_{\bf n }^+: (\boldsymbol \omega, \boldsymbol\gamma)\in \cC}}
  \|A_{{\bf s}(\boldsymbol \omega, \boldsymbol\gamma)}x\|^2\tau_{(\boldsymbol\omega, \boldsymbol\gamma)}^2
  \end{split}
\end{equation*}
which is finite due to elation \eqref{con}, and proves our assertion.
Now, using Lemma \ref{monom2} and the results above, we deduce that
\begin{equation*}
\begin{split}
\left<\varphi_T({\bf W}, {\bf W}^*)(x\otimes e_{\boldsymbol \gamma}, y\otimes e_{\boldsymbol\omega}
\right>
&=
\sum_{(\boldsymbol \alpha, \boldsymbol \beta)\in \cJ}\left<A_{   (\boldsymbol \alpha, \boldsymbol \beta)}x,y\right>\left< {\bf W}_{\boldsymbol\alpha}{\bf W}_{\boldsymbol\beta}^*e_{\boldsymbol \gamma}, e_{\boldsymbol\omega}\right>\\
&=\begin{cases}
\tau_{(\boldsymbol\omega,\boldsymbol\gamma)}\left<A_{{\bf s}(\boldsymbol \omega, \boldsymbol\gamma)}x,y\right>, & \text{ if  } (\boldsymbol \omega, \boldsymbol\gamma)\in \cC,\\
0,&  \text{ if  } (\boldsymbol \omega, \boldsymbol\gamma)\in ({\bf F}_{\bf n }^+\times {\bf F}_{\bf n }^+)\backslash \cC,
\end{cases}
\\
&=
\left<T(x\otimes e_{\boldsymbol\gamma }), y\otimes e_{\boldsymbol\omega} \right>.
\end{split}
\end{equation*}
Hence, we have  $\varphi_T({\bf W}, {\bf W}^*)(x\otimes e_{\boldsymbol \gamma})=T(x\otimes e_{\boldsymbol\gamma })$ for every $x\in \cK$ and $\boldsymbol\gamma \in {\bf F}_{\bf n }^+$.

To prove uniqueness, assume that
$
\varphi({\bf W}, {\bf W}^*):=\sum_{(\boldsymbol \alpha, \boldsymbol \beta)\in \cJ}A_{   (\boldsymbol \alpha, \boldsymbol \beta)}'\otimes {\bf W}_{\boldsymbol\alpha}{\bf W}_{\boldsymbol\beta}^*
$
is a formal series such that $Tp=\varphi({\bf W}, {\bf W}^*)p$ for every $p\in \cK$. Then we must have
$\varphi_T({\bf W}, {\bf W}^*)p=\varphi({\bf W}, {\bf W}^*)p$. On the other hand,  if $(\boldsymbol \alpha, \boldsymbol \beta)\in \cJ$ then ${\bf s}(\boldsymbol \alpha, \boldsymbol \beta)=(\boldsymbol \alpha, \boldsymbol \beta)$. In this case, we have
$$
\left<\varphi_T({\bf W}, {\bf W}^*)(x\otimes e_{\boldsymbol \beta}, y\otimes e_{\boldsymbol\alpha}
\right>=\tau_{(\boldsymbol \alpha, \boldsymbol \beta))}\left<A_{(\boldsymbol \alpha, \boldsymbol \beta)}x,y\right>
$$
and
$$
\left<\varphi({\bf W}, {\bf W}^*)(x\otimes e_{\boldsymbol \beta}, y\otimes e_{\boldsymbol\alpha}
\right>=\tau_{(\boldsymbol \alpha, \boldsymbol \beta)}\left<A_{(\boldsymbol \alpha, \boldsymbol \beta)}'x,y\right>.
$$
Since $\tau_{(\boldsymbol \alpha, \boldsymbol \beta)}\neq 0$, the relations above imply
$A_{(\boldsymbol \alpha, \boldsymbol \beta)}=A_{(\boldsymbol \alpha, \boldsymbol \beta)}'$  for every
 $(\boldsymbol \alpha, \boldsymbol \beta)\in \cJ$.
The proof is complete.
\end{proof}

We remark that the formal Fourier  series $\varphi_T({\bf W}, {\bf W}^*)$  associated with the weighted multi-Toeplitz operator $T$ can be viewed as  its  noncommutative symbol.

\begin{theorem}\label{main2}
If $T\in B(\cK\otimes \bigotimes_{s=1}^k  F^2(H_{n_s}))$, then the following statements are equivalent.
\begin{enumerate}
\item[(i)] $T$ satisfies the Brown-Halmos condition \eqref{BH}.

\item[(ii)]   $T$ is a weighted multi-Toeplitz operator.
\end{enumerate}
\end{theorem}
\begin{proof}  Assume that $T$ satisfies the Brown-Halmos condition \eqref{BH}.
Due to Theorem \ref{main}, we have
$$
T\in\text{\rm span}\left\{ C\otimes {\bf W}_{\boldsymbol\alpha}{\bf W}_{\boldsymbol\beta}^* :\   C\in B(\cK),   (\boldsymbol \alpha; \boldsymbol \beta)\in \cJ\right\}^{-\text{\rm WOT}}.
$$
According to Lemma \ref{monom2}, any operator of the form $C\otimes {\bf W}_{\boldsymbol\alpha}{\bf W}_{\boldsymbol\beta}^* $ is a weighted multi-Toeplitz operator. Since the set of all weighted multi-Toeplitz operators is WOT-closed, we deduce that $T$ is a weighted multi-Toeplitz operator.

Now, we prove the implication (ii)$\implies$(i). To this end, assume that $T$ is a weighted multi-Toeplitz operator. Due to Theorem \ref{formal2},
$T$ has
a unique formal Fourier representation
$$
\varphi_T({\bf W}, {\bf W}^*):=\sum_{(\boldsymbol \alpha, \boldsymbol \beta)\in \cJ}A_{   (\boldsymbol \alpha, \boldsymbol \beta)}\otimes {\bf W}_{\boldsymbol\alpha}{\bf W}_{\boldsymbol\beta}^*
$$
where $\{A_{   (\boldsymbol \alpha, \boldsymbol \beta)}\}_{(\boldsymbol \alpha, \boldsymbol \beta)\in \cJ}$ are some operators on the Hilbert space $\cK$ such that
$$
Tp=\varphi_T({\bf W}, {\bf W}^*)p,\qquad p\in \cP_K.
$$
Let $\{T_{s_1,\ldots, s_k}\}_{(s_1,\ldots, s_k)\in \ZZ^k}$ be  the multi-homogeneous parts of $T$.
Recall from Section 2 that, for every $f\in \cK\otimes \cE_{p_1,\ldots, p_k}$, $(p_1,\ldots, p_k)\in \ZZ^k$,
we have
\begin{equation*}
\begin{split}
T_{s_1,\ldots, s_k}f&=\left(\frac{1}{2\pi}\right)^k
\int_0^{2\pi}\cdots \int_0^{2\pi} e^{-i(s_1+p_1)\theta_1}\cdots e^{-i(s_k+p_k)\theta_k}\left(I_\cK\otimes \Gamma (e^{i {\boldsymbol\theta}})\right) Tf d\theta_1\ldots d\theta_k\\
 &=(I_\cK\otimes{\bf P}_{s_1+p_1,\ldots, s_k+p_k})\varphi_T({\bf W}, {\bf W}^*)f
\end{split}
\end{equation*}
where ${\bf P}_{p_1,\ldots, p_k}\in B(\otimes_{s=1}^k F^2(H_{n_s}))$ is the orthogonal projection  onto the subspace $\cE_{p_1,\ldots, p_k}$. On the other hand,
$$
\varphi_T({\bf W}, {\bf W}^*):=\sum_{(\boldsymbol \alpha, \boldsymbol \beta)\in \cJ}A_{   (\boldsymbol \alpha, \boldsymbol \beta)}\otimes {\bf W}_{\boldsymbol\alpha}{\bf W}_{\boldsymbol\beta}^*=\sum_{t_1\in \ZZ}\cdots \sum_{t_k\in \ZZ}
 q_{t_1,\ldots, t_k}({\bf W}, {\bf W}^*)
$$
where
 \begin{equation*}
 q_{t_1,\ldots, t_k}({\bf W}, {\bf W}^*):=
\sum_{{\alpha_i,\beta_i\in \FF_{n_i}^+, i\in \{1,\ldots, k\}}\atop{|\alpha_i|=t_i^+, |\beta_i|=t_i^-}} A_{(\alpha_1,\ldots,\alpha_k;\beta_1,\ldots, \beta_k)}\otimes {\bf W}_{1,\alpha_1}\cdots {\bf W}_{k,\alpha_k}{\bf W}_{1,\beta_1}^*\cdots {\bf W}_{k,\beta_k}^*.
\end{equation*}
Combining these results and  using  that fact that  $q_{t_1,\ldots, t_k}({\bf W}, {\bf W}^*)f\in \cK\otimes\cE_{t_1+p_1,\ldots t_k+p_k}$  and the subspaces $\{\cE_{p_1,\ldots, p_k}\}_{(p_1,\ldots, p_k)\in \ZZ^k}$ are pairwise orthogonal, we obtain
$$
T_{s_1,\ldots, s_k}f=(I_\cK\otimes{\bf P}_{s_1+p_1,\ldots, s_k+p_k})\sum_{t_1\in \ZZ}\cdots \sum_{t_k\in \ZZ}
 q_{t_1,\ldots, t_k}({\bf W}, {\bf W}^*)f=q_{s_1,\ldots, s_k}({\bf W}, {\bf W}^*)f
$$
for every $f\in \cK\otimes \cE_{p_1,\ldots, p_k}$ and $(p_1,\ldots, p_k)\in \ZZ^k$. Therefore,
$$
T_{s_1,\ldots, s_k}=q_{s_1,\ldots, s_k}({\bf W}, {\bf W}^*), \qquad (s_1,\ldots, s_k)\in \ZZ^k.
$$
Due to Lemma \ref{qBH},  each operator $T_{s_1,\ldots, s_k}$ satisfies  the Brown-Halmos condition \eqref{BH}.
 On the other hand, due to Proposition \ref{homo-decomp},  the operator $T$ can be reconstructed from its multi-homogeneius parts, i. e.
    $$
   Tg=\lim_{N_1\to \infty}\ldots \lim_{N_k\to \infty} \sum_{(s_1,\ldots, s_k)\in \ZZ^k, |s_j|\leq N_j}
   \left(1-\frac{|s_1|}{N_1+1}\right)\cdots \left(1-\frac{|s_k|}{N_k+1}\right) T_{s_1,\ldots, s_k}g
   $$
for every $g\in \cK\otimes \bigotimes_{s=1}^k  F^2(H_{n_s})$, where the limit is in norm.  Since the set of all operators satisfying the Brown-Halmos condition is WOT-closed, we deduce that $T$  satisfies the condition as well.
The proof is complete.
 \end{proof}

If $ (\boldsymbol \alpha, \boldsymbol \beta)\in {\bf F}_{\bf n}^+\times {\bf F}_{\bf n}^+$, we define its length to be
$| (\boldsymbol \alpha, \boldsymbol \beta)|:=|\alpha_1|+\cdots +|\alpha_k|+|\beta_1|+\cdots +|\beta|_k$.

\begin{theorem}   \label{Fourier} Let $\{A_{   (\boldsymbol \alpha, \boldsymbol \beta)}\}_{(\boldsymbol \alpha, \boldsymbol \beta)\in \cJ}$ be  a family of  operators on the Hilbert space $\cK$ and let
$$
\varphi({\bf W}, {\bf W}^*):=\sum_{(\boldsymbol \alpha, \boldsymbol \beta)\in \cJ}A_{   (\boldsymbol \alpha, \boldsymbol \beta)}\otimes {\bf W}_{\boldsymbol\alpha}{\bf W}_{\boldsymbol\beta}^*=\sum_{s_1\in \ZZ}\cdots \sum_{s_k\in \ZZ}
 q_{s_1,\ldots, s_k}({\bf W}, {\bf W}^*)
$$
be a formal series.
Then the following statements are equivalent.
\begin{enumerate}
\item[(i)]   For each $\boldsymbol\gamma\in {\bf F}_{\bf n}^+$, the series
$$
\sum_{\boldsymbol\omega\in {\bf F}_{\bf n}^+: (\boldsymbol\omega,\boldsymbol\gamma)\in \cC}
\tau_{(\boldsymbol\omega,\boldsymbol\gamma)}^2 A_{{\bf s}(\boldsymbol\omega,\boldsymbol\gamma)}^*
A_{{\bf s}(\boldsymbol\omega,\boldsymbol\gamma)}\qquad \text{is {\rm WOT}-convergent}
$$
 and
$$\sup_{r\in [0,1)} \sup_{p\in \cP_\cK, \|p\|\leq 1} \|\varphi(r{\bf W}, r{\bf W}^*)p\|<\infty.$$

\item[(ii)] $\varphi({\bf W}, {\bf W}^*)$ is the formal Fourier representation of a weighted multi-Toeplitz operator  $T$.
\item[(iii)] For each $r\in [0,1)$, the series
$$
\varphi(r{\bf W}, r{\bf W}^*):= \sum_{s_1\in \ZZ}\cdots \sum_{s_k\in \ZZ}
 r^{|s_1|+\cdots+|s_k|}q_{s_1,\ldots, s_k}({\bf W}, {\bf W}^*)
$$
converges in the operator norm topology
and \
$\sup_{r\in [0,1)}   \|\varphi(r{\bf W}, r{\bf W}^*)\|<\infty.$
 \item[(iv)]    There is $M>0$ such that
$$
\|q_{s_1,\ldots, s_k}({\bf W}, {\bf W}^*)\|\leq M,\qquad (s_1,\ldots, s_k)\in \ZZ^k,
$$
$$\sup_{r\in [0,1)}   \|\varphi(r{\bf W}, r{\bf W}^*)\|<\infty.$$
 \end{enumerate}
In this case,
$$
T=\text{\rm SOT-}\lim_{r\to 1} \varphi (r{\bf W}, r{\bf W}^*)\quad \text{and}\quad \|T\|=\sup_{r\in [0,1)}   \|\varphi(r{\bf W}, r{\bf W}^*)\|.
$$
\end{theorem}
\begin{proof}
Assume that item (i) holds. As in the proof of Theorem \ref{formal2}, it is easy to see that, for every $x\in\cK$,
$$
 \sum_{\boldsymbol\omega\in \cF: (\boldsymbol\omega, \boldsymbol\gamma)\in \cC}
 \tau_{(\boldsymbol\omega,\boldsymbol\gamma)}A_{{\bf s}(\boldsymbol \omega, \boldsymbol\gamma)}x \otimes e_{\boldsymbol\omega }
$$
is a vector in $\cK\otimes \bigotimes_{i=1}^k F^2(H_{n_i})$ and, therefore, so are  $ \varphi ({\bf W}, {\bf W}^*)p$ and $ \varphi (r{\bf W}, r{\bf W}^*)p$ for every $p\in \cP_\cK$, $r\in [0,1)$.
Hence, we deduce that
\begin{equation}
\label{frr}
 \lim_{r\to 1}\varphi (r{\bf W}, r{\bf W}^*)p= \varphi ({\bf W}, {\bf W}^*)p,\qquad p\in \cP_\cK.
\end{equation}
Using the fact that, for each $r\in [0,1)$,
 \begin{equation*}
 \sup_{p\in\cP_\cK, \|p\|\leq 1}\|\varphi (r{\bf W}, r{\bf W}^*)p\|<\infty,
\end{equation*}
we conclude that there is a bounded linear operator $T_r\in B(\cK\otimes \bigotimes_{i=1}^k F^2(H_{n_i}))$ such that
\begin{equation}
\label{Trf}
  T_r p=\varphi (r{\bf W}, r{\bf W}^*)p,\qquad p\in \cP_\cK.
\end{equation}
Note that, for every $ \boldsymbol\omega, \boldsymbol\gamma \in {\bf F}_{n}^+$ and $x,y\in \cK$,
\begin{equation*}
\begin{split}
\left<T_r(x\otimes e_{\boldsymbol\gamma }), y\otimes e_{\boldsymbol\omega} \right>
&=
\left<\varphi (r{\bf W}, r{\bf W}^*)(x\otimes e_{\boldsymbol\gamma }), y\otimes e_{\boldsymbol\omega} \right>\\
&=\begin{cases}
\tau_{(\boldsymbol\omega,\boldsymbol\gamma)}\left<r^{|{\bf s}(\boldsymbol \omega, \boldsymbol\gamma)|}A_{{\bf s}(\boldsymbol \omega, \boldsymbol\gamma)}x,y\right>, & \text{ if  } (\boldsymbol \omega, \boldsymbol\gamma)\in \cC,\\
0,&  \text{ if  } (\boldsymbol \omega, \boldsymbol\gamma)\in ({\bf F}_{\bf n }^+\times {\bf F}_{\bf n }^+)\backslash \cC.
\end{cases}
\end{split}
\end{equation*}
Consequently, $T_r$ is a weighted multi-Toeplitz operator.
Now, note that due to relation \eqref{frr} and the fact that  $\sup_{ r\in[0,1)} \sup_{p\in\cP_\cK, \|p\|\leq 1}\|\varphi (r{\bf W}, r{\bf W}^*)p\|<\infty$, we deduce that
$$\sup_{p\in \cP_\cK, \|p\|\leq 1}\|\varphi ({\bf W}, {\bf W}^*)p\|<\infty.
$$
Consequently,  there is a bounded linear operator $T$ on
$\cK\otimes \bigotimes_{i=1}^k F^2(H_{n_i})$ such that $Tp=
\varphi ({\bf W}, {\bf W}^*)p $ for every $p\in \cP_\cK$.
Now, it is clear that
$$
\lim_{r\to 1}T_r p=\lim_{r\to 1} \varphi (r{\bf W}, r{\bf W}^*)p=
\varphi ({\bf W}, {\bf W}^*)p=Tp
$$
for every $p\in \cP_\cK$ and, due to item (i), $\sup_{r\in [0,1)}\|T_r\|<\infty$. This implies $T=\text{\rm SOT-}\lim_{r\to 1}T_r$.  Since $T_r$ is a weighted multi-Toeplitz operator, we can use Lemma \ref{monom2} and relation \eqref{Trf}, to deduce that  so is $T$.
Since $Tp=\varphi ({\bf W}, {\bf W}^*)p$ for every $p\in \cP_\cK$, Theorem \ref{formal2} shows that
$ \varphi ({\bf W}, {\bf W}^*)$ is the formal Fourier representation of $T$. Therefore, item (ii) holds.

Now, we prove that (ii)$\implies$(iii) and (ii)$\implies$(iv). Assume that $\varphi({\bf W}, {\bf W}^*)$ is the formal Fourier representation of a weighted multi-Toeplitz operator  $T$.
Due to Theorem \ref{main2} (see also its proof), $T$ satisfies the Brown-Halmos condition \eqref{BH} and the multi-homogeneous parts of $T$ are $T_{s_1,\ldots, s_k}=q_{s_1,\ldots, s_k}({\bf W}, {\bf W}^*)$ for every $(s_1,\ldots, s_k)\in \ZZ^+$.
As we saw in the proof of Theorem \ref{main}, we have
$\|T_{s_1,\ldots, s_k}\|\leq \|T\|$ for every $(s_1,\ldots, s_k)\in \ZZ^+$ and, as a consequence, the series
$\sum_{(s_1,\ldots, s_k)\in \ZZ^+}r^{|s_1|+\cdots +|s_k|} T_{s_1,\ldots, s_k}$ is convergent in the operator norm topology.
Moreover, according to inequality \eqref{VN}, we have
\begin{equation*}
\left\|\sum_{(s_1,\ldots, s_k)\in \ZZ^k}q_{s_1,\ldots, s_k}(r{\bf W}, r{\bf W}^*)\right\|\leq \|T\|,\qquad r\in[0,1),
\end{equation*}
which implies
$\sup_{r\in [0,1)}\|\varphi (r{\bf W}, r{\bf W}^*)p\|<\infty$.
Therefore, items (iii) and (iv) hold. Moreover, in the proof of Theorem \ref{main}, we also proved that
$
T=\text{\rm SOT-}\lim_{r\to 1} \varphi (r{\bf W}, r{\bf W}^*).
$
On the other hand, Corollary \ref{norm} shows that  $\|T\|=\sup_{r\in [0,1)}\|\varphi (r{\bf W}, r{\bf W}^*)\|$.

Since the implication (iv)$\implies$(iii) is obvious, it remains to prove that (iii)$\implies$(i). To this end, assume that item (iii) holds. Then, for each $\gamma\in {\bf F}_{\bf n}^+$ and $x\in \cK$,
$\sup_{r\in [0,1)}\|\varphi (r{\bf W}, r{\bf W}^*)(x\otimes e_\gamma)\|<\infty$. Since
$$
\varphi (r{\bf W}, r{\bf W}^*)(x\otimes e_\gamma)
=
\sum_{\boldsymbol\omega\in {\bf F}_{\bf n}^+: (\boldsymbol\omega,\boldsymbol\gamma)\in \cC} r^{2|{\bf s}(\boldsymbol \omega, \boldsymbol\gamma)|}
\tau_{(\boldsymbol\omega,\boldsymbol\gamma)}^2\|
\|A_{{\bf s}(\boldsymbol\omega,\boldsymbol\gamma)}x\|^2,
$$
we deduce that
$$
\sum_{\boldsymbol\omega\in {\bf F}_{\bf n}^+: (\boldsymbol\omega,\boldsymbol\gamma)\in \cC}
\tau_{(\boldsymbol\omega,\boldsymbol\gamma)}^2 A_{{\bf s}(\boldsymbol\omega,\boldsymbol\gamma)}^*
A_{{\bf s}(\boldsymbol\omega,\boldsymbol\gamma)}\qquad \text{is {\rm WOT}-convergent}.
$$
Due to the fact that
$$\sup_{r\in [0,1)} \sup_{p\in \cP_\cK, \|p\|\leq 1} \|\varphi(r{\bf W}, r{\bf W}^*)p\|
=\sup_{r\in [0,1)}\|\varphi (r{\bf W}, r{\bf W}^*)\|
<\infty,$$
item (i) holds.
The proof is complete.
\end{proof}

For each $i\in \{1,\ldots, k\}$, let $F_{n_i,m_i}^2$ be the Hilbert space of formal power series in noncommutative indeterminates $Z_{i,1},\ldots, Z_{i,n_i}$ with complete orthogonal basis $\{Z_{i,\alpha}: \ \alpha \in \FF_{n_i}^+\}$  such that $\|Z_{i,\alpha}\|_{i,m_i}:=\frac{1}{\sqrt{b_{i,\alpha}^{(m_i)}}}$.  It is clear that
$$
F_{n_i,m_i}^2=\left\{ \varphi:=\sum_{\alpha\in \FF_{n_i}^+} a_\alpha Z_{i,\alpha}: \ a_\alpha\in \CC \ \text{\rm and }\  \|\varphi\|_{i,m_i}^2:=
\sum_{\alpha\in \FF_{n_i}^+} \frac{1}{b_{i,\alpha}^{(m_i)}} |a_\alpha|^2<\infty\right\},
$$
which can be seen as a weighted Fock space with $n_i$ generators.
The left multiplication operators $L_{i,1} ,\ldots, L_{i,n_i} $ are defined by
$L_{i ,j}\xi:=Z_{i,j}\xi$ \, for all $\xi\in F^2_{i,m_i}$.
 For each $i\in \{1,\ldots, k\}$ and $j\in \{1,\ldots, n_i\}$, we
define the operator ${\bf L}_{i,j}$ acting on the tensor Hilbert space
$F^2_{n_1,m_1}\otimes\cdots\otimes F^2_{n_k,m_k}$ by setting
$${\bf L}_{i,j}:=\underbrace{I\otimes\cdots\otimes I}_{\text{${i-1}$
times}}\otimes L_{i,j}\otimes \underbrace{I\otimes\cdots\otimes
I}_{\text{${k-i}$ times}}.
$$

 Note that the operator
$U_{i,m_i}:F^2(H_{n_i})\to F^2_{n_i,m_i}$  defined by
$
U_{i,m_i}(e^i_\alpha):=\sqrt{b_\alpha^{(m_i)}} Z_{i,\alpha}$, $ \alpha\in \FF_{n_i}^+,
$
is unitary and
   $U_{i,m_i}W_{i,j}=L_{i,j} U_{i,m_i}$  for every $ j\in \{1,\ldots, n_i\}.
$
Consequently, the operator  ${\bf U}:=U_{1,m_1}\otimes\cdots \otimes U_{k,m_k}:\otimes_{i=1}^k F^2(H_{n_i})\to\otimes_{i=1}^k F^2_{n_i,m_i}$
is unitary and
   ${\bf U}{\bf W}_{i,j}={\bf L}_{i,j} {\bf U}$  for every  $i\in \{1,\ldots, k\}$ and $ j\in \{1,\ldots, n_i\}.
$
A straightforward calculation reveals that $T\in B(\cK\bigotimes  \otimes_{i=1}^k F^2(H_{n_i}))$  is  a  weighted multi-Toeplitz operator if and only if  there exist operators $\{A_{(\boldsymbol \sigma; \boldsymbol \beta)}\}_{(\boldsymbol \sigma; \boldsymbol \beta)\in \cJ}\subset B(\cK)$ such that the operator $T':={\bf U}T{\bf U}^*$   satisfies the relation
$$
\left<T'(x\otimes {\bf Z}_{\boldsymbol\gamma }), y\otimes {\bf Z}_{\boldsymbol\omega} \right>
=\begin{cases}
\mu_{(\boldsymbol\omega,\boldsymbol\gamma)}\left<A_{{\bf s}(\boldsymbol \omega, \boldsymbol\gamma)}x,y\right>, & \text{ if  } (\boldsymbol \omega, \boldsymbol\gamma)\in \cC,\\
0,&  \text{ if  } (\boldsymbol \omega, \boldsymbol\gamma)\in ({\bf F}_{\bf n}^+\times {\bf F}_{\bf n}^+)\backslash \cC,

\end{cases}
$$
for every $ \boldsymbol\omega, \boldsymbol\gamma \in {\bf F}_{\bf n}^+$,
where the weights $ \{\mu_{(\boldsymbol\omega,\boldsymbol\gamma)}\}_{(\boldsymbol \omega, \boldsymbol\gamma)\in \cC}$ are given by
$$ \mu_{(\boldsymbol\omega,\boldsymbol\gamma)}:=
 \prod_{i=1}^k  \frac{1}{b^{(m_i)}_{i,\max\{\omega_i, \gamma_i\}}}.
 $$
We should mention that all the results of  our paper can be written in the setting of multi-Toeplitz operators on tensor products of weighted Fock spaces.

 In the particular case when $k=1$ and $n_1=1$, the space $F^2_{1,m_1}$  coincides with the weighted Bergman space $A_{m_1}(\DD)$. The results of this section imply the fact that $T'$ is a  Toeplitz operator with operator-valued bounded harmonic symbol on $\DD$  if and only if it satisfies the Brown-Halmos equation where the weighted right creation operators $\Lambda_{i,j}$ are replaced by the right creation operators $R_{i,j}$ acting on the weighted Fock space $F^2_{i,n_i}$ by  $R_{i,j}\xi:=\xi Z_{i,j}$, $j\in \{1,\ldots, n_i\}$.   In the scalar case when $\cK=\CC$,   we recover the corresponding result obtained by Louhichi and Olofsson   in  \cite{LO}.

 We remark that, when $n_i=m_i=1$ for $i\in\{1,\ldots, k\}$, the tensor product  $F^2_{1,1}\otimes \cdots \otimes F^2_{1,1}$ is identified with the Hardy space $H^2(\DD^k)$ and the Brown-Halmos condition becomes $M_{z_i }^*T'M_{z_i}=T'$ for every $i\in \{1,\ldots, k\}$.
 In this case, $T'$ is a multi-Toeplitz operator if and only if $T'=P_{H^2(\DD^k)} M_\varphi |_{H^2(\DD^k)}$ for some
 $\varphi\in L^\infty(\TT^k)$.  We should mention  that the Brown-Halmos type characterization of Toeplitz operators on $H^2(\DD^k)$ was  recently obtained in \cite{MSS}.

 In the particular  case when $n_i=1$  for $i\in \{1,\ldots, k\}$ and and ${\bf m}=(m_1,\ldots, m_k)\in \NN^k$, the tensor product $F^2_{1,m_1}\otimes \cdots \otimes F^2_{1,m_k}$ is identified with the  reproducing kernel Hilbert space with    reproducing kernel
 $$
 \kappa_{\bf m}(z,w):=\prod_{i=1}^k \frac{1}{(1-\bar z_i w_i)^{m_i}}, \qquad z=(z_1,\ldots, z_k), w=(w_1,\ldots, w_k) \in \DD^k.
 $$
 In this case, the standard orthonormal basis is
 $$
 \left\{\sqrt{\prod_{i=1}^k\left(\begin{matrix} s_i+m_i-1\\s_i
\end{matrix}\right)} z_1^{s_1}\cdots z_k^{s_k}:\  (s_1,\ldots, s_k)\in \{0, 1, \ldots\}\right\}.
$$
 All the results of the present paper hold,  in particular,  for these reproducing kernel Hilbert spaces, which
 include the  Hardy space, the Bergman space, and the weighted Bergman space  over the polydisk.

\bigskip

\section{Bounded free $k$-pluriharmonic functions}

In this section,  we prove that the   bounded free $k$-pluriharmonic  functions on the radial poly-hyperball
  are precisely those  that are noncommutative Berezin transforms of the weighted multi-Toeplitz operators. In this setting, we solve the Dirichlet extension problem.

Denote by  ${\bf PH}_\cK^\infty({\bf D}_{{\bf n},rad}^{\bf m})$  the set of all bounded free
 $k$-pluriharmonic functions on the radial  poly-hyperball ${\bf D}_{{\bf n},rad}^{\bf m}$ with coefficients in
$B(\cK)$.
We define the norms $\|\cdot
\|_m:M_m\left({\bf PH}_\cE^\infty({\bf D}_{{\bf n},rad}^{\bf m})\right)\to [0,\infty)$, $m\in \NN$,  by
setting
$$
\|[F_{ij}]_m\|_m:= \sup \|[F_{ij}({\bf X})]_m\|,
$$
where the supremum is taken over all elements ${\bf X}\in {\bf D}_{{\bf n},rad}^{\bf m}(\cH)$ and any Hilbert space $\cH$. It is easy to see that the norms
$\|\cdot\|_m$, $m\in\NN$, determine  an operator space
structure  on ${\bf PH}_\cK^\infty({\bf D}_{{\bf n},rad}^{\bf m})$,
 in the sense of Ruan (see e.g. \cite{ER}).

The   {\it extended noncommutative Berezin transform at} ${\bf X}\in {\bf D}_{{\bf n},rad}^{\bf m}(\cH)$
 is the map
 $$\widetilde{\bf B}_{\bf X}: B(\cK \otimes  \bigotimes_{i=1}F^2(H_{n_i}))\to B(\cK)\otimes_{min} B(\cH)
 $$
 defined by
 \begin{equation*}
 \widetilde{\bf B}_{\bf X}[g]:= \left(I_\cK\otimes {\bf K}_{\bf X} ^*\right) (g\otimes I_\cH)\left(I_\cK\otimes{\bf K}_{\bf X}\right),
 \quad g\in B(\cK\otimes \bigotimes_{i=1}^k F^2(H_{n_i})),
 \end{equation*}
where  ${\bf K}_{\bf X}:\cH \to \left(\otimes_{i=1}^kF^2(H_{n_i})\right)\otimes
\cH$  is noncommutative Berezin kernel associated with ${\bf X}\in {\bf D}_{{\bf n},rad}^{\bf m}(\cH)$.

Throughout this section we assume that $\cH$ is a separable infinitely dimensional Hilbert space. Consequently, one can identify any free $k$-pluriharmonic function with its representation on $\cH$.
Let  $\boldsymbol{\cT }$ be the set of all  of all weighted multi-Toeplitz operators  on $\cK\otimes\bigotimes_{i=1}^k F^2(H_{n_i})$.
The main result of this section is the following characterization of
bounded  free $k$-pluriharmonic  functions on ${\bf D_n^m}$.

 \begin{theorem}\label{bounded}
  If $F: {\bf D}_{{\bf n},rad}^{\bf m}(\cH)\to B(\cK)\otimes_{min}B(\cH)$, then the following statements are equivalent.
\begin{enumerate}
\item[(i)] $F$ is a bounded free $k$-pluriharmonic function.
\item[(ii)]
There exists $T\in \boldsymbol{\cT }$ such
that
$$F({\bf X})=  \widetilde{\bf B}_{\bf X}[T], \qquad
{\bf X}\in {\bf D}_{{\bf n},rad}^{\bf m}(\cH).
$$
\end{enumerate}
In this case,
  $T=\text{\rm SOT-}\lim\limits_{r\to 1}F(r{\bf W}).
  $
   Moreover, the map
$$
\Phi:{\bf PH}_\cK^\infty({\bf D}_{{\bf n},rad}^{\bf m})\to \boldsymbol{\cT}\quad
\text{ defined by } \quad \Phi(F):=T
$$ is a completely   isometric isomorphism of operator spaces.
\end{theorem}
\begin{proof}  Assume that  item (i) holds and let  $F$ have the representation
$$
F({\bf X})=\sum_{s_1\in \ZZ}\cdots \sum_{s_k\in \ZZ}
 q_{s_1,\ldots, s_k} ({\bf X}, {\bf X}^*),\qquad {\bf X}\in {\bf D}_{{\bf n},rad}^{\bf m}(\cH).
$$
Then $F(r{\bf W})=\sum_{s_1\in \ZZ}\cdots \sum_{s_k\in \ZZ}
 r^{|s_1|+\cdots+|s_k|}q_{s_1,\ldots, s_k}({\bf W}, {\bf W}^*)
$
is convergent in the operator norm topology and, due to the von Neumann inequality for polydomains (see \cite{Po-Berezin1}),
we have
$\sup_{r\in [0,1)}\|F(r{\bf W})\|=\|F\|$.
According to Theorem \ref{Fourier}, $F({\bf W})$ is the formal Fourier representation of a weighted multi-Toeplitz operator $T$ and
 $
T=\text{\rm SOT-}\lim_{r\to 1} F(r{\bf W}).
$
Due to Theorem \ref{main} the operator $T$ satisfies the Brown-Halmos condition and, Corollary \ref{norm} shows that
$\|T\|=\sup_{r\in [0,1)}\|F(r{\bf W})\|$. On the other hand, due to the properties of the Berezin transform, we have
$$
F(r{\bf X})=\widetilde{\bf B}_{{\bf X}}[F(r{\bf W})]=\left(I_\cK\otimes {\bf K}_{{\bf X}} ^*\right) (F(r{\bf W})\otimes I_\cH)\left(I_\cK\otimes{\bf K}_{{\bf X}}\right),\qquad r\in [0,1),  {\bf X}\in {\bf D}_{{\bf n},rad}^{\bf m}(\cH).
$$
Since the map $Y\mapsto Y\otimes I_\cH$ is SOT-continuous on bounded subsets of $B(\cK \otimes  \bigotimes_{i=1}F^2(H_{n_i}))$ and $F$ is continuous on  ${\bf D}_{{\bf n},rad}^{\bf m}(\cH)$, we obtain
$$
F({\bf X})=\text{\rm SOT-}\lim_{r\to 1} F(r{\bf X})=\left(I_\cK\otimes {\bf K}_{{\bf X}} ^*\right) (T\otimes I_\cH)\left(I_\cK\otimes{\bf K}_{{\bf X}}\right)=\widetilde{\bf B}_{\bf X}[T].
$$
Therefore, item (ii) holds. Conversely, assume that item (ii) is satisfied and let
$$
\varphi({\bf W}, {\bf W}^*): =\sum_{s_1\in \ZZ}\cdots \sum_{s_k\in \ZZ}
 q_{s_1,\ldots, s_k}({\bf W}, {\bf W}^*)
$$
be the formal Fourier representation of a weighted multi-Toeplitz operator  $T\in B(\cK\otimes \bigotimes_{s=1}^k  F^2(H_{n_s}))$ (see Theorem \ref{formal2}). According to Theorem \ref{Fourier}, $\varphi(r{\bf W}, r{\bf W}^*)$ is convergent in the operator norm topology and $T=\text{\rm SOT-}\lim_{r\to 1} \varphi(r{\bf W}, r{\bf W}^*)$.
Due to Theorem \ref{main}, $T$ satisfies the Brown-Halmos condition and, consequently,
$\|T\|=\sup_{r\in [0,1)}\|\varphi(r{\bf W}, r{\bf W}^*)\|$ (see Corollary \ref{norm}).
Consequently, the function ${\bf X}\mapsto \varphi({\bf X}, {\bf X}^*)$ is a free $k$-pluriharmonic on the radial poly-hyperball ${\bf D}_{{\bf n},rad}^{\bf m}(\cH)$.  According to relation \eqref{KK}, we have
\begin{equation*}
(I_\cK\otimes {\bf K}_{r{\bf W}}^*) (T\otimes I_{\otimes_{i=1}^kF^2(H_{n_{i}})})(I_\cK\otimes {\bf K}_{r{\bf W}})
=\sum_{(s_1,\ldots, s_k)\in \ZZ^k}q_{s_1,\ldots, s_k}(r{\bf W}, r{\bf W}^*)
\end{equation*}
for every $r\in [0,1)$. Since we assume   item (ii), we also have
$$
F(r{\bf W})=(I_\cK\otimes {\bf K}_{r{\bf W}}^*) (T\otimes I_{\otimes_{i=1}^kF^2(H_{n_{i}})})(I_\cK\otimes {\bf K}_{r{\bf W}}), \qquad r\in[0,1).
$$
Combining these relations, we obtain $\varphi(r{\bf W}, r{\bf W}^*)=F(r{\bf W})$ for every $r\in [0,1)$, which implies
$\varphi=F$.

To prove the last part of the theorem, let $[F_{ij}]_m\in M_m({\bf PH}_\cK^\infty({\bf D}_{{\bf n},rad}^{\bf m})$ be a matrix  and use the noncommutative von Neumann inequality for polydomains to obtain
$$
\|[F_{ij}]_m\|=\sup_{{\bf X}\in {\bf D}_{{\bf n},rad}^{\bf m}(\cH)}\|[F_{ij}({\bf X})]_m\|
=\sup_{r\in [0,1)}\|[F_{ij}(r{\bf W})]_m\|.
$$
On the other hand, $T_{ij}:=\text{\rm SOT-}\lim_{r\to 1} F_{ij}(r{\bf W})$ is a weighted multi-Toeplitz operator and
$$
 F_{ij}(r{\bf W})=(I_{\cK}\otimes {\bf K}_{r{\bf W}}^*)(T_{ij}\otimes I_{\otimes_{i=1}^kF^2(H_{n_{i}})})
(I_{\cE}\otimes {\bf K}_{r{\bf W}}),\qquad r\in[0,1).
$$
Hence, we obtain
$$
\sup_{r\in [0,1)}\|[F_{ij}(r{\bf W})]_m\|\leq \|[T_{ij}]_m\|.
$$
Since  $[T_{ij}]_m:=\text{\rm SOT-}\lim_{r\to 1} [F_{ij}(rW)]_m$, we deduce that the inequality above is   an equality. This shows that $\Phi$ is a completely isometric isomorphisms of operator spaces. The proof is complete.
\end{proof}

As a consequence, we can obtain the following Fatou type result concerning the boundary behaviour of bounded free $k$-pluriharmonic functions.

\begin{corollary}
If $F:{\bf D}_{{\bf n},rad}^{\bf m}(\cH)\to B(\cK)\otimes_{min} B( \cH)$ is a bounded free $k$-pluriharmonic function and ${\bf X}$ is a pure element in
${\bf D}_{{\bf n}}^{\bf m}(\cH)$, then the limit
$$
\text{\rm SOT-}\lim_{r\to 1} F(r{\bf X})
$$
exists.
\end{corollary}

We denote by ${\bf PH}_\cK^c({\bf D}_{{\bf n},rad}^{\bf m}) $ the set of all
  free $k$-pluriharmonic functions on the radial part of ${\bf D_n^m}$ with operator-valued coefficients in $B(\cK)$, which
 have continuous extensions   (in the operator norm topology) to
  ${\bf D_n^m}(\cH)$, for every Hilbert space $\cH$.
In what follows we solve the Dirichlet extension problem for poly-hyperballs.

\begin{theorem}\label{Dirichlet}  If $F:{\bf D}_{{\bf n},rad}^{\bf m}(\cH)\to B(\cK)\otimes_{min} B( \cH)$, then
 the following statements are equivalent.
\begin{enumerate}
\item[(i)] $F$ is a free $k$-pluriharmonic function on the radial poly-hyperball ${\bf D}_{{\bf n},rad}^{\bf m}$
such that \ $F(r{\bf W})$ converges in the operator norm
topology, as $r\to 1$.

\item[(ii)]
There exists $T\in {\boldsymbol\cG}:=\text{\rm span}\left\{ C\otimes {\bf W}_{\boldsymbol\alpha}{\bf W}_{\boldsymbol\beta}^* :\   C\in B(\cK),   (\boldsymbol \alpha; \boldsymbol \beta)\in \cJ\right\}^{\|\cdot\|}$ such
that $$F({\bf X})= \widetilde{\bf B}_{\bf X}[T], \qquad
 {\bf X}\in {\bf D}_{{\bf n},rad}^{\bf m}(\cH).
  $$
   \item[(iii)] $F$ is a free $k$-pluriharmonic function on the radial poly-hyperball      ${\bf D}_{{\bf n},rad}^{\bf m}(\cH)$ which
 has a continuous extension  (in the operator norm topology) to the   poly-hyperball
  ${\bf D}_{{\bf n}}^{\bf m}(\cH)$.

\end{enumerate}
In this case, $T=\lim\limits_{r\to 1}F(r{\bf W})$, where
the convergence is in the operator norm. Moreover, the map
$$
\Phi:{\bf PH}_\cK^c({\bf D}_{{\bf n}}^{\bf m} ) \to \boldsymbol\cG
 \quad \text{ defined
by } \quad \Phi(F):=T
$$ is a  completely   isometric isomorphism of
operator spaces.
\end{theorem}
\begin{proof} We prove the equivalence of (i) with (ii).
Let $F$ have a representation
$$
F({\bf X})=\sum_{s_1\in \ZZ}\cdots \sum_{s_k\in \ZZ}
 q_{s_1,\ldots, s_k} ({\bf X}, {\bf X}^*),\qquad {\bf X}\in {\bf D}_{{\bf n},rad}^{\bf m}(\cH),
$$
where the series converge in the operator norm topology, such that
$T:=\lim_{r\to 1} F(r{\bf W})$ exists in the operator norm topology. Since
$$
F(r{\bf X})=\widetilde{\bf B}_{{\bf X}}[F(r{\bf W})]=\left(I_\cK\otimes {\bf K}_{{\bf X}} ^*\right) (F(r{\bf W})\otimes I_\cH)\left(I_\cK\otimes{\bf K}_{{\bf X}}\right),\qquad r\in [0,1), {\bf X}\in {\bf D}_{{\bf n},rad}^{\bf m},
$$
and taking $r\to 1$, we deduce that
$$
F({\bf X})=\left(I_\cK\otimes {\bf K}_{{\bf X}} ^*\right) (T\otimes I_\cH)\left(I_\cK\otimes{\bf K}_{{\bf X}}\right),
$$
which proves item (ii).

Conversely, assume that item (ii) holds. According to Theorem \ref{main} and Theorem \ref{main2}, the operator $T$ is a weighted multi-Toeplitz operator. Due to Theorem \ref{bounded}, the function $F$ defined by
$F({\bf X})=\widetilde{\bf B}_{{\bf X}}[T]$, ${\bf X}\in {\bf D}_{{\bf n},rad}^{\bf m}(\cH)$, is a bounded free
 $k$-pluriharmonic function and $\|T\|=\sup_{r\in [0,1)} \|F(r{\bf W})\|$.
Since $T\in {\boldsymbol\cG}$, we can find a sequence
$$
g_n\in \text{\rm span}\left\{ C\otimes {\bf W}_{\boldsymbol\alpha}{\bf W}_{\boldsymbol\beta}^* :\   C\in B(\cK),   (\boldsymbol \alpha; \boldsymbol \beta)\in \cJ\right\}
$$
such that $g_n\to T$ in norm as $n\to\infty$.
Let $\epsilon>0$ and choose $N$ such that $\|T-g_N\|<\epsilon$. Choose also $\delta\in (0,1)$ such that
$\|\widetilde{\bf B}_{r{\bf W}}[g_N]-g_N\|<\epsilon$ for every $r\in (\delta, 1)$.
Since
\begin{equation*}
\begin{split}
\|\widetilde{\bf B}_{r{\bf W}}[T] -T\|&\leq \|\widetilde{\bf B}_{r{\bf W}}[T-g_N]\|+\|\widetilde{\bf B}_{r{\bf W}}[g_N]-g_N\|+\|g_N-T\|\\
&\leq \|T-g_N\|+2\epsilon<3\epsilon,
\end{split}
\end{equation*}
for every $r\in (0,\delta)$, we deduce that $T=\lim_{r\to 1} \widetilde{\bf B}_{r{\bf W}}[T]$ in the norm topology.
Taking into account that $F(r{\bf W})=\widetilde{\bf B}_{r{\bf W}}[T]$, we conclude that $T=\lim_{r\to 1} F(r{\bf W})$ in the norm topology. Therefore, item (i) holds.
Since the implication (iii)$\implies (i)$ is clear, it remains to prove that (ii)$\implies$(iii). To this end, assume that item (ii) holds. According to Theorem \ref{bounded}, $F$ is a bounded free $k$-pluriharmonic function on the radial poly-hyperball.
Let ${\bf Y}\in {\bf D}_{{\bf n}}^{\bf m}(\cH)$ and note that,  as in the proof of the implication (ii)$\implies(i)$, one can show that
$G(Y):=\lim_{r\to 1}\widetilde{\bf B}_{r{\bf Y}}[T]$
exists in the operator norm and $\|G({\bf Y})\|\leq \|T\|$.
Note that $G$ is an extension of $F$. It remains to prove that
$G$ is continuous on ${\bf D}_{{\bf n}}^{\bf m}(\cH)$.
Due to the equivalence of (i) with (ii) and its proof, we have $T=\lim_{r\to 1} F(r{\bf W})$ in norm. Consequently, if $\epsilon>0$, we can find $t_0\in (0,1)$ such that $\|T-F(t_0{\bf W})\|<\epsilon$. Since $T-F(t_0{\bf W})\in \boldsymbol\cG$, we have
\begin{equation*}
\begin{split}
\|G({\bf Y})-F(t_0 {\bf Y})\|&= \|\lim_{r\to 1}\widetilde{\bf B}_{r{\bf Y}}[T]-F(t_0{\bf Y})\|\\
&\leq \limsup_{r\to 1}\| \widetilde{\bf B}_{r{\bf Y}}[T-F(t_0{\bf W})]\|\leq \|T-F(t_0{\bf W})\|<\epsilon
\end{split}
\end{equation*}
for every ${\bf Y}\in {\bf D}_{{\bf n}}^{\bf m}(\cH)$. Due to the continuity of $F$  on ${\bf D}_{{\bf n},rad}^{\bf m}(\cH)$, there is $\delta>0$ such that $\|F(t_0{\bf Y})-F(t_0{\bf Z})\|<\epsilon$ for every ${\bf Z}\in {\bf D}_{{\bf n}}^{\bf m}(\cH)$ with $\|{\bf Z}-{\bf Y}\|<\delta$.
Since
$$
\|G({\bf Y})-G({\bf Z})\|\leq \|G({\bf Y})-F(t_0{\bf Y})\|+\|F(t_0{\bf Y})-F(t_0{\bf Z})\|+\|F(t_0{\bf Z})-G({\bf Z})\|<\epsilon
$$
for every  for every ${\bf Z}\in {\bf D}_{{\bf n}}^{\bf m}(\cH)$ with $\|{\bf Z}-{\bf Y}\|<\delta$, the proof is complete.
\end{proof}

      \bigskip

       %


\begin{thebibliography}{99}

















 %




















\bibitem{BS} {\sc A.~Bottcher and B.~Silbermann}, {\it Analysis of Toeplitz operators}, Springer-Verlag, Berlin, 1990.

\bibitem{BH}{\sc A.~Brown  and  P.R.~Halmos},  Algebraic properties of Toeplitz operators, {\it J. Reine Angew. Math.} {\bf 213} 1963/1964 89--102.

















\bibitem{DP2} {\sc K.~R.~Davidson and D.~Pitts},
The algebraic structure of non-commutative analytic Toeplitz algebras,
{\it  Math. Ann.}
   {\bf 311} (1998),  275--303.


\bibitem{DKP}  {\sc K.R.~Davidson, E.~Katsoulis, and D.~Pitts},
  The structure of free semigroup algebras,
 {\it J. Reine Angew. Math.}
  {\bf 533} (2001), 99--125.



\bibitem{DLP} {\sc K.~R.~Davidson, J.~Li, and D.R.~Pitts},
  Absolutely
continuous representations and a Kaplansky density theorem for free
semigroup algebras, {\it J. Funct. Anal.} {\bf 224} (2005), no. 1,
160--191.

\bibitem{Dou} {\sc   R. G.~Douglas}, {\em Banach algebra techniques in operator theory}, Second edition. Graduate Texts in Mathematics, {\bf 179}, Springer-Verlag, New York, 1998. xvi+194 pp.
















\bibitem{ER} {\sc E.G.~Effros and Z.J.~Ruan},
  {\em Operator spaces},
 London Mathematical Society Monographs. New Series, {\bf 23}.
 The Clarendon Press, Oxford University Press, New York, 2000.





\bibitem{EL} {\sc J.~Eschmeier and S.~Langend\" orfer}, Toeplitz operators with pluriharmonic symbol, preprint 2017.





\bibitem{HKZ} {\sc H.~Hedenmalm, B.~ Korenblum, and K.~Zhu},  {\it Theory of Bergman spaces}, Graduate Texts in Mathematics, {\bf 199}, Springer-Verlag, New York, 2000. x+286 pp.










\bibitem{K} {\sc I.~Katznelson} {\em An introduction to harmonic analysis}, Cambridge University Press, Cambridge, 2004.

    \bibitem{Ken1} {\sc M.~ Kennedy}, Wandering vectors and the reflexivity of free semigroup algebras, {\it J. Reine Angew. Math.} {\bf 653} (2011), 47--73.

\bibitem{Ken2} {\sc M.~ Kennedy}, The structure of an isometric tuple, {\it  Proc. Lond. Math. Soc.} (3) {\bf 106} (2013), no. 5, 1157--1177.



     \bibitem{LO} {\sc I.~Louhichi and A.~Olofsson}, Characterizations of Bergman space Toeplitz operators with harmonic symbols, {\it J.Reine Angew. Math.} {\bf 617} (2008), 1--26.

     \bibitem{MSS} {\sc  A.~Maji, J.~ Sarkar and S.~ Sarkar},  Toeplitz and asymptotic Toeplitz operators on $H^2(\DD^n)$, {\it  Bull. Sci. Math.}  {\bf 146} (2018), 33--49.




















\bibitem{Po-multi} {\sc G.~Popescu},
 Multi-analytic operators and some  factorization theorems,
 {\it Indiana Univ. Math.~J.}
 {\bf 38} (1989),   693--710.










      \bibitem{Po-analytic} {\sc G.~Popescu},
      {Multi-analytic operators on Fock spaces,}
      {\it Math. Ann.} {\bf 303} (1995), 31--46.










      \bibitem{Po-poisson} {\sc G.~Popescu},
     {Poisson transforms on some $C^*$-algebras generated by isometries,}
       {\it J. Funct. Anal.} {\bf 161} (1999),  27--61.
















\bibitem{Po-entropy} {\sc G.~Popescu}, Entropy and multivariable interpolation, {\it Mem. Amer. Math. Soc.}  {\bf 184} (2006), no. 868, vi+83 pp.





\bibitem{Po-pluriharmonic} {\sc G.~Popescu},
{Noncommutative transforms and free pluriharmonic functions},
 {\it Adv. Math.} {\bf 220} (2009), 831-893.

 \bibitem{Po-domains-models} {\sc G.~Popescu}, Noncommutative Berezin transforms and multivariable operator model theory, {\it  J. Funct. Anal.}  {\bf 254}  (2008),  no. 4, 1003--1057.



\bibitem{Po-domains} {\sc  G.~Popescu},
Operator theory on noncommutative domains, {\it Mem. Amer. Math. Soc.}  {\bf 205}  (2010),  no. 964, vi+124 pp.

\bibitem{Po-Berezin2} {\sc  G.~Popescu}, Berezin transforms on noncommutative varieties in polydomains, {\it  J. Funct. Anal.}  {\bf 265}  (2013),  no. 10, 2500--2552.



\bibitem{Po-Berezin1} {\sc  G.~Popescu}, Berezin transforms on noncommutative polydomains, {\it Trans. Amer. Math. Soc.}  {\bf 368}  (2016),  no. 6, 4357--4416.





     \bibitem{Po-Toeplitz} {\sc G.~Popescu},   Multi-Toeplitz operators and free pluriharmonic functions,  {\it J. Math. Anal. Appl.}  {\bf 478} (2019), no. 1, 256--293.





\bibitem{RR} {\sc M.~Rosenblum and J.~Rovnyak}, {\it Hardy classes and operator theory}, Oxford University Press-New York,  1985.




%









\bibitem{T} {\sc O.~Toeplitz}, Zur Theorie der quadratischen und bilinearen Formen von unendlichvielen Ver\" anderlichen,  {\it Math. Ann.}  {\bf 70}  (1911),  no. 3, 351--376.

\bibitem{U} {\sc H.~Upmeier}, {\it Toeplitz operators and index theory in several complex variables}, Operator Theory: Advances and Applications  {\bf 81}, Birkhauser Verlag, Basel, 1996.
        %


       \end{thebibliography}
      \end{document}